\documentclass[11pt,reqno]{amsart}
\usepackage{amsmath, amssymb, amsthm}
\usepackage[linktocpage=true]{hyperref}
\hypersetup{colorlinks,linkcolor=blue,urlcolor=blue,citecolor=blue}
\usepackage[table]{xcolor}
\usepackage[all,cmtip]{xy}
\usepackage{stmaryrd} %\llbracket, \rrbracket
\usepackage{bbm} % \mathbbm
\usepackage{leftidx} %\ltrans, \leftidx
\usepackage{mathabx} %\widecheck
\usepackage[bbgreekl]{mathbbol} %\bbm for Greek letters
\usepackage{rotating} %\rotatebox
\usepackage{wasysym} %\lhd
\usepackage[mmddyy]{datetime}
\usepackage{qtree}
\usepackage{comment}
\usepackage[T1]{fontenc} 
\usepackage[utf8]{inputenc}
\usepackage{mathtools} %smallmatrix
\DeclareSymbolFontAlphabet{\mathbbm}{bbold}
\DeclareSymbolFontAlphabet{\mathbb}{AMSb}
\usepackage{arydshln}
\makeatletter
\renewcommand*\env@matrix[1][*\c@MaxMatrixCols c]{%
  \hskip -\arraycolsep
  \let\@ifnextchar\new@ifnextchar
  \array{#1}}
\makeatother
\usepackage{youngtab,young}
\newcommand{\iyng}[1]{{\tiny\Yvcentermath1 \yng(#1)}}
\usepackage{todonotes}
\usepackage{pgf, tikz, colortbl}
\usetikzlibrary{arrows, positioning, calc, chains, cd}
\tikzset{
	ch/.style={circle,draw,on chain,inner sep=2pt},
	chj/.style={ch,join},
	every path/.style={shorten >=4pt,shorten <=4pt}
	}

\numberwithin{equation}{subsection}
\usepackage[margin = 0.9in]{geometry}
\def\<{\langle}
\def\>{\rangle}

\newcommand{\cF}{\mathcal F}
\newcommand{\cH}{\mathcal{H}}

\newcommand{\cN}{\mathcal{N}}
\newcommand{\tcN}{\widetilde{\cN}}
\newcommand{\cO}{\mathcal{O}}

\newcommand{\CC}{\mathbb{C}}
\newcommand{\codim}{\mathop{\mrm{codim}}}
\newcommand{\crl}{\curvearrowleft}

\newcommand{\diag}{\mrm{diag}}

\newcommand{\End}{\mrm{End}}

\newcommand{\aff}{\mrm{aff}}

\newcommand{\Fstd}{F^{\textup{std}}}

\newcommand{\fa}{\mathfrak{a}}
\newcommand{\fB}{\mathfrak{B}}

\newcommand{\fb}{\mathfrak{b}}

\newcommand{\fgl}{\mathfrak{gl}}

\newcommand{\fn}{\mathfrak{n}}

\newcommand{\gl}{\mathfrak{gl}}

\newcommand{\GL}{\mrm{GL}}
\newcommand{\Hom}{\mrm{Hom}}
\newcommand{\HBMt}{H^{\mrm{BM}}_{\mrm{top}}}
\newcommand{\HBM}{H^{\mrm{BM}}}

\newcommand{\II}{\mathbb{I}}
\newcommand{\Ind}{{\mrm{Ind}}}
\newcommand{\Infl}{\mathop{\mrm{Infl}}}
\newcommand{\Irr}{\mathop{\mrm{Irr}}}

\newcommand{\id}{\mathbbm{1}}
\newcommand{\inv}{^{-1}}

\newcommand{\mrm}{\mathrm}

\newcommand{\PP}{\mathbb{P}}

\newcommand{\QQ}{\mathbb{Q}}

\newcommand{\red}[1]{{\color{black}#1}}
\newcommand{\redtwo}[1]{{\color{black}#1}}
\newcommand{\RR}{\mathbb{R}}
\newcommand{\Stab}{\textup{Stab}}

\newcommand{\tif}{\textup{if }}

\newcommand{\ZZ}{\mathbb{Z}}

\newcommand\twoheaduparrow{\mathrel{\rotatebox{90}{$\twoheadrightarrow$}}}

\usepackage[shortlabels]{enumitem}
\newenvironment{enu}{\begin{enumerate}}{\end{enumerate}}
\newenvironment{enua}{\begin{enumerate}[label=\textup{(\alph*)}]
}{\end{enumerate}}
\newenvironment{enuA}{\begin{enumerate}[label=\textup{(\Alph*)}]
}{\end{enumerate}}
\newenvironment{eq}{\begin{equation}}{\end{equation}}

\theoremstyle{definition}

\newtheorem{Def}{Definition}[subsection]

\newtheorem{exa}[Def]{Example}
\newtheorem{rmk}[Def]{Remark}

\theoremstyle{plain}
\newtheorem{prop}[Def]{Proposition}
\newtheorem{thm}[Def]{Theorem}
\newtheorem*{thmN}{Main Theorem}
\newtheorem{theorem}[Def]{Theorem}
\newtheorem{lemma}[Def]{Lemma}
\newtheorem{lem}[Def]{Lemma}
\newtheorem{cor}[Def]{Corollary}

\begin{document}

\title{On the Springer correspondence for wreath products}%Wreath Flags and Wreath Steinberg Varieties}

\begin{abstract} 
We establish a Bruhat decomposition indexed by the wreath product $\Sigma_m\wr \Sigma_d$ between two symmetric groups -- note that $\Sigma_m\wr \Sigma_d$ is not a Coxeter group in general. 
We show that such a decomposition affords a geometric variant in terms of the \red{Bia\l ynicki-Birula} decomposition for varieties with $\mathbb{C}^*$-actions.
Next, we construct a Steinberg variety whose top \red{Borel--Moore} homology realizes the group algebra $\mathbb{Q}[\Sigma_m\wr \Sigma_d]$ as a proper subalgebra.
Such a geometric realization leads to a Springer-type correspondence which identifies the irreducible representations of $\Sigma_m\wr \Sigma_d$ with isotypic components of certain unconventional Springer fibers using type A geometry.
In other words, we obtain a geometric counterpart of the (algebraic) Clifford theory, for the first time.
Consequently, we obtain a new Springer correspondence of Weyl groups of type B/C/D using essentially type A geometry.
\end{abstract}

\author{\sc You-Hung Hsu}
\address{Institute of Mathematics \\ Academia Sinica \\ Taipei 106319, Taiwan} 
\email{youhunghsu@gate.sinica.edu.tw}

\author{\sc Chun-Ju Lai}
\address{Institute of Mathematics \\ Academia Sinica \\ Taipei 106319, Taiwan} 
\email{cjlai@gate.sinica.edu.tw}
\thanks{Research of the authors were supported in part by
NSTC grants 111-2628-M-001-007, 112-2628-M-001-003, and the National Center of Theoretical Sciences}

%\subjclass[2020]{17B08, 14L35, 20E22}
\maketitle

%\tableofcontents

%=============================================
\section{Introduction}
%=============================================
\subsection{Wreath Products between Symmetric Groups}
%=============================================
Recall that the Hecke algebra of a Coxeter group is a deformation of its group algebra, with quantized quadratic relations. 
A quantum wreath product introduced in \cite{LNX24}, roughly speaking, 
 is a deformation of the group algebra of the wreath product $G\wr \Sigma_d$ of a (possibly infinite) group $G$ by a symmetric group $\Sigma_d$, with quantized wreath relations and quantized quadratic relations in the sense that coefficients are in certain (not necessarily commutative) tensor algebra.
A prototypical example, introduced by Jun Hu in \cite{Hu02}, is the Hecke subalgebra $\mathcal{A}(m) \subseteq \cH_q(\Sigma_{2m})$ (which we call the Hu algebra) appearing in a Morita equivalence theorem between the Hecke algebras of type A and of type $D_{2m}$. 
In other words, the Hu algebra is a unconventional deformation of the group algebra of $\Sigma_m \wr \Sigma_2$.

It is shown in \cite{LNX24} that $\mathcal{A}(m)$ and its generalization $\cH_q({m\wr d})$ shared many favorable properties with the Hecke algebras.
In particular,  a positivity pattern in terms of the dual \red{Kazhdan--Lusztig} basis is observed. 
It is thus tempting to claim that $\cH_q({m\wr d})$ should be regarded as the Hecke algebra for the wreath product $\Sigma_m \wr \Sigma_d$.
%In this paper, we establish the very first step towards a geometric representation theory for $\Sigma_m \wr \Sigma_d$.

While $\Sigma_2 \wr \Sigma_d$, the Weyl group of type B, affords a theory of Hecke algebras, as a special case of the \red{Ariki--Koike} algebras \cite{AK94} that quantize the wreath product $C_m \wr \Sigma_d$; our wreath product $\Sigma_m \wr \Sigma_d$ between symmetric groups are not complex reflection groups in general, and hence the most general theory of Hecke algebras due to \red{Brou\'e--Malle--Rouquier} \cite{BMR98} does not apply.
%=============================================
\subsection{The Lagrangian Construction}
%=============================================
On one hand, it is well-known that the \red{Ariki--Koike} algebras are cyclotomic quotients of the affine Hecke algebras of type A, 
which can be obtained via the equivariant K-theory on the Steinberg varieties (see \cite{KL87}).
On the other hand, a classical (i.e., $q \mapsto 1$) result is also available via a Lagrangian construction due to Ginzburg \cite{Gi86} and \red{Kashiwara--Tanisaki} \cite{KT84} by considering the \red{Borel--Moore} homology. 

In \cite{CG97}, it is established a uniform geometric approach to the representation theory of the Weyl groups, their corresponding Lie algebras, and various quantizations of these objects.
Such a Lagrangian construction appeared also in \cite{Li21} for certain symmetric pairs.
In this paper, we show that the wreath product $\Sigma_m \wr \Sigma_d$  is a new example that affords a Lagrangian construction of a similar flavor.

%=============================================
\subsection{Bruhat Decompositions for Wreath Products}
%=============================================
A typical ingredient of geometric representation theory is the Tits system, or, an algebraic group $G$ admitting a Bruhat decomposition $G = \bigsqcup_{w \in \Sigma_m \wr \Sigma_d} BwB$, indexed by $\Sigma_m \wr \Sigma_d$, for some subgroup $B \subseteq G$.

We construct a quadruple $(\GL_m \wr \Sigma_d, B_m^d, N_{m\wr d}, \Sigma_m \wr \Sigma_d)$, and then investigate how far it is away to form a Tits system. 
Our first result is that a Bruhat decomposition with respect to $\Sigma_m \wr \Sigma_d$ does exist, even when $\Sigma_m \wr \Sigma_d$ is not a Coxeter group. 
Our proof utilizes Iwahori's generalized Tits system \cite{I65}, which was originally developed to deal with the subtle difference between the extended and non-extended affine Weyl groups (and hence  extended versus non-extended  Hecke algebras).

%=============================================
\subsection{Wreath Steinberg Varieties} \label{1.4}
%=============================================
In a Tits system $(G, B, N,W)$, the quotient $G/B$ identifies with the flag variety. %, on which one can develop deep geometric representation theory.
For our new generalized Tits system $(\GL_m \wr \Sigma_d, B_m^d, N_{m\wr d}, \Sigma_m \wr \Sigma_d)$,
we introduce the {\em wreath flag varieties} $\cF l_{m \wr d}$, which carries an action of the wreath product $\GL_m \wr \Sigma_d$ %of the general linear group by a symmetric group, 
and can be identified with the quotient $(\GL_m \wr \Sigma_d)/B^d_m$.

%To see this, note that our group is $\GL_m \wr \Sigma_d$ and the Borel subgroup is $B^d_m$ with the Borel subalgebra $\fb^d_m$. Thus, under the above bijection maps, the conjugation by an element $\gamma =(g_i)_iw \in \GL_m \wr \Sigma_d$ is given by 
%\eq
%(g_i)_iw (\fb_m,\dots,\fb_m,id)((g_i)_iw)^{-1}=(g_1\fb_m,\dots,g_d\fb_m,w)((g_i)_iw)^{-1}=(g_1\fb_mg^{-1}_1,\dots, g_d\fb_mg^{-1}_d,id)
%\endeq which is still an element in $\fB^d_m$ and clearly not a bijection.

While one main feature for the flag variety for $\GL_n$ %with $G=\GL_m$ and $B=B_m$ 
is the natural bijection
$\GL_n/B_m \to \fB_m \to \cF l_m$, %  given by $g B_m \mapsto g \fb_m g\inv \mapsto g \Fstd_{\bullet}$. 
where $\fB_m$ is the variety of all Borel subalgebras in $\fgl_m$, and $\cF l_m$ is the variety of complete flags in $\CC^m$. 
We remark that such bijections only partially generalize in our setup for wreath flag variety. 
The reason is that the obvious analog $\fB_{m \wr d} := \fB_{m}^d$ is  not in bijection with the other varieties.  

Thus, we  have to avoid using arguments involving Borel subalgebras.
We then  construct a desirable Steinberg variety $Z_{m \wr d}$ via an unconventional Springer resolution.
In Table~\ref{table:1} below, we provide a comparison of the two setups.

\begin{table}[h!]
\centering
\begin{tabular}{|l||l|ll|} 
 \hline
  & Type A Construction   & Wreath Construction & \\ [0.5ex] 
 \hline\hline
(generalized) Weyl group  &$\Sigma_m$ & $\Sigma_m \wr \Sigma_d$ &
 \\
 Lie group & $\GL_m$ & $ {G}_{m\wr d} \subseteq \GL_{md}$ & Def.~\ref{def:Gmwrd}
 \\
&&$\cong \GL_m \wr \Sigma_d$& Prop.~\ref{prop:GLmwrd}
 \\
BN-pair &$B_m, ~ N_m \subseteq \GL_m$& $B_m^d, ~ N_m \wr \Sigma_d  \subseteq G_{m \wr d} $ &\eqref{asm:N}
\\
Flag variety & $\cF l_m$ & $\cF l_{m \wr d} \subseteq \cF l_{md}$ &Def.~\ref{wrflag} 
\\
&&$\cong \cF l_m \wr \Sigma_d $& Prop.~\ref{compatabilitywrfl}(a)
\\
\quad Homogeneous space & $\cong \GL_m / B_m$ & $\cong {G}_{m \wr d}/B_m^d$ & Prop.~\ref{compatabilitywrfl}(b)
\\
\quad Set of all Borels & $\cong \fB_m$ & irrelevant & non-isomorphic
\\
Nilpotent cone & $\cN_m \subseteq \fgl_m \crl \GL_m$ & $\cN_m^d \subseteq \fgl_m^d \crl G_{m \wr d}$ &\eqref{nilact}
\\
&$\twoheaduparrow$&$\twoheaduparrow$&
\\
Springer resolution  &$\widetilde{\cN}_{m}:=T^*\cF l_m$ & $\widetilde{\cN}_{m \wr d}:=T^*\cF l_{m \wr d} $&\eqref{wreath-springer-resolution}
%\\
%\quad Equiv. Vector bundle &$\cong \GL_m \times_{B_m} \fn_m$ &$\cong \GL_{m \wr d} \times_{B_m^d} \fn_m^d$& Cor.~\ref{nilbuniso}
\\
Steinberg variety & $Z_m :=\widetilde{\cN}_{m} \times_{\cN_m} \widetilde{\cN}_{m}$ & $Z_{m\wr d} :=\widetilde{\cN}_{m\wr d} \times_{\cN_m^d} \widetilde{\cN}_{m\wr d}$& Def.~\ref{Steinberg}
 \\ [0.5ex] 
 \hline
\end{tabular}
\caption{A comparison of objects in the two constructions}
\label{table:1}
\end{table}
\subsection{Main results}
With the above setup, we are able to prove the first few steps towards a geometric representation theory for wreath products, which are summarized below:
\begin{thmN}
\begin{enuA}
\item (Corollary~\ref{gtsbd}, Theorem~\ref{thm:GBD}) There is a Bruhat decomposition of $G$ with Bruhat cells indexed by $\Sigma_m \wr \Sigma_d$, for some algebraic groups $B, G$. 
Moreover, it affords a geometric variant $G/B = \bigsqcup_{w \in \Sigma_m \wr \Sigma_d} Bw B/B$ via the \red{Bia\l ynicki-Birula} decomposition \cite{BB73}.
\item (Theorem~\ref{mainthm}) The top \red{Borel--Moore} homology of $Z_{m \wr d}$ realizes the group algebra $\QQ[\Sigma_m \wr \Sigma_d]$ as a proper subalgebra.
\item (Theorem~\ref{irredrep}, Proposition~\ref{prop:C(x)id}, Corollary~\ref{identindexset}) There is a Springer correspondence between irreducibles over $\CC[\Sigma_m \wr \Sigma_d]$ and certain isotypic components of the top \red{Borel--Moore} homology of Springer fibers. 
Moreover, an identification with the Clifford theory is obtained:
\[
\widehat{\Sigma_m \wr \Sigma_d} 
= \{ \HBMt(\fB_x)_\psi ~|~ [x,\psi] \in I^S_{m \wr d}\}
= \{ \Ind_{\Sigma_m \wr \Sigma_{\gamma}}^{\Sigma_m \wr \Sigma_d}(\widetilde{S^{\gamma}} \otimes \Infl S^\bblambda) ~|~ \bblambda \in I^C_{m \wr d}\},
\]
where the bijection $I^S_{m \wr d} \to I^C_{m \wr d}$ between indices describing the same simple module is also given.
\end{enuA}
\end{thmN}
\begin{rmk}
\enua
\item 
For Weyl groups, the Springer correspondence can be obtained via
Ginzburg's construction \cite{Gi86} (see also equivalent constructions of Springer \cite{Sp76, Sp78},  and of Lusztig \cite{Lu81}),
where geometry of the symplectic and orthogonal groups are used for type B/C/D.
\item
By setting either $m=2$ or $d=2$, our Springer correspondence leads to a new Springer correspondence of type B/C/D, using essentially type A geometry (see Section~\ref{sec:EC}).

\item 
To our best knowledge, for disconnected groups, there is a  Springer correspondence appeared in Lusztig's work \cite[11.10(a)]{Lu04}.
Such a correspondence enumerates a disjoint union of irreducibles over certain Coxeter groups,
and it appears that such a set is not in bijection with the set of simples over $\Sigma_m \wr \Sigma_d$ that we consider.

\item
The statements of our main results are regarding the most symmetric possible family of (normalizers of) Levi factors in connected reductive groups.
While we are aware of that some of our results can be generalized to a more general setup with unequal block sizes, we do not pursue this direction since 
(i) the classification of simple modules is eventually reduced to the special case, with equal block size, covered in our main result, and 
(ii) the other examples we want to geometrize all arise from (quantum) wreath products and hence have equal block size.

\item 
From the algebraic side, in \cite{LNX} there is a Clifford-type theory which relates simple module over the quantum wreath product with simple modules over the base algebra $B$. 
We hope that the methods develop in this article can hint on a potential geometrization of quantum wreath products whose base algebra has a geometric origin.
\endenua
\end{rmk}
Furthermore, the proof of our Springer correspondence requires new ideas as in the following.
%=============================================
\subsection{Technicalities} \label{sec:1.5}
%=============================================
Here, we list some essential technicalities that prohibit us to apply the geometric representation theoretic techniques developed in \cite[\S3]{CG97}.
\begin{enumerate}
\item It is not obvious how one can construct a Steinberg variety $Z$ such that its top Borel--Moore homology realizes the group algebra of $\Sigma_m \wr \Sigma_d$. 
The obvious Springer resolution does not work. 
After several unsuccessful attempts, we settle down to an unconventional Springer resolution which leads to a Steinberg variety $Z_{m \wr d}$ that affords the main theorem.
\item The convolution on Borel--Moore homology is generally difficult to compute.
We do not have the most general multiplication formula since the transversality condition in \cite[Theorem 2.7.26]{CG97} does not hold for an arbitrary pair of elements in $\HBMt(Z_{m \wr d})$.
Lacking of such a formula makes it impossible to check whether $\HBMt(Z_{m \wr d})$ is semisimple.
In turn, we cannot apply the powerful tools of Chriss--Ginzburg since we are unable to verify \cite[Claim 3.5.6]{CG97}.

However,  we can still multiply certain pairs of elements in $\HBMt(Z_{m \wr d})$ (see Lemma~\ref{convolution}). Such a multiplication lemma makes it possible to locate a semisimple subalgebra $A_{m\wr d}$ of $\HBMt(Z_{m \wr d})$ with the right dimension. 
\item We will then need to prove a variant of \cite[Theorem 3.5.7]{CG97} that relates representation theory of a certain semisimple subalgebra $A$ of $\HBMt(Z)$ for a Steinberg variety $Z$.
We include all the additional conditions needed for such a variant in Theorem~\ref{thm:irreps}.

While the conditions (A1--2) therein are straightforward generalizations of \red{Chriss--Ginzburg's} conditions (see (C1--2) in Proposition~\ref{prop:GinzLagr}); the condition (A3) actually provides a hint on how one should construct the desired subalgebra $A$. 

\item Last but not least, in the final statement of such a classification theorem, the modules are isotypic components of the top Borel--Moore homology $\HBMt(\fB_x)$ of a Springer fiber, which are well-defined only when  $\HBMt(\fB_x)$ affords a $(\HBMt(Z), C(x))$-bimodule structure. 

However, it can been seen from Example~\ref{ex:A4} that it can be the case that the $G$-action does not commute with the $\HBMt(Z)$-action. Thus, we need to impose condition (A4) so that the final statement makes sense.

\item With Theorem~\ref{thm:irreps} established, the construction of the desired semisimple proper subalgebra $A$ is still non-trivial. 
It is a pleasant surprise for us that the algebra $A_{m \wr d}$ we constructed in Theorem~\ref{mainthm} does satisfy (A2--4). In turn, we obtain the Springer correspondence for wreath products $\Sigma_m \wr \Sigma_d$. 
\end{enumerate}
%Finally, we remark that the obvious convolution algebras obtained by the wreath flag varieties over finite fields realizes the ``Hecke algebra'' in \cite{I65} that is a twisted tensor product $\cH_q(\Sigma_m^d) \hat{\otimes} \CC[\Sigma_d]$, instead of the Hu algebra. 
%We plan to pursue the (equivariant) K-theory of our Steinberg variety $Z_{m \wr d}$ and its potential connection with the Hu algebras in a sequel. 

\red{\subsection{Conventions}
Throughout this article, unless we explict mention, we are working over the field of complex numbers $\CC$. 
}

\vskip .25cm
\noindent {\bf Acknowledgements.}
We acknowledge Cheng-Chiang Tsai and Matthew Douglass for helpful comments on an early version. 
Research of the authors was supported in part by
MSTC grants 113-2628-M-001-011, 112-2628-M-001-003, 111-2628-M-001-007, and the National Center of Theoretical Sciences.
%=============================================
\section{The Wreath Product $\Sigma_m\wr \Sigma_d$} 
%=============================================
%=============================================
\subsection{Wreath Products} 
%=============================================
Denote by $\Sigma_d$ the symmetric group on $d$ letters with simple transpositions $\red{s_j := (j~j+1)}$ for $1\leq j < d$.
For any set $X$, we define a set
\eq\label{eq:XwrG}
X\wr \Sigma_d :=  X^d \times \Sigma_d.
\endeq
For $w\in \Sigma_d$, $g_i \in X$,
by a slight abuse of notation, 
we use the shorthand notation below: 
\eq\label{eq:shorthand}
(g_1, \dots, g_d, 1_{\Sigma_d}) =: (g_i)_i,
\quad
(1_G, \dots, 1_G, w) =: w \in X\wr \Sigma_d,
\endeq
where the later shorthand only makes sense when $X=G$ is a group.
In this case, the wreath product carries a group structure $G\wr \Sigma_d = G^d \rtimes \Sigma_d$ from a semidirect product, whose multiplication is determined by 
\eq \label{wreathproductgroupstructure}
w(g_1, \dots, g_d) = (g_{w^{-1}(1)}, \dots,  g_{w^{-1}(d)})w \in G \wr \Sigma_d
\quad
\textup{for}
\quad
w\in\Sigma_d,
g_i \in G.
\endeq
%We sometimes omit the multiplication $.$ when it's convenient.
For $g \in G$, we write 
\eq\label{def:upj}
g^{(j)} = (1_G^{j-1}, g, 1_G^{d-j})  \in G^d.
\endeq
Thus, for each $g\in G$, those $g^{(j)}$ are conjugate to each other since $g^{(j+1)} = t_j g^{(j)} t_j \in G \wr \Sigma_d$.

%In other words, the wreath product $\Sigma_m \wr \Sigma_d$ admits a presentation with generators $s_1, \cdots, s_{m-1}, t_1, \cdots, t_{d-1}$ subject to the usual relations for $\Sigma_m \cong \langle s_i\rangle$, $\Sigma_d \cong \langle t_j\rangle$, together with the mixed relations $(g_1, \dots, g_d)t_j = t_j(g_{t_j(1)}, \dots g_{t_j(d)})$ for $g_i \in \Sigma_m$.
\exa\label{ex:Smd}
In particular, $\Sigma_m \wr \Sigma_d$ can be thought as a subgroup of $\Sigma_{md}$ generated by simple transpositions $\redtwo{s_i = (i~i+1)}$ for $1\leq i \leq m-1$ and the following elements $\redtwo{t_1, \dots, t_{d-1}}$ of length $m^2$:
\eq
%\begin{split}
t_{j+1}:\{1, \dots, md\} \to \{1, \dots, md\},
\quad
k \mapsto
\begin{cases}
k+m&\tif jm+1\leq k \leq (j+1)m;
\\
k-m &\tif (j+1)m+1\leq k \leq (j+2)m;
\\
k&\textup{otherwise}.
\end{cases} 
%&\equiv \begin{array}{|cccccc|}
%jm+1&\cdots&(j+1)m&(j+1)m+1&\cdots &(j+2)m\\
%(j+1)m+1&\cdots &(j+2)m& jm+1&\cdots&(j+1)m
%\end{array} 
%\\
%&
%= (s_{jm+m} s_{jm+m+1} \cdots s_{jm+2m-1}) \cdots (s_{jm+1} s_{jm+2} \cdots s_{jm+m}),
%\end{split}
\endeq 
\redtwo{To be precise, each $t_i \in \Sigma_{md}$ is identified with the element $(1_{\Sigma_m},...,1_{\Sigma_m}, (i~i+1)) \in \Sigma_m \wr \Sigma_d$.}
%\redtwo{where we use the notation $t_i=(1_{\Sigma_m},...,1_{\Sigma_m},t_i) \in \Sigma_m \wr \Sigma_d$ from (\ref{eq:shorthand}) for all $1 \leq i \leq d-1$.} 
Moreover, for $1 \leq i, j <d, s \in \Sigma_m$, \redtwo{the following equality holds in $\Sigma_m \wr \Sigma_d$:}
\eq\label{eq:stts}
t_i s^{(j)}t_i\inv = \begin{cases}
s^{(j+1)} &\tif i=j;
\\
s^{(j-1)} &\tif i=j-1;
\\
s^{(j)} &\textup{otherwise}.
\end{cases}
\endeq
In general, $\Sigma_m \wr \Sigma_d$ is not a Coxeter group except that $\Sigma_m \wr \Sigma_d$ is of type $G(m,1,d)$ when  $m \leq 2$.
\endexa

\red{
The following is another example for wreath product that will be used later. 

\exa \label{ex:C*Sd}
Let $X=\mathbb{G}_m=\CC^{\times}$ be  the one-dimensional complex multiplicative group. Then, we consider the wreath product $X \wr \Sigma_d=\CC^{\times} \wr \Sigma_d=(\CC^{\times})^{d} \rtimes \Sigma_d$, which has a group structure from (\ref{wreathproductgroupstructure}).

On the other hand, we denote $N_d$ to be the group of all monomial $d \times d$ matrices with complex entries. Then, it is clear to see we have the following group isomorphism
\begin{align*}
\CC^{\times} \wr \Sigma_d &\rightarrow N_{d} \\
(a_1, \dots, a_n, w) &\mapsto \sum_{i=1}^n a_iE_{i,w\inv(i)}  
\end{align*} where $E_{ij}$ are the matrices with one at the $i$th row and the $j$th column, and zero else.
\endexa
}
For a composition $\gamma = (\gamma_1, \dots, \gamma_r)$ of $d$, there is an associated Young subgroup $\Sigma_\gamma := \Sigma_{\gamma_1} \times \dots \times \Sigma_{\gamma_r} \subseteq \Sigma_d$.
There is a canonical identification $G\wr \Sigma_\gamma \equiv (G \wr \Sigma_{\gamma_1}) \times \dots \times (G \wr \Sigma_{\gamma_r})$.

%=============================================
\subsection{Tits Systems} 
%=============================================
Our first step is to show that there is a Bruhat decomposition for some group $G_{m\wr d}$ into cells indexed by $\Sigma_m \wr \Sigma_d$.
In order to achieve that, we first examine how far the group $\Sigma_m \wr \Sigma_d$ is from a Weyl group arising from a Tits system.

\Def
A Tits system consists of the following quadruple $(G,B,N,W)$ in which $G$ is generated by its subgroups $B$ and $N$, $W:= N/(B\cap N)$ is well-defined and is generated by involutive elements $s_i (i\in I)$. Moreover, for all $s = s_i$ and \red{$\dot{w} \in W$ with a lift $w \in N$}:
\eq\label{eq:Tits}
\red{s Bw \subseteq BswB \cup BwB,
\quad
sB \neq Bs.}
\endeq
\endDef

\exa \label{ex:Tits}
Let $K$ be a field,  $B_n := B_n(K)$ be the standard Borel subgroup of $\GL_n := \GL_n(K)$, and $N_n:= N_n(K)$ be the group of monomial matrices in $\GL_n$. 
\begin{enumerate}[(i)]
\item
The symmetric group $W = \Sigma_n$ is produced from the Tits system $(\GL_n, B_n, N_n, \Sigma_n)$.
\item 
Let $\QQ_p$ be the field of $p$-adic numbers, and $\ZZ_p$ be the ring of $p$-adic integers.
The (non-extended) affine Weyl group $W = \Sigma^\aff_n$ is produced from  the Tits system $(\GL_n(\QQ_p), B, N_n(\QQ_p), \Sigma^\aff_n)$ where $B = \{(a_{ij}) \in \GL_n(\ZZ_p) ~|~ a_{ij} \in p\ZZ_p ~\tif i>j\}$.
\end{enumerate}
\endexa

\red{
We would like to find a Tits-type system, denoted by $(G_{m \wr d},B_{m \wr d},N_{m \wr d},W_{m \wr d})$, so that $W_{m \wr d} = \Sigma_m \wr \Sigma_d$. From Example \ref{ex:Smd}, we know that $\Sigma_m \wr \Sigma_d$ is a subgroup of $\Sigma_{md}$. Moreover, we also know that $(\GL_{md},B_{md},N_{md},\Sigma_{md})$ is a Tits system from Example \ref{ex:Tits} (i). So, it is natural to expect that the other groups $G_{m \wr d}$, $B_{m \wr d}$, $N_{m \wr d}$ are subgroups of $GL_{md}$, $B_{md}$, $N_{md}$, respectively. Furthermore, one of the axioms for a Tits system is the group $G$ is generated by $B$ and $N$. Thus it suffices to determine the two groups $B_{m \wr d}$ and $N_{m \wr d}$. 

We start from $N_{m \wr d}$. From Example \ref{ex:C*Sd}, we know that there is an isomorphism $N_{md} \cong \CC^{\times} \wr \Sigma_{md}$. By replacing $\Sigma_{md}$ with its subgroup $\Sigma_m \wr \Sigma_d$, we obtain
\eq
\CC^\times \wr (\Sigma_m \wr \Sigma_d) =(\CC^\times \wr \Sigma_m) \wr \Sigma_d \cong N_m \wr \Sigma_d = \{ (A_1, \dots, A_{d}, w)  ~|~ w \in \Sigma_d\} 
\endeq where we use the isomorphism from Example \ref{ex:C*Sd} 
in the second equality. Thus, it is natural to assume that 
\eq \label{asm:N}
N_{m \wr d} \coloneqq N_m \wr \Sigma_d.
\endeq}

%$N_{m\wr d}$ is the group of monomial matrices corresponding to elements in $\Sigma_m \wr \Sigma_d$, regarded as a subgroup of $\Sigma_{md}$.

Next, having  $\Sigma_m \wr \Sigma_d \cong N_{m \wr d}/(N_{m \wr d} \cap B)$ in mind, we need to find a subgroup $B \subseteq \GL_{md}$ whose intersection with $N_{m \wr d}$ is equal to the subgroup of diagonal matrices in $\GL_{md}$. 
Hence, it makes sense to assume further 
\eq\label{asm:B}
B = B_m^d := \{\textup{diag}(b_1, \dots, b_d) \in  \GL_{md} ~|~ b_i \in B_m ~\textup{for all}~ i\} = B_m \wr \{1_{\Sigma_d}\},
\endeq
where $\textup{diag}(b_1, \dots, b_d)$ is the block matrix obtained by putting $b_i$'s in the diagonal.
Thus, we are in a position to define our group $G$.
\Def\label{def:Gmwrd}
Recall from \eqref{asm:N}--\eqref{asm:B} the subgroups $B_m^d, N_{m\wr d}$ of $\GL_{md}$. Set
\[
G_{m\wr d}:= \langle B_m^d, N_{m\wr d}\rangle \subseteq \GL_{md}.
\] 
\endDef

\red{
\rmk
Although the group $G_{m\wr d}$ is a bit abstract from the above definition, we will show that it is isomorphic to $\GL_{m} \wr \Sigma_d$ in Proposition \ref{prop:GLmwrd}.
\endrmk
}

Now, we look back to our group $W = \Sigma_m \wr \Sigma_d$. 
While its generators $s_1, \dots, s_{m-1}, t_1, \dots, t_{d-1}$ are indeed of order two and that  $s_i$'s satisfy \eqref{eq:Tits}, 
the remaining generators $t_j$'s do not satisfy \eqref{eq:Tits}. 
In particular, $t_j B  = B t_j$ for all $j$.

In the following section, we will see that one can still obtain a Bruhat decomposition for $G_{m\wr d}$ via Iwahori's generalized Tits system.
%=============================================
\subsection{Iwahori's  Generalized Tits system}
%=============================================
Iwahori showed that we can relax the conditions on a Tits system to obtain a large index set $W$ for the cells such that $W$ is a semidirect product of a Coxeter group $W_0$ and a group $\Omega$ consisting of certain elements $t$ such that $t B  = Bt$.
\Def(\cite{I65})
Iwahori's  {\em generalized Tits system} is a  quadruple $(G,B,N,W)$ in which $G$ is generated by its subgroups $B$ and $N$, $W:= N/(B\cap N)$ is well-defined. Moreover,
\eq
W \cong W_0 \rtimes \Omega,
\endeq
where $W_0$ is generated by involutive elements $s_i (i\in I)$,  
any element $t \in \Omega$ normalizes both $B$ and $\{s_i\}_{i\in I}$,
and $Bt \neq B$ unless $t = 1$.
Finally, \eqref{eq:Tits} holds for all $s = s_i$, $w\in W$.
\endDef
%
%\Def \label{defgt} (\cite{I})
%Let $G$ be a group and $B,\ N$ subgroups of $G$. The triple $(G,B,N)$ is called a \textit{generalized Tits system} if the following conditions are all satisfied
%\begin{enumerate}
%	\item $H=B \cap N$ is a normal subgroup of $N$.
%	\item The factor group $N/H$ is a semidirect prodcut of a subgroup $\Omega$ and a normal subgroup $W$: $N/H = W \rtimes \Omega$.
%	\item $W$ is generated by a set $S=\{w_{i}\}_{i \in I}$ of involutive elements (or simple reflections), where $I$ is an index set.
%	\begin{enumerate}
%		\item For any $\sigma \in W \rtimes \Omega$ and $w_{i} \in S$, 
%		\eq
%		\sigma B w_{i} \subseteq B \sigma w_{i} B \cup B \sigma B.
%		\endeq
%		\item  $w_i B w^{-1}_{i} \neq B$ for all $i \in I$.
%	\end{enumerate}
%	\item Any element $\rho$ in $\Omega$ normalizes $S$: $\rho S \rho^{-1}=S$.
%	\item $\rho B \rho^{-1}=B$ for all $\rho \in \Omega$; $B\rho \neq B$ for any $\rho \in \Omega-\{1\}$.
%	\item G is generated by $B$ and $N$.
%\end{enumerate}
%\endDef
The generalized Tits systems enjoy the following properties.
\begin{lemma} [\cite{I65}] \label{lemgt} 
Assume that $(G,B,N, W)$ is a generalized Tits system. Then: 
\begin{enua}
	\item There is a Bruhat decomposition $G= \bigsqcup_{w \in W} Bw B$.
	\item $G_0 := BW_0B$ is a normal subgroup of $G$. Moreover, $(G_0, B, N \cap G_0, W_0)$ is a Tits system. In particular, $W_0$ is a Coxeter group.
\end{enua}
\end{lemma}

Then we have the following result.

\prop \label{prop1} \label{gts}
The quadruple $(G_{m \wr d}, B_{m}^d, N_{m \wr d}, \Sigma_m \wr \Sigma_d)$  forms a generalized Tits system. 
\endprop

\proof
Let $T_m:=B_m \cap N_m$. We first check that $T:= B_{m}^d \cap N_{m \wr d} = T_m^d $ is a normal subgroup of $N_{m \wr d}$. 
Recall the shorthand we adapt in \eqref{eq:shorthand}\red{, and the assumption $N_{m \wr d}=N_m \wr \Sigma_d$ from (\ref{asm:N}).} Indeed, for any $(\tau_i)_i \in T$ and $\gamma = (g_i)_i w \in N_{m \wr d}$,
\eq
\gamma (\tau_i)_i  \gamma\inv 
=(g_i)_i w (\tau_i)_i w\inv (g_i\inv)_i 
=(g_i\tau_{w\inv(i)})_i w  w\inv (g_i\inv)_i 
= (g_i\tau_{w\inv(i)}g\inv_i)_i  \in T,
\endeq since each $g_i \in N_m$. 
Thus,  
$W= N_{m \wr d}/T 
\red{=} (N_m \wr \Sigma_d)/T_m^d  
= (N_m/T_m) \wr \Sigma_d 
\cong \red{\Sigma_m  \wr \Sigma_d}$.

Following the notations in Example~\ref{ex:Smd}, the semidirect product $W = W_0 \rtimes \Omega$ can be described via
\eq
W_0 = N_m^d/T \cong \red{\langle s^{(b)}_a |1 \leq a < m, 1 \leq b \leq d \rangle }
%\langle s_i \in \Sigma_m \wr \Sigma_d ~|~ i \in I\rangle,
\quad 
\Omega = \red{\{1_{\Sigma_m}\} \wr \Sigma_d = \langle t_k \in \Sigma_m \wr \Sigma_d ~|~ 1\leq k < d\rangle,}
\endeq
where \red{we recall the notation $s_a^{(b)}$ from \eqref{def:upj}.} 
It then follows from \eqref{eq:stts} that  $\Omega$ normalizes \red{$\{s^{(b)}_a |1 \leq a < m, 1 \leq b \leq d\}$}.

%I = [1, md] \setminus m\ZZ$, and $s_i = s_a^{(b)}$ (cf. \eqref{def:upj}) if $i = (b-1)m+a$.

In order to verify whether $t_k \in \Omega$ normalizes $B$, we write $t_k = \dot{t}_k T$ for a fixed representative $\dot{t}_k \in  N_{m \wr d}$. Then, any element in $B t_k$ is of the form $(b_i)_i \dot{t}_k T_m^d$ for some $(b_i)_i \in B_m^d$, and thus 
\eq
(b_i)_i \dot{t}_k T_m^d = \dot{t}_k (b_{t_k(i)})_i T_m^d \in \dot{t}_k BT = \dot{t}_k TB = t_k B. 
\endeq 
That is, $tB = Bt$ for all $t \in \Omega$, by symmetry.
Moreover, for $t \in\Omega$, $Bt \neq B$ unless $t=1$.

Finally, we verify that \eqref{eq:Tits} holds for all $s = \sigma^{(j)} ( \red{\sigma=s_i \in \Sigma_m})$ and $w = (w_i)_i t \in W (t\in \Omega)$:
Note that $sBw = \sigma^{(j)} B_m^d (w_i)_i  t = (\pi_i B_m w_i)_i  t$ where $\pi_i$ is identity except for that $\pi_j = \sigma$. 
Since $(\GL_m, B_m ,N_m, \Sigma_m)$ is a Tits system, \red{$\sigma B_m w_j \subseteq B_m \sigma w_j B_m \cup B_m w_j B_m$ for all $\sigma=s_i \in \Sigma_m$}, and hence
\eq
sBw \subseteq B \sigma^{(j)} (w_i)_i B t \cup B (w_i)_i B t =  B \sigma^{(j)} (w_i)_i t B  \cup B (w_i)_i t B = BswB \cup BwB.
\endeq
Also, $sB \neq Bs$ follows from that $\sigma B_m \neq B_m \sigma$.
%\begin{equation*}
%    \sigma B_{m}^ds = \{ \sigma (b_i)_i s =(w_ib_{w\inv(i)},...,w_lb_{w\inv(l)}s_i,...,w_db_{w\inv(d)},w)\}.
%\end{equation*} Since $w B_m s_i \subseteq B_mws_iB_m \cup B_mwB_m$ for all $w \in \Sigma_m$, we have $w_kb_{w\inv(k)} \in B_mw_kB_m$ for all $k \neq i$. For $w_lb_{w\inv(l)}s_i$, there are two cases, either $w_lb_{w\inv(l)}s_i \in B_mw_ls_iB_m$ or $w_lb_{w\inv(l)}s_i  \in B_mw_ls_iB_m$.
%
%Writing $w_kb_{w\inv(k)}=b'_{k}w_kb''_{k}$ for some $b'_k, b''_k \in B_m$ for all $k \neq l$, and  
%\begin{equation*}
%   w_lb_{w\inv(l)}s_i=\begin{cases}
%        b'_{l}w_lb''_l \ & \text{if} \ w_lb_{w\inv(l)}s_i \in B_mw_lB_m \\ 
%        b'_{l}w_ls_ib''_l & \ \text{if} \ w_lb_{w\inv(l)}s_i \in B_mw_ls_iB_m.
%    \end{cases}
%\end{equation*} Then we have 
%\begin{align*}
%&(w_1b_{w\inv(1)},...,w_lb_{w\inv(l)}s_i,...,w_db_{w\inv(d)},w) \\
%&=\begin{cases}
%    (b'_1,...,b'_d,id)(w_1,...,w_d,w)(b''_{w^{-1}(1)},...,b''_{w^{-1}(d)},id) \ \text{if} \ w_lb_{w\inv(l)}s_i  \in B_mw_iB_m \\ 
%    (b'_1,...,b'_d,id)(w_1,...,w_ls_i,...,w_d,w)(b''_{w^{-1}(1)},...,b''_{w^{-1}(d)},id) \ \text{if} \ w_lb_{w\inv(l)}s_i  \in B_mw_ls_iB_m
%\end{cases} 
%\end{align*} which is either in $B^d_m \sigma B_m^d$ or $B^d_m \sigma s_i^{(j)} B_m^d$
\endproof

\cor \label{gtsbd}
There is a Bruhat decomposition 
\begin{equation} \label{bruhat}
 G_{m \wr d} = \bigsqcup_{w \in \Sigma_{m} \wr \Sigma_{d}} B^d_{m} w B^d_{m}.   
\end{equation}
\endcor
\proof
It follows by combining Lemma~\eqref{lemgt}(a) and Proposition~\ref{prop1}.
\endproof

\rmk
The exact argument in the proof of Proposition~\ref{prop1} can be used to prove that if $(G, B, N, W)$ is a Tits system, then $(\langle B^d, N \wr \Sigma_d \rangle, B^d, N \wr \Sigma_d, W \wr \Sigma_d)$ is a generalized Tits system. 
However, we did not find it relevant to consider these systems other than the case when $W = \Sigma_m$.
\endrmk

\subsection{A New Bruhat Order}

%In this subsection, we study the order of the group $\Sigma_m \wr \Sigma_d$ that is naturally induced from the geometric property. Note that although we have the inclusion of subgroups $\Sigma_m \wr \Sigma_d \leq \Sigma_{md}$ and the Bruhat order on $\Sigma_{md}$, since $\Sigma_m \wr \Sigma_d$ is not a Coneter group, the order on $\Sigma_m \wr \Sigma_d$ is not necessary the induced Bruhat order.
%and also $(\Sigma_m \wr \Sigma_d) \times \Sigma_d$ (which is the index set for the basis of $H(Z_{m \wr d})$) We begin with $\Sigma_m \wr \Sigma_d$. 
One could have defined a Bruhat order on the subgroup $\Sigma_m \wr \Sigma_d \subseteq \Sigma_{md}$ via the type A Bruhat order $\leq_{md}$. 
However, such a naive definition disagrees with the closure relations via the Bruhat decomposition (Corollary \ref{gtsbd}).
For example, consider the generator $t:= s_2s_3s_1s_2 \in \Sigma_2 \wr \Sigma_2 \subseteq \Sigma_4$. 
While $s_1 \leq_{4} t$ with respect to the Bruhat order of $\Sigma_4$; we will see in Example~\ref{ex:newBruhat} that $t$ is not compatible with $s_1$ in terms of the closure relations.
In other words, elements in $\Sigma_d$ should be treated as zero lengths elements as in the extended affine Weyl groups.

\Def
The Bruhat order $\leq_{m\wr d}$ on $\Sigma_m \wr \Sigma_d$ is given by
\eq
x \leq_{m \wr d} y \quad \iff \quad C(x) \subseteq \overline{C(y)},
\endeq
where $C(w):= B^d_m w B^d_m$ the Bruhat cell for $w\in \Sigma_m \wr \Sigma_d$ in $G_{m \wr d}$.
\endDef
%------------------------------------------------------
\begin{lemma} \label{order}
The Bruhat order $\leq_{m \wr d}$ has the following combinatorial description: \red{for any $(w_i)_i\sigma, (w'_i)_i\sigma' \in \Sigma_m \wr \Sigma_d$, we have 
\eq\label{def:newBruhat}
(w'_i)_i \sigma' \leq_{m\wr d} (w_i)_i\sigma
\quad\iff\quad
w'_i \leq_{m} w_i \in\Sigma_m \ \textup{for all} \ 1\leq i\leq m, \ \textup{and} \ \sigma'=\sigma \in \Sigma_d, 
\endeq
where $\leq_{m}$ is the corresponding Bruhat order on $\Sigma_m$.}
\end{lemma}

\begin{proof}
Note that %$G_{m \wr d}$ is an algebraic subgroup of $\GL_{md}$ with the topology induced from the Zariski topology on $\GL_{md}$.
%Moreover, 
$G_{m \wr d}$ is disconnected, with connected components  $\{\GL_m^d \times \sigma~|~ \sigma \in \Sigma_d \}$. 
Since each connected component is closed, for $w=(w_i)_i\sigma \in \Sigma_m \wr \Sigma_d$, the closure of $C(w)$
is given by
\begin{equation}
\overline{C(w)}=\overline{B^d_m (w_i)_i\sigma B^d_m}= \prod_{i=1}^d \overline{B_mw_iB_m} \times \{\sigma\}
=  \prod_{i=1}^d  \bigcup_{w'_i \leq w_i}{B_mw'_iB_m} \times \{\sigma\},
\end{equation}
i.e., those $C(w')$ appearing in $\overline{C(w)}$ satisfy the right hand side of \eqref{def:newBruhat}.
%Let $x=(w'_i)_i\sigma'$. Then $C(x) \subset \overline{C(w)}$ becomes
%\begin{equation} \label{inclusion}
%C(x)=\prod_{i=1}^d B_mw'_iB_m \times \{\sigma'\} \subset \overline{C(w)}=\prod_{i=1}^d \overline{B_mw_iB_m} \times \{\sigma\}.
%\end{equation}
%
%For (\ref{inclusion}) to be holded, we must have $\sigma'=\sigma$ and $B_mw'_iB_m \subset \overline{B_mw_iB_m}$ for all $i$. Moreover, since $B_mw_iB_m$ is a usual Bruhat cell in $\GL_m$ for all $i$, $B_mw'_iB_m \subset \overline{B_mw_iB_m}$ implies that $w'_i \leq_m w_i$ for all $i$. 
\end{proof}

\exa\label{ex:newBruhat}
The Hasse diagram for $(\Sigma_2\wr \Sigma_2, \leq_{2\wr 2})$ is depicted as below:
\red{\begin{equation*}
\xymatrix{
&  s^{(1)}_1s^{(2)}_1 \ar[ld] \ar[rd] & & & s^{(1)}_1s^{(2)}_1t_1 \ar[ld] \ar[rd] \\
s^{(1)}_1 \ar[rd] & & s^{(2)}_1 \ar[ld] & s^{(1)}_1t_1 \ar[rd] & & s^{(2)}_1t_1 \ar[ld] \\
& e & & & t_1%:=s_2s_3s_1s_2
}
\end{equation*}}
We remark that this new Bruhat order for wreath products $\Sigma_2 \wr \Sigma_d$ does not coincide with the (finer) Bruhat order of the type B Weyl group $W(B_d)$.
Recall that as a Coxeter group, $W(B_d)$ is generated by $s^B_0, \dots, s^B_{d-1}$, in which $s^B_0$ corresponds to the type B node in the Dynkin diagram.
The canonical isomorphism $\Sigma_2 \wr \Sigma_d \to W(B_d)$ is given by $s_1 \mapsto s^B_0,\  t_i \mapsto s_i^B$ for $1\leq i < d$. Thus, the Hasse diagram of $W(B_2)$ is as below, in which we use a shorthand notation $s^B_{i_1\dots i_N} := s^B_{i_1} \dots s^B_{i_N}$:

\endexa 
\begin{equation*}
    \xymatrix{
       &  s_1s_3 \equiv s^B_{0101} \ar[rrr] \ar[rd] & & & s_1s_3t \equiv s^B_{010} \ar[ld] \ar[rd] \\
       s_1 \equiv s^B_{0} \ar[rd] & & s_3 \equiv s^B_{101} \ar[r] \ar@/^2.0pc/[rrr]& s_1t \equiv s^B_{01} \ar[rd] \ar@/^1.0pc/[lll]& & s_3t  \equiv s^B_{10}\ar[ld] \ar@/^2.5pc/[lllll]\\
       & e & & & t\equiv s^B_{1} \ar[lll]
    }
\end{equation*}

%=============================================
\subsection{Lie Group Structure}
Recall from Definition~\ref{def:Gmwrd} that $G_{m\wr d}$ is a subgroup of the Lie group $\GL_{md}$.
We now identify the (abstract) wreath product $\GL_{m} \wr \Sigma_d$ with the group $G_{m\wr d}$. 
The proposition below follows from a routine calculation:
\begin{prop}\label{prop:GLmwrd}
The assignment below gives a group isomorphism $\GL_m \wr \Sigma_d \to G_{m \wr d}$:
\eq \label{gpinj}
	(g_i)_iw \mapsto 
\diag(g_1, \dots, g_d) \Theta_w = \Theta_w \diag(g_{w(1)}, \dots, g_{w(d)}),
\endeq
where $\diag(g_1, \dots, g_d)$ is the corresponding block diagonal matrix  in $\GL_{md}$,
and $\Theta_w \in \GL_{md}$ is the permutation matrix corresponding to $(1^d, w) \in \Sigma_m \wr \Sigma_d \subseteq \Sigma_{md}$.
\end{prop}
For example, when $d=3$, the tuple $(A_1, A_2, A_3, \red{(1~2~3)})$ is sent to 
\eq
\begin{psmallmatrix} A_1&0&0 \\ 0 &A_2&0 \\ 0& 0& A_3 \end{psmallmatrix}
\begin{psmallmatrix} 0 & 0 & I_m \\ I_m&0&0 \\ 0&I_m&0 \end{psmallmatrix}
=
\begin{psmallmatrix} 0 & 0 & A_1 \\ A_2&0&0 \\ 0&A_3&0 \end{psmallmatrix}
=
\begin{psmallmatrix} 0 & 0 & I_m \\ I_m&0&0 \\ 0&I_m&0 \end{psmallmatrix}
\begin{psmallmatrix} A_2&0&0 \\ 0 &A_3&0 \\ 0& 0& A_1 \end{psmallmatrix}.
\endeq

\rmk\label{rmk:LieA}
The subgroup $G_{m \wr d} \subseteq \GL_{md}$ is a {\em matrix Lie group} in the sense that, 
for any sequence $\{A_i\}_{i \geq 1}$ in $G_{m \wr d}$ that converges to some matrix $A$, either $A \in G_{m \wr d}$ or $A$ is not invertible.
Thus, $G_{m \wr d}$ is a closed subgroup of $\GL_{md}$.
Applying the closed subgroup theorem, one gets that
\eq
\textup{Lie}(G_{m \wr d}) = \{ X \in \gl_{md} ~|~ e^{tX} \in G_{m \wr d} \textup{~for all~} t \in \RR \} = \gl_m^d.
\endeq
%We remark that the corresponding Lie algebra is not the Lie subalgebra $\gl_m \wr \Sigma_d \subseteq \fgl_{md}$.
\endrmk

\section{Flag Varieties}
\subsection{Complete Flag Varieties}
For now, let $\GL_n = \GL_n(\redtwo{\CC})$, $B_n = B_n(\redtwo{\CC})$ with corresponding Lie algebras $\fgl_n = \textup{Lie}(\GL_n)$ and $\fb_n = \textup{Lie}(B_n)$, respectively.
It is well-known (see \cite[\S3]{CG97}) that the following maps are bijections:
\eq\label{eq:FLequiv}
\GL_n/B_n \to \fB_n \to \cF l_n,
\quad
g B_n \mapsto g \fb_n g\inv \mapsto g \Fstd_\bullet,
\endeq
where $\fB_n :=\{ \textup{solvable Lie subalgebra } \fa \subseteq \fgl_n ~|~ \dim \fa = {\dim \fb_n}\} \hookrightarrow \PP(\bigwedge^{{\dim \fb_n}} \fgl_n)$ is the projective variety consisting of all Borel subalgebras of $\fgl_n$,
and $\cF l_n \hookrightarrow \prod_i \PP(\bigwedge^i \redtwo{\CC}^{n})$ is the (projective) complete flag variety, in which the standard flag in $\redtwo{\CC}^{n}$ is denoted by $\Fstd_\bullet$. 

\rmk
In our case, it turns out that one can only establish a bijection $G_{m \wr d}/B^d_m  \cong \cF l_{m \wr d}$ for some projective variety $\cF l_{m \wr d}$ to be constructed. Note that, as we point out in Subsection \ref{1.4}, the variety $\fB_{m \wr d}$ of all Borel subalgebras in the wreath setting is still $\fB^d_m$, which is  not in bijection with $G_{m \wr d} / B_m^d$. For example, when $m=d=2$,  
we can pick $\gamma = (g_1, g_2)s, \gamma' = (g_1, g_2) \in G_{2 \wr 2}$ where $g_i \in \GL_2,  \Sigma_2 = \{1,s\}$ such that, for any  $(\fb, \fb) \in \fB^2_2$:
\begin{equation}
	\gamma (\fb,\fb) \gamma^{-1}=(g_1\fb g^{-1}_1, g_2\fb g^{-1}_2) = \gamma'(\fb, \fb)(\gamma')^{-1}.
\end{equation}
\endrmk

\subsection{Wreath Flags}
%
%With the group $G_{m \wr d}$ and its Borel subgroup $B^d_{m}$, we have its associated flag variety $G_{m \wr d}/B^d_{m}$. By Proposition \ref{stgp}, we have 
%\begin{equation*}
%G_{m \wr d}/B^d_{m}=(\textbf{Im}\GL^d_m \rtimes_{\phi} \textbf{Im}\Sigma_d)/ \textbf{Im}B^d_m = \bigsqcup_{(g_1,...,g_d,w) \in G_{m \wr d}} (g_1,...,g_d,w)B^d_m
%\end{equation*} where the elements in $\textbf{Im}B^d_m$ are $(b_1,...,b_d,id)$ with $b_i \in B_m$.
%
%For each $(g_1,...,g_d,w) \in G_{m \wr d}$, we have the left coset 
%\begin{equation*}
%    (g_1,...,g_d,w)B^d_m=\{(g_1,...,g_d,w) \cdot (b_1,...,b_d,id) = (g_1b_{w(1)},...,g_db_{w(d)},w)\ |\ g_i \in \GL_{m}, \ b_i \in B_m \}.
%\end{equation*}
%
%
%The left coset description is not sufficient enough to see the geometry of $G_{m \wr d}/B^d_{m}$. We would like to have a flag description of the left coset $(g_1,...,g_d,w)B^d_m$. From the classical theory, we know that there is an identification of the quotient $\GL_{m}/B_{m}$ with the space of complete flags
%\begin{align*}
%    &\GL_{m}/B_{m} \rightarrow Fl(\CC^m) \\
%    & gB_m \mapsto gF^{std}_{\bullet}.
%\end{align*} In order to construct such a identification for our group $G_{m \wr d}$, we also need to define the action of $(g_1,...,g_d,w)$ on flags.

We first construct a subvariety $\cF l_{m \wr d} \subseteq \cF l_{md}$ which affords a $G_{m\wr d}$-equivariant bijection with the homogeneous space $G_{m \wr d}/B^d_m$.
We will show that $\cF l_{m \wr d}$ can be identified with the wreath product $\cF l_m \wr \Sigma_d$.
Such an identification allows us to verify technical conditions in our proof of Springer correspondence for wreath products.

Fix a basis $\{e_1, \dots, e_{md}\}$ of $\redtwo{\CC}^{md}$, and let $V_i=\textup{Span}_\redtwo{\CC}\{e_{(i-1)m+1},...,e_{im}\}$ for $1\leq i \leq d$.

\Def \label{wrflag} 
Denote the set which we call the {\em wreath flag variety} by 
\eq
\cF l_{m\wr d} := \{ F_\bullet \in \cF l_{md} ~|~ F_{im} = V_{w(1)} \oplus \dots \oplus V_{w(i)} \textup{ for all }1\leq i \leq d, \textup{ for some }w\in \Sigma_d\} .
\endeq 
%In other words, 
%\eq
%	\cF l_{m \wr d} = \{ (F^1_{\bullet},...,F^d_{\bullet}, w) \ | \ F^i_{\bullet} \in \cF l_m, \ w \in \Sigma_d \}.
%\endeq
\endDef
For any complete flag $F_\bullet$ in $V_1$ (i.e.,  $F_\bullet \in \cF l_m$), denote by ${}^iF_\bullet$ the corresponding complete flag in $V_i$ obtained by shifting the indices of basis elements by $(i-1)m$.
In particular, the corresponding standard flag in each $V_i$ is
\eq
{}^i\Fstd_\bullet = (
 0 \subseteq \redtwo{\CC} e_{(i-1)m+1} \subseteq \redtwo{\CC} e_{(i-1)m+1} \oplus \redtwo{\CC} e_{(i-1)m+2} \subseteq \dots \subseteq   V_i).
\endeq 
\begin{prop} \label{compatabilitywrfl}
%Then: 
\begin{enua}
\item $\cF l _m \wr \Sigma_d \equiv \cF l_{m \wr d}$ as sets via $(F^1_\bullet, \dots, F^d_\bullet, w) \mapsto \cF_\bullet = \cF_\bullet(F^1_\bullet, \dots, F^d_\bullet, w)$, where 
\eq
\cF_{im+j} = V_{w(1)} \oplus \dots \oplus V_{w(i)} \oplus ({}^{w(i+1)}F^{w(i+1)}_j),
\quad
(0\leq i < d, \red{1\leq j \leq m}). 
\endeq
\item  There is a natural identification
%\begin{equation} \label{identification} 
$G_{m \wr d}/B^d_m \to \cF l_{m \wr d}$,  $(g_i)_iwB^d_m \mapsto (g_1\Fstd_{\bullet},...,g_d\Fstd_{\bullet},w)$ \red{where we use the identification from part (a) in the image.}
%\end{equation} 
\item Under the identifications in Part (a) and in Proposition~\ref{gpinj}, the $G_{m\wr d}$-action on $\cF l_{m \wr d}$ given by
\eq \label{GactionFL}
(g_i)_iw  (F^1_{\bullet},...,F^d_{\bullet}, \sigma) = (g_1F^{w^{-1}(1)}_{\bullet},...,g_dF^{w^{-1}(d)}_{\bullet},w\sigma)
\endeq
is compatible with the $\GL_{md}$-action on $\cF l_{md}$.
\end{enua}
\end{prop}
\proof
Part (a) follows from a direct verification. Part (b) follows from Part (a). 
For Part (c), it suffices to show that  
\begin{align}
\Theta_w \cF_{\bullet}(F^1_\bullet, \dots, F^d_\bullet, \sigma) &= \cF_{\bullet}(F^{w^{-1}(1)}_\bullet, \dots, F^{w^{-1}(d)}_\bullet, w\sigma) \label{theta}, \\
\diag (g_1, \dots, g_d)  \cF_{\bullet}(F_{\bullet}^1,...,F_{\bullet}^d, \sigma) &=\cF_{\bullet}( g_{1} F_{\bullet}^1,..., g_{d} F_{\bullet}^d, \sigma). \label{diag}
\end{align}
Write $(F^1_{\bullet},...,F^d_{\bullet},\sigma)=(h_i)_i\sigma(\Fstd_{\bullet},...,\Fstd_{\bullet},id)$ for some $(h_i)_i\sigma \in G_{m \wr d}$. Then, a standard calculation shows that 
	$\diag(h_1,\dots,h_d)\Theta_\sigma \cF_{\bullet}(\Fstd_{\bullet},...,\Fstd_{\bullet},id)=\cF_{\bullet}(F^1_{\bullet},...,F^d_{\bullet},\sigma)$. 
Thus, \eqref{theta} is given by
\eq
\begin{split}
	\Theta_w \cF_{\bullet}(F^1_\bullet, \dots, F^d_\bullet, \sigma) &=\Theta_w\diag(h_1,\dots,h_d)\Theta_\sigma \cF_{\bullet}(\Fstd_{\bullet},...,\Fstd_{\bullet},id) \\
	&=\diag(h_{w^{-1}(1)},\dots,h_{w^{-1}(d)})\Theta_w\Theta_\sigma \cF_{\bullet}(\Fstd_{\bullet},...,\Fstd_{\bullet},id) \\
	&=\diag(h_{w^{-1}(1)},\dots,h_{w^{-1}(d)}) \Theta_{w\sigma} \cF_{\bullet}(\Fstd_{\bullet},...,\Fstd_{\bullet},id) \\
	&=\cF_{\bullet}(F^{w^{-1}(1)}_\bullet, \dots, F^{w^{-1}(d)}_\bullet, w\sigma).
\end{split}
\endeq 
\eqref{diag} follows from a similar verification.
\endproof

Using Proposition~\ref{compatabilitywrfl}(a), we define a map $\varphi^{\tau}_i$, for each $1\leq i \leq d$ and $\tau \in \Sigma_d$:
\eq\label{def:phii}
\varphi^{\tau}_i: \cF l_m \to \cF l_{m \wr d},
\quad
F_\bullet \mapsto \cF_\bullet(F^1_\bullet, \dots, F^d_\bullet, \tau)
\quad 
\textup{where}
\quad
F^j_\bullet = 
\begin{cases}
F_\bullet &\tif j=i;
\\
\Fstd_\bullet &\textup{otherwise.}
\end{cases}
\endeq

\subsection{Geometric Bruhat decomposition}
%Note that an (algebraic) proof of the following correspondence
%\begin{equation} \label{corr} 
%	\Sigma_{m} \wr \Sigma_{d} \longleftrightarrow  B^d_{m}\text{-orbits of } \cF l_{m \wr d} \equiv  G_{m \wr d}/B^d_{m}. 
%\end{equation} comes from the fact that $(G_{m \wr d}, B_{m}^d, N_{m \wr d}, \Sigma_m \wr \Sigma_d)$  is a generalized Tits system (Proposition \ref{gts}).
Next, we give a geometric proof of the Bruhat decomposition of the wreath flag variety $\cF l_{m \wr d}$ based on the \red{Bia\l ynicki-Birula} decomposition.

%Before we state the result, we recall some terminologies that are required for the setup. 
Let $X$ be a smooth complex projective variety with an algebraic $\CC^*$-action.
Assume that the set $X^{\CC^*}$ of $\CC^*$-fixed points of $X$ is finite. 
For $w\in X^{\CC^*}$, denote by $X_w = \{ x \in X \ | \ \lim_{z \rightarrow 0} z \cdot x =w \}$ the attracting set.
Since $\CC^*$ fixes $w$, there is a natural $\CC^*$-action on the tangent space $T_{w}X$.
Set $T^+_{w}X:=\bigoplus_{n \in \ZZ_{>0}} T_{w}X(n)$, where  $T_{w}X(n) \coloneqq \{x \in T_{w}X \ | \ z \cdot x = z^n x \ \forall z \in \CC^* \}$. 

%and gives weight spaces decomposition
%\begin{equation*}
%	T_{w}X=T^+_{w}X \oplus T^-_{w}X, \ T^+_{w}X=\bigoplus_{n > 0} T_{w}X(n), \ T^-{w}X=\bigoplus_{n < 0} T_{w}X(n)
%\end{equation*} where $T_{w}X(n) \coloneqq \{x \in T_{w}X \ | \ z \cdot x = z^n x \ \forall z \in \CC^* \}$. 

\begin{prop}[\red{Bia\l ynicki-Birula} decomposition \red{\cite{BB73}}] \label{bbdecomp}
Let $W = X^{\CC^*}$ be the (finite) set of $\CC^*$-fixed points of a smooth complex projective variety $X$. Then,
\begin{enua}
	\item The attracting sets form a decomposition $X = \bigsqcup_{w \in W}X_w$ into smooth locally closed subvarieties;
	\item Each attracting set $X_w$ is isomorphic to  $T_{w}X_w = T^+_{w}X$ as algebraic varieties. The isomorphism commutes with the $\CC^*$-action.
\end{enua}
\end{prop}

The first step is to choose a suitable $\CC^*$-action on our wreath flag variety $X=\cF l_{m \wr d} \equiv G_{m \wr d}/B^d_m$. 
Fix a maximal torus $T^d_m \subseteq B^d_m \subseteq \GL_m^d$, and let it act on $X$  by 
\eq \label{Tmdaction}
(t_1,...,t_d) \cdot (g_1,...,g_d,w)B^d_m= (t_1g_1,...,t_dg_d,w)B^d_m.
\endeq 
The next step is to analyze the set  of $T^d_m$-fixed points. 

\begin{lemma} \label{Tfixedpoint}
%For the $T^d_m$-action defined in (\ref{Tmdaction}), there is a bijection 
There is a bijection $\cF l_{m\wr d}^{T_m^d} \to \Sigma_m \wr \Sigma_d$.
%\eq \label{eq:fixedptTC}
%\{T^d_m\text{-fixed points of} \ G_{m \wr d}/B^d_m \equiv 
%\cF l_{m \wr d}\}  
%\longleftrightarrow \Sigma_m \wr \Sigma_d.
%\endeq
\end{lemma}

\begin{proof}
Let $(g_i)_iwB^d_m$ be a $T^d_m$-fixed point in $\cF l_{m\wr d} \equiv G_{m \wr d}/B^d_m$.
Then, for any $(t_i)_i \in T^d_m$, 
$(t_i)_i  (g_i)_iwB^d_m = (g_i)_iwB^d_m$, and hence
\begin{equation} \label{tfixed} 
B^d_m = ((g_i)_iw)^{-1}  (t_i)_i  (g_i)_iwB^d_m 
= (g^{-1}_{w(i)})_iw^{-1}  (t_ig_i)_i w B^d_m =(g^{-1}_{w(i)}t_{w(i)}g_{w(i)})_i B^d_m .
\end{equation}   
Thus,  for each $1 \leq i \leq d$, we have $g^{-1}_it_ig_iB_m=B_m$, or equivalently,  
	$T_m \subseteq g_iB_mg^{-1}_i$.	
Therefore, there is a bijection
\eq
\cF l_{m\wr d}^{T_m^d} \to \{\textup{Borel subgroups of }\GL_m^d \textup{ containing }T_m^d \}\times \Sigma_d,
\quad
(g_i)_iwB^d_m \mapsto \textstyle{\left(\prod_{i=1}^mg_iB_mg^{-1}_i, w \right)}.
\endeq

Since $g_iB_mg_i^{-1} \subseteq \GL_m$ is a Borel subgroup which contains $T_m$, the set of $T^d_m$-fixed points of $\cF l_{m\wr d}$ is the same as the set of Borel subgroups of $\GL_m^d$ containing $T^d_m$ times the extra copy $\Sigma_d$.
The lemma concludes from the fact that there is a bijection between the set of all Borel subgroups of $\GL_m$ containing $T_m$ and the symmetric group $\Sigma_m$. 

\end{proof}

%Note that the $T^d_m$-action is compatible with the identification (\ref{identification}). 
%From the description of $\cF l_{m \wr d}$, it is a disjoint union of $d!$-copies of the product of complete flag variety $Fl(\CC^m)^d$, i.e.,
%\begin{equation*}
%G_{m \wr d}/B^d_m \cong \cF l_{m \wr d} = \bigsqcup_{w \in \Sigma_d} Fl(\CC^m)^d(w) 
%\end{equation*} where 
%\begin{align*}
%Fl(\CC^m)^d(w)=\{ (F^1,...,F^d,w) \ | \ F^i \in Fl(\CC^m) \} &\cong (\GL_m/B_m)^d \times \{w\} &\cong (\GL^d_m \times \{w\})/B^d_m \\
%(g_1F^s,...,g_dF^s,w) &\mapsto (g_1,...,g_d)B^d_m \times \{w\} &\mapsto (g_1,...,g_d,w)B^d_m
%\end{align*} where we use $(e,...,e,w)B^d_m=B^d_m(e,...,e,w)$. Note that $\GL_m^d \times \{w\}$ itself is not a group. Note that each element in $Fl(\CC^m)^d(w)$ can be written as $(g_1F^s,...,g_dF^s,w)$ for some $g_i \in \GL_m$. We have , and its action on $Fl(\CC^m)^d(w)$ (and thus on $\cF l_{m \wr d}$) is given by
%\begin{equation*}
%(t_1,...,t_d) \cdot (g_1F^s,...,g_dF^s,w) = (t_1g_1F^s,...,t_dg_dF^s,w).
%\end{equation*} 
%Then we obtain that the $T^d_m$-fixed points of $Fl(\CC^m)^d(w)$ are of the form 
%\begin{align*}
%Fl(\CC^m)^d(w) &\cong (\GL_m^d \times \{w\})/B^d_m \\ 
%(w_1F^s,...,w_dF^s,w)=(w_1,...,w_d,id)(F^s,...,F^s,w) & \mapsto (w_1,...,w_d,id)(B^d_m \times \{w\})/B^d_m
%\end{align*} where $w_i \in \Sigma_m$ for all $i$, which are bijective to $\Sigma_m^d \times \{w\}$.
%Now, for the $T_m^d$-action on $\cF l_{m \wr d}$, we would also like to have a similar result, i.e.,

\begin{thm}[Geometric Bruhat decomposition]\label{thm:GBD}
Let $X = \cF l_{m \wr d} \equiv G_{m \wr d}/B^d_m$ be the wreath flag variety. Then,
\begin{enua}
\item
There is a $\CC^*$-action on  $X$ such that the fixed-point set $X^{\CC^*}$  is in bijection with $\Sigma_m \wr \Sigma_d$. 
\item $G_{m \wr d}/B^d_m = \bigsqcup_{\sigma \in \Sigma_m \wr \Sigma_d} X_\sigma$, where $X_\sigma$ is precisely the cell $B_m^d \sigma B^d_m/B_m^d$. 
\end{enua}
\end{thm}

\proof
For (a), \redtwo{we fix distinct integers $n_1,\dots,n_{md}$. Then, we choose a one-parameter subgroup $\CC^* \simeq \{\diag(a^{n_1}, \dots, a^{n_{md}}) \in T_{md}~|~  a \in \CC^* \}$} as in \cite[Lemma 3.1.10]{CG97} so that $\cF l^{\CC^*}_{md} = \cF l^{T_{md}}_{md}$.
Taking intersections with $\cF l _{m\wr d}$, we have $\cF l^{\CC^*}_{m \wr d} = \cF l_{m \wr d} \cap \cF l^{\CC^*}_{md} = \cF l_{m \wr d} \cap \cF l^{T_{md}}_{md} = \cF l^{T_{md}}_{m \wr d}$. 
(a) then follows from Lemma \ref{Tfixedpoint}, and thus the assumptions for the \red{Bia\l ynicki-Birula} decomposition (Proposition~ \ref{bbdecomp}) hold. 

Combining (a) and Proposition \ref{bbdecomp}, we obtain that $X$ decomposes into locally closed subvarieties of the form $X_\sigma \cong T^{+}_{\sigma} \cF l_{m \wr d}$ ($\sigma \in \Sigma_m \wr \Sigma_d$). 
%Note that $\cF l_{m \wr d}$ is a homogeneous $G_{m \wr d}$-space and is also a disjoint union of $d!$-copies of the product of full flag varitey $\cF l^d_m$. 
Write $\sigma = (g_i)_i w$ for some $g_i \in \Sigma_m, w \in \Sigma_d$. Then, 
\eq
	T_{\sigma}X \cong T_{(g_1,...,g_d)}\GL^d_{m}/B^d_m \times \{w\}={\textstyle\prod_{i=1}^{d}} T_{g_i}(\GL_{m}/B_m) \times \{w\}.
\endeq 
By the proof of \cite[Corollary 3.1.12]{CG97}, $T^+_{g_i}(\GL_{m}/B_m) \cong B_m g_iB_m/B_m$ for all $i$. 
Thus, 
\eq
T^+_{\sigma}\cF l_{m \wr d} 
	\cong {\textstyle\prod_{i=1}^{d}} T^+_{g_i}(\GL_{m}/B_m) \times \{w\}
	\cong  {\textstyle\prod_{i=1}^{d}} (B_m g_iB_m/B_m) \times \{w\} 
	\cong B^d_m \sigma B^d_m/B^d_m.
\endeq
\endproof
%Extend the action to $\cF l_{m \wr d}$, we have shown that every $B^d_m$-orbit of $\cF l_{m \wr d}$ contains exactly one $T^d_m$-fixed point, which is of the form $(w_1,...,w_d,w)B^d_m/B^d_m$ for some $(w_1,...,w_d,w) \in \Sigma_m \wr \Sigma_d$. Hence, we prove the correspondence (\ref{corr}) and obtain the following Bruhat decomposition
%\begin{equation*}
%\cF l_{m \wr d}= \bigsqcup_{w \in \Sigma_d} Fl(\CC^m)^d(w) = \bigsqcup_{w \in \Sigma_d} \bigsqcup_{(w_1,...,w_d) \in \Sigma^d_m} B^d_m \cdot (w_1,...,w_d,w)B^d_m /B^d_m.
%\end{equation*}

%=============================================
\section{Wreath Steinberg varieties}
%=============================================
In this section, we introduce the wreath Steinberg variety $Z_{m \wr d}$, in order to prepare a realization of the group algebra of $\Sigma_m \wr \Sigma_d$ via the top Borel-Moore homology of $Z_{m \wr d}$.
Thus, we analyze the irreducible components of $Z_{m \wr d}$ in terms of the conormal bundles.
We remark that the argument presented here is not a consequence of \cite[Proposition 3.3.4]{CG97} since the obvious wreath variants of Borel subalgebras are not in bijection with the wreath flags.

\subsection{Nilpotent Orbits}
As observed in Remark \ref{rmk:LieA}, we consider nilpotent cones for the Lie algebra $\gl_{m}^d $, instead of a wreath version. 
%
%Since the conjugation of $G_{m \wr d}$ on itself preserves the subgroup $\GL_m^d$, it induces the adjoint action of $G_{m \wr d}$ on the Lie algebra $\g=\gl^d_m$ which is also by conjugation. Note that we will denote the elements in $\gl^d_m$ to be $(x_1,..,x_d,id)$ when we compute the action of $G_{m \wr d}$ on $\gl^d_m$ by conjugation, i.e.
%\begin{align*}
%	Ad_{(g_1,...,g_d,w)}(x_1,...,x_d,id)&=(g_1,...,g_d,w)(x_1,...,x_d,id)(g_1,...,g_d,w)^{-1} \\
%	&=(g_1x_{w^{-1}(1)}g^{-1}_1,...,g_dx_{w^{-1}(d)}g^{-1}_d,id).
%\end{align*} Finally, taking the differential of the adjoint action at $(g_1,...,g_d,w)=(e,...,e,id)$, we obtain the adjoint action of $\gl^d_m$ on itself which is also its Lie algebra structure 
%\begin{equation*}
%	[(x_1,...,x_d),(y_1,...,y_d)]=([x_1,y_1],...,[x_d,y_d])
%\end{equation*} for all $(x_1,...,x_d), (y_1,...,y_d) \in \gl^d_m$. Let $\fb_0 = \mrm{Lie}(B^d_{m})$ be the Lie algebra of the standard Borel subgroup. Then we have $\fb_0=\fb^d_{m}$ where $\fb_m=\text{Lie}(B_{m})$.
Let $\cN_{m}$ be the variety of all nilpotent elements in $\gl_{m}$.
The $\GL_m$-orbits in $\cN_m$ are parametrized by the set  $\Pi_m$ of partitions $\nu$ of $m\in \ZZ_{\geq 0}$ (write $\nu \vdash m$ for short).

Let $G_{m \wr d}$ act on $\cN^{d}_{m}$ by, for $\gamma = (g_i)_i w \in G_{m \wr d}$ and $(\redtwo{x}_i)_i \in \cN_m^d$:
\eq\label{nilact}
\gamma \cdot (\redtwo{x}_i)_i := \gamma (\redtwo{x}_i)_i \gamma\inv  =(g_i\redtwo{x}_{w^{-1}(i)}g^{-1}_i)_i.
\endeq
Note that \eqref{nilact} is compatible with the $\GL_{md}$-action on $\cN_{md}$ via the embedding $\cN_m^d \hookrightarrow \cN_{md}, (\redtwo{x}_i)_i \mapsto \diag(\redtwo{x}_1, \dots, \redtwo{x}_d)$.
Let $\Sigma_d$ act on $\Pi_m^d$ from the right by place permutations.
Thus, the $G_{m\wr d}$-orbits in $\cN_m^d$ are parametrized by the $\Sigma_d$-orbits in the multipartitions $(\lambda_1, \dots, \lambda_d)$ where $\lambda_i \vdash m$ for all $i$, i.e.,
\eq \label{orbitdecomp}
G_{m\wr d} \backslash \cN_{m}^{d} \overset{1:1}{\longleftrightarrow} \{ (\lambda_1,...,\lambda_{d}) \cdot \Sigma_d ~|~ \lambda_i \vdash m \textup{ for all }i\}.
\endeq
%, corresponding to conjugacy classes of nilpotent $m \times m$ matrices with Jordan forms given by those partitions under the action of $\Sigma_d$. More precisely, we have the following $G_{m \wr d}$-orbit decomposition
%\begin{equation} \label{orbitdecomp}
%	\cN_{m}^{d}= \bigsqcup_{\lambda_i \vdash m} [(\lambda_1,...,\lambda_{d})]
%\end{equation} where we denote $[,]$ to be the class under the $\Sigma_{d}$ action.

%{
%\Def \label{nil}
%Given $F_{\bullet} \in \cF l_{m \wr d}$ and $(n_i)_i \in \cN$. We say that $(n_i)_i$ is  \textit{nilpotent} for the flag $F_{\bullet}$ if $n=\iota((n_i)_i)$ is nilpotent for $F_{\bullet}$, i.e., $nF_j \subset F_{j-1}$ for all $j$.
%\endDef
%
%
%}

%
%It is worth noting that the above action (\ref{nilact}) of $G_{m \wr d}$ on nilpotent elements is compatible with the action on the wreath flags, i.e., 
%\begin{equation} \label{diacom}
%\xymatrixcolsep{5pc}
%\xymatrix{
%		(n_i)_i \ar[d]^{\gamma} \ar@{<->}[r]^{nilpotent} & F((F^i_{\bullet})_i \sigma) \ar[d]^{\gamma} \\
%		(g_in_{w^{-1}(i)}g^{-1}_i)  \ar@{<->}[r]^{nilpotent} & F((g_iF^{w^{-1}(i)}_{\bullet})_i w\sigma) 
%	}
%\end{equation} where $\gamma=(g_i)_i w$, $(n_i)_i$ is nilpotent for the wreath flag $F((F^i_{\bullet})_i\sigma)$.
%
%Note that the action changes the order of the nilpotent types. For example, if the nilpotent type of $n_i$ is $\lambda_i$, then the above calculation tells us that the action moves $(\lambda_1,...,\lambda_d)$ to $(\lambda_{w^{-1}(1)},...,\lambda_{w^{-1}(d)})$. Moreover, $\cN$ is not a single nilpotent $G_{m \wr d}$-orbit. 

\rmk[\red{Nilpotency on wreath flags}]
Under the identification $\cF l_{m \wr d} \equiv \cF l_{m} \wr \Sigma_d$, any  $\cF \in \cF l_{m \wr d} \subseteq \cF l_{md}$ is identified with some element of the form $(F^i_{\bullet})_i w \in \cF l_{m} \wr \Sigma_d$. 
\red{It is standard to check that an element $(\redtwo{x}_i)_i \in \cN_m^d$ is nilpotent on a flag $\cF \in \cF l_{md}$ if and only if $\redtwo{x}_iF^i_j \subseteq F^i_{j-1}$ for all $i,j$.}

\endrmk

\subsection{A Variant of the Springer Resolution}

Denote the Springer resolution of type A by $\mu_m : \widetilde{\cN}_m \to \cN_m$, $(\redtwo{x}, F_\bullet) \mapsto \redtwo{x}$, where
\eq
	 \widetilde{\cN}_m = T^*\cF l_m \equiv \{(\redtwo{x},F_{\bullet}) \in \cN_{m} \times \cF l_m \ | \ \redtwo{x}(F_{i}) \subseteq F_{i-1}\textup{ for all }i \}.
\endeq
%The cotangent bundle to $\cF l_{m \wr d}$ is a vector bundle over $\cF l_{m \wr d}$ whose fiber over a wreath flag is the space of all the nilpotent elements associated with it.
%with the identification  (\ref{geomflag}) of $G_{m \wr d}/B^d_m$ with the wreath full flag variety $\cF l$ in Definition \ref{flag},
%From the identification (\ref{flag}), we know that it is given by 
Instead of the usual Springer resolution $\mu_m^d: \widetilde{\cN}_m^d \to \cN_m^d$,
we consider the following Springer resolution $\mu_{m \wr d}:\widetilde{\cN}_{m \wr d} \rightarrow \cN_m^d, \ (\redtwo{x}, \cF_\bullet) \mapsto \redtwo{x}$,
where 
\eq\label{wreath-springer-resolution}
\widetilde{\cN}_{m \wr d} := T^*\cF l_{m \wr d} 
\equiv \{ ((\redtwo{x}_i)_i, (F^i_{\bullet})_iw) \in \cN_m^d \times \cF l_{m \wr d}  \ | \ \redtwo{x}_iF^i_{j} \subseteq F^i_{j-1} \ \text{for \ all} \ i, j, \ \text{and} \ \red{w \in \Sigma_d} \}.     
\endeq
%\eq%\begin{split}
%&{= \{ \biggl( \gamma \cdot (n_i)_i, \  \gamma(F^{std}_{\bullet},...,F^{std}_{\bullet}, id) \biggr)  \ | \ (n_i)_i \in \cN, \ \gamma=(g_i)_iw \in G_{m \wr d} \}. (\text{may delete}) }
%\end{split}
%\endeq 
%{
%For the second description, note that if $(F^1_{\bullet},..., F^d_{\bullet},w) \in \cF l_{m} \wr \Sigma_{d}$ is a wreath flag, then there exists $(g_1,...,g_d,w) \in G_{m \wr d}$ such that 
%\begin{equation*}
%	(F^1_{\bullet},..., F^d_{\bullet},w) = (g_1,...,g_d,w)(F^{std}_{\bullet},..., F^{std}_{\bullet},id).
%\end{equation*} Thus if $(n_1,...,n_d)$ is a nilpotent element for $(F^1_{\bullet},..., F^d_{\bullet},w)$, then by diagram (\ref{diacom}),  $(n'_1,...,n'_d)$ is given by the conjugation of $(n_1,...,n_d)$ by $(g_1,...,g_d,w)^{-1}$, which is clearly a nilpotent element for $(F^{std}_{\bullet},..., F^{std}_{\bullet},id)$.
%}
%\begin{remark}
%Note that for a nilpotent block matrix $n=\iota((n_i)_i) \in \cN_{md}$, the degree of its nilpotency satisfies $nildeg(n) \leq \text{max}_{i}nildeg(n_i)$.
%\end{remark}
%
Let $\fn_m^d \subseteq \fb_m^d$ be the nilpotent radical. 
Recall $G_{m \wr d}$ acts on $\cN_{m}^d$ and $\cF l_{m \wr d}$  by \eqref{nilact} and \eqref{GactionFL}, respectively, 
and hence $G_{m \wr d}$ acts on $\widetilde{\cN}_{m \wr d} := T^*\cF l_{m \wr d}$. 
Then, we have the following analog of \cite[Corollary 3.1.33]{CG97} regarding $G_{m \wr d}$-equivariant vector bundles.

{\prop \label{nilbuniso}
As  $G_{m \wr d}$-equivariant vector bundles, $G_{m \wr d} \times_{B_m^d} \fn_m^d \cong T^*\cF l_{m \wr d}$ via
\eq 
(\gamma, \red{x} ) \mapsto   ( \gamma \cdot \red{x}, \  \gamma (F^{std}_{\bullet},...,F^{std}_{\bullet},1_{\Sigma_d}) ),
\quad \gamma \in G_{m \wr d}, \red{x} \in \fn_m^d.
\endeq
\endprop}

\begin{proof}
For injectivity, suppose that %$(\gamma_1,\nu_1), (\gamma_2,\nu_2)$ lie in $G_{m \wr d} \times_{B_m^d} \fn_m^d$ such that 
$( \gamma_1 \cdot \red{x_1}, \  \gamma_1 (\Fstd_{\bullet},...,\Fstd_{\bullet},1_{\Sigma_d}) )=( \gamma_2 \cdot \red{x_2}, \  \gamma_2 (\Fstd_{\bullet},...,\Fstd_{\bullet},1_{\Sigma_d}))$
for some $(\gamma_i,\red{x_i})\in G_{m \wr d} \times_{B_m^d} \fn_m^d$.
%\endeq 
Then, %$ \gamma_1 (F^{std}_{\bullet},...,F^{std}_{\bullet},1_{\Sigma_d}) = \gamma_2 (F^{std}_{\bullet},...,F^{std}_{\bullet},1_{\Sigma_d})$ implies that 
%\begin{equation*}
$\gamma_1^{-1}\gamma_2 \in \Stab_{G_{m \wr d}}((\Fstd_{\bullet},...,\Fstd_{\bullet},1_{\Sigma_d}))=B^d_m$,
%\end{equation*} 
and thus $\gamma_2=\gamma_{1}b$ for some $b \in B^d_m$. 
Moreover, %it also gives the following equality
\eq
\gamma_2 \cdot \red{x_2} = \gamma_2\red{x_2}\gamma_2^{-1}=\gamma_1b\red{x_2}b^{-1}\gamma_1^{-1}=\gamma_1\red{x_1}\gamma_1^{-1}=\gamma_1 \cdot \red{x_1},
\endeq
which implies that $b\red{x_2}b^{-1}=\red{x_1}$. Therefore, $(\gamma_1,\red{x_1})=(\gamma_2b^{-1},b\red{x_2}b^{-1})=(\gamma_2,\red{x_2})$.
 	
For surjectivity, assume that $((\redtwo{x}_i)_i, (F^i_{\bullet})_iw)  \in T^*\cF l_{m \wr d}$.
Then, there exists $\gamma=(g_i)_iw \in G_{m \wr d}$ such that $\gamma (F^{std}_{\bullet},...,F^{std}_{\bullet},1_{\Sigma_d})=(F^i_{\bullet})_iw$. 
The condition $\redtwo{x}_iF^i_{j} \subseteq F^i_{j-1}$ implies that %$g_i^{-1}n_ig_i$ is an nilpotent element for the standard flag of $\CC^m$, i.e., 
$g_i^{-1}\redtwo{x}_ig_i \in \fn_m$ for all $i$. 
Thus, $((g_i)_iw ,(g^{-1}_{w(i)}\redtwo{x}_{w(i)}g_{w(i)})_i)$ lies in the inverse image, and we are done.
%Finally, the $G_{m \wr d}$-equivariant structure is clear to see.
\end{proof}
We identify $\gl_m^d \cong \gl_m^{d,*}$ by an invariant bilinear form on $\gl_m^d$, 
and let $(\fb^d_m)^{\perp} \subseteq \gl_m^{d,*}$ be the image of $\fn^d_m$ under the identification. 
Then, we obtain the following analog of \cite[Proposition 1.4.9]{CG97}.
%Then, like \cite[Proposition 1.4.9]{CG97}, we have the following corollary which will be used in the next subsection for proof of the decomposition of Steinberg variety into union of conormal bundles.

{\cor \label{gmdequivariantiso}
There is a $G_{m \wr d}$-equivariant isomorphism
$T^*\cF l_{m \wr d} %\cong G_{m \wr d} \times_{B_m^d} \fn_m^d 
\cong  G_{m \wr d} \times_{B_m^d} (\fb^d_m)^{\perp}$.
\endcor}

\subsection{Wreath Steinberg Varieties}
%We already defined the nilpotent cone and the cotangent bundle from the last subsection. We have the following definition by mimicking the usual definition of the Steinberg variety.

Using our unconventional Springer resolution $\widetilde{\cN}_{m \wr d} := T^*\cF l_{m\wr d}$, we introduce the following Steinberg varieties whose irreducible components are conormal bundles indexed by $\Sigma_m \wr \Sigma_d$. 

\Def \label{Steinberg}
Define the \textit{wreath Steinberg variety} by $Z_{m \wr d} := \widetilde{\cN}_{m \wr d}  \times_{\cN_m^d}\widetilde{\cN}_{m \wr d} $, i.e, 
\eq
Z_{m \wr d} =\{ (  ( \redtwo{x}, \cF_{\bullet}) , ( \redtwo{x}', \cF'_{\bullet})  ) \in \widetilde{\cN}_{m \wr d}  \times \widetilde{\cN}_{m \wr d}  \ | \ \redtwo{x} = \redtwo{x}' \}.
\endeq
\endDef
\red{From the Bruhat decomposition \eqref{bruhat}, we know that the $G_{m \wr d}$-orbits of $G_{m \wr d}/B^d_{m} \times G_{m \wr d}/B^d_{m}$ are parametrized by elements in the Weyl group $\Sigma_{m} \wr \Sigma_{d}$, more precisely 
\eq
(G_{m \wr d}/B^d_{m}) \times (G_{m \wr d}/B^d_{m})= \bigsqcup_{w \in \Sigma_{m} \wr \Sigma_{d}}\cO'(w),
\quad \cO'(w):=  G_{m \wr d} \cdot ( B^d_m, w B^d_m).
\endeq Moreover, we also have the identification $G_{m \wr d}/B^d_m \simeq \cF l_{m \wr d}$ from Proposition \ref{compatabilitywrfl} (b). Thus, there is a $G_{m \wr d}$-orbits decomposition
\eq
\cF l_{m \wr d} \times \cF l_{m \wr d}= \bigsqcup_{w \in \Sigma_{m} \wr \Sigma_{d}}\cO(w), \quad
\cO(w):=  G_{m \wr d} \cdot ( (\Fstd_\bullet)_i, w (\Fstd_\bullet)_i).
\endeq
}

Thus, for $w\in \Sigma_m \wr \Sigma_d$, there is a short exact sequence via the identification (Proposition~\ref{compatabilitywrfl}(b)):
\eq\label{def:Yw}
0 \rightarrow Y_{w} \rightarrow T^*(\cF l_{m \wr d} \times \cF l_{m \wr d}) \rightarrow T^*\cO(w) \rightarrow 0,
\endeq
 where $Y_w := T^*_{\cO(w)}(\cF l_{m \wr d} \times \cF l_{m \wr d})$ is the cornormal bundle. 
On the other hand, there is a natural projection $\pi:Z_{m \wr d} \rightarrow G_{m \wr d}/B^d_m \times G_{m \wr d}/B^d_m$ given by
\eq
Z_{m \wr d} \hookrightarrow T^*\cF l_{m \wr d} \times T^*\cF l_{m \wr d} \xrightarrow{\pi' \times \pi'} (G_{m \wr d}/B^d_m) \times (G_{m \wr d}/B^d_m),
\endeq
where $\pi' : T^*\cF l_{m \wr d} \to G_{m\wr d}/B_m^d$ is the evident projection.
In the following, \red{we show that the conormal bundle $Y_w$ admits an alternative description.}
 
\lemma\label{Y=Z}
For $w\in \Sigma_m \wr \Sigma_d$, let $Z_{w} \coloneqq \pi^{-1}(\cO(w))$.
Then, $Z_w=Y_w$.
\endlemma

\proof
%Write $w=(w_1,...,w_d, w') \in \Sigma_m \wr \Sigma_d$.
An arbitrary geometric point in the orbit $\cO(w)$ is of the form $s = (F^{(1)}, F^{(2)})$ where
\eq
F^{(j)}=(g^{(j)}_i)_i\sigma_{j}(F^{std}_{\bullet},...,F^{std}_{\bullet}), 
\quad ( j =1,2, \ g^{(j)}_i \in \GL_m, \sigma_j \in \Sigma_d)
\endeq 
such that  $((g^{(1)}_i)_i \sigma_1)^{-1} (g^{(2)}_i)_i \sigma_2 \in B^d_m w B^d_m$.
Then, {by Corollary \ref{gmdequivariantiso}}, the fiber $Y_{w,s}$ of the cornormal bundle of $\cO(w)$ at the point $s$ consists of elements of the form
$( (\redtwo{x}^{(1)}, \redtwo{x}^{(2)}), (F^{(1)}, F^{(2)}) )$, where {$\redtwo{x}^{(1)}, \redtwo{x}^{(2)} \in \gl_m^{d,*}$} and $(\redtwo{x}^{(1)}, \redtwo{x}^{(2)})$ annihilates the tangent space $T_{s}\cO(w)$. 
Since $\Stab_{G_{m \wr d}}(\Fstd_{\bullet},...,\Fstd_{\bullet}) = B^d_m$, the stabilizer of $F^{(j)}$ in $G_{m \wr d}$ is a Borel subgroup $B_j$ of $G_{m\wr d}$, i.e.,
\eq
	\Stab_{G_{m \wr d}}(F^{(j)})=(g^{(j)}_i)_i \sigma_j B^d_m ((g^{(j)}_i)_i\sigma_j)^{-1} = B_j.
\endeq
Therefore, \red{the orbit $\cO(w)= G_{m \wr d}\cdot s$ }is given by
\eq
	\cO(w) = G_{m \wr d} /(\Stab(F^{(1)}) \cap\Stab(F^{(2)})) = G_{m \wr d}/(B_1 \cap B_2).
\endeq
Write $\fb_j = \mathop{\mrm{Lie}} B_j$. Since $G_{m \wr d}/(B_1 \cap B_2)=\{g(B_1\cap B_2) \ | \ g \in G_{m \wr d} \} \cong \{ (gB_1,gB_2) \ | \ g \in G_{m \wr d} \}$,
\eq
	T_s\cO(w) \cong T_sG_{m \wr d}/(B_1 \cap B_2) \cong T_s\{ (gB_1,gB_2) \ | \ g \in G_{m \wr d} \} =\{ (\redtwo{a}\fb_1,\redtwo{a}\fb_2) \ | \ \redtwo{a} \in \gl^d_m \}.
\endeq
Now, the condition that $(\redtwo{x}^{(1)},\redtwo{x}^{(2)})$ annihilates the tangent space $T_{s}\cO(w)$ is equivalent to 
	$\langle \redtwo{x}^{(1)}, x \rangle + \langle \redtwo{x}^{(2)},x \rangle =0$ for all $x \in \gl^d_m$, 
which implies that $\redtwo{x}^{(1)}=-\redtwo{x}^{(2)}$.  Thus, we have 
\eq
Y_{w,s}=\{ (  (\redtwo{x}^{(1)}, \redtwo{x}^{(2)}), (F^{(1)}, F^{(2)}) ) \in T_s^*(\cF l_{m \wr d} \times \cF l_{m \wr d})  \ | \ \redtwo{x}^{(1)}=-\redtwo{x}^{(2)} \}.
\endeq
Finally, $Y_{w,s}$ coincides with the fiber $Z_{w,s}$ thanks to the sign isomorphism %\cite[(3.3.2)]{CG97}
%On the other hand, we have the following sign isomorphism from \cite[(3.3.2)]{CG97}, 
$T^*(\cF l_{m \wr d} \times \cF l_{m \wr d}) \xrightarrow{sign} T^*\cF l_{m \wr d} \times T^*\cF l_{m \wr d}=  \widetilde{\cN}_{m \wr d}  \times \widetilde{\cN}_{m \wr d}$, 
$(  (\redtwo{x}^{(1)}, \redtwo{x}^{(2)}), (F^{(1)}, F^{(2)}) ) \mapsto (  (\redtwo{x}^{(1)},F^{(1)}), (-\redtwo{x}^{(2)}, F^{(2)}) )$.
%Hence, under the above identification, $Y_{w,s}$ coincides with the fiber $Z_{w,s}$.
\endproof

Note that since $G_{m \wr d}$ is disconnected, the connected components for each orbit $\cO(w)$ are indexed by $\tau \in \Sigma_d$, and are denoted by 
\begin{equation} \label{componentoforbits} 
\cO(w)_{\tau}=\{ ( (g_i\Fstd_\bullet)_i\tau, (g_iw_{\tau^{-1}(i)}\Fstd_\bullet)_i\tau\sigma )\ | \ (g_i)_i \in \GL_m^d \}.
\end{equation}
For $w \in \Sigma_m \wr \Sigma_d$ and $\tau \in \Sigma_d$, we further define
\eq\label{eq:Ywt}
Y_{w,\tau} \coloneqq \pi^{-1}(\cO(w)_\tau)=\{(\redtwo{x},\cF_{\bullet},\cF'_{\bullet}) \in Z_{m \wr d} \ | \ (\cF_{\bullet},\cF'_{\bullet}) \in \cO(w)_\tau \}.
\endeq

%obtain the decomposition $Z_{m \wr d} = \bigsqcup_{w \in \Sigma_{m} \wr \Sigma_{d}, \tau \in \sigma
%_d} Y^\tau_{w}$. Moreover, we know from (\ref{componentoforbits}) that each $Z_w$ is also disconnected, and $\Sigma_d$ indexes the connected components. We define $Y^{\tau}_w \coloneqq \pi^{-1}(\cO(w)^\tau)$ for each $w \in \Sigma_m \wr \Sigma_d$ and $\tau \in \Sigma_d$ such that $Y_w=\bigsqcup_{\tau \in \Sigma_d} Y^{\tau}_w$. 

\begin{prop} \label{basis} 
We have $Z_{m \wr d}=\bigsqcup_{w \in \Sigma_m \wr \Sigma_d, \tau \in \Sigma_d} Y_{w,\tau}$. 
Moreover, each irreducible component of $Z_{m \wr d}$ is the closure \red{$\overline{Y_{w,\tau}}$} for a uniquely $w \in \Sigma_m \wr \Sigma_d$ and $\tau \in \Sigma_d$.
\end{prop} 

\begin{proof}
Observe that $Z_{m \wr d} \subseteq \cN_m^d \times G_{m \wr d}/B^d_m \times G_{m \wr d}/B^d_m=\cN_m^d \times \bigsqcup_{w \in \Sigma_m \wr \Sigma_d, \tau \in \Sigma_d} \cO(w)_{\tau}$. Thus 
\eq
Z_{m \wr d}
%=Z_{m \wr d} \cap ( \cN_m^d \times \textstyle{\bigsqcup_{w \in \Sigma_m \wr \Sigma_d, \tau \in \Sigma_d} \cO(w)^{\tau}} ) 
= \bigsqcup_{w \in \Sigma_m \wr \Sigma_d, \tau \in \Sigma_d} Z_{m \wr d} \cap (\cN_m^d \times  \cO(w)_{\tau})
= \bigsqcup_{w \in \Sigma_m \wr \Sigma_d} Y_{w,\tau}. 
\endeq 
The proposition then follows from Lemma~\ref{Y=Z}.
\end{proof}
{For each $1\leq i \leq d$, recall $\varphi^{\tau}_i$ from  \eqref{def:phii}, the Steinberg variety $Z_m$ of type A can be embedded into the wreath Steinberg variety at the $(\tau,\tau)$-component in the following sense :
\eq\label{def:zetai}
\zeta^{\tau}_i :Z_m \to Z_{m \wr d},
\quad (\redtwo{x}, F_\bullet, F'_\bullet) \mapsto (\redtwo{y}, \varphi^{\tau}_i(F_\bullet), \varphi^{\tau}_{\tau^{-1}(i)}(F'_\bullet)), \ \textup{where} \ \redtwo{y}_j = \begin{cases} \redtwo{x} &\tif j = \tau^{-1}(i); \\ 0 &\textup{otherwise.} \end{cases}
\endeq
}

%The following corollary is a direct consequence.
%
%\begin{cor} \label{fundamentalclasses}
%As a vector space, $H(Z_{m \wr d})$ is spanned by the basis given by the set of fundamental classes $\{[\overline{Y^{\tau}_w}]\}_{w \in \Sigma_m \wr \Sigma_d}^{\tau \in \Sigma_d}$.
%\end{cor}

%=====

% INCORRECT 
%\begin{prop} \label{conormal}
%There is a decomposition $Z_{m \wr d} = \bigsqcup_{w \in \Sigma_{m} \wr \Sigma_{d}} Y_{w}$.
%Moreover, each irreducible component of $Z_{m \wr d}$ is a closure of $Y_w$ for a unique  $w \in \Sigma_{m} \wr \Sigma_{d}$.
%\end{prop}
%\begin{proof}
%
%Observe that $Z_{m \wr d} \subseteq \cN_m^d \times G_{m \wr d}/B^d_m \times G_{m \wr d}/B^d_m=\cN_m^d \times \bigsqcup_{w \in \Sigma_m \wr \Sigma_d} \cO(w)$. Thus 
%\begin{equation*}
%	Z_{m \wr d}
%	=Z_{m \wr d} \cap ( \cN_m^d \times \textstyle{\bigsqcup_{w \in \Sigma_m \wr \Sigma_d} \cO(w)} ) 
%	= \bigsqcup_{w \in \Sigma_m \wr \Sigma_d} Z_{m \wr d} \cap (\cN_m^d \times  \cO(w))
%	= \bigsqcup_{w \in \Sigma_m \wr \Sigma_d} Z_w. 
%\end{equation*} 
%The proposition then follows from Lemma~\ref{Y=Z}.
%\end{proof}

\section{Borel-Moore Homology of Wreath Steinberg Varieties}%by direction computation}

%In this section, we prove the main result (Theorem \ref{mainthm}) by construct an explicit isomorphism from $H_{top}^{BM}(Z_{m \wr d})$ to $\CC[\Sigma_m \wr \Sigma_d]$.
\subsection{Borel-Moore Homology}
An introduction to Borel-Moore homology  can be found in \cite[\S2.6--7]{CG97}. In particular, we will use the following result:

\begin{prop}\label{prop:HBM}
Let $\pi:M\to N$ be a proper map from a smooth complex manifold onto a variety, and let $Z = M \times_N M := \{(m, m')\in M^2 ~|~ \pi(m) = \pi(m')\}$.
Denote by $\HBM(Z;K)$ the group of Borel-Moore homology  with coefficients in a field $K$ of characteristic 0,
and by $\HBMt(Z; K)$ its top degree homologies. Then,
 \begin{enua}
 \item {\cite[Corollary 2.7.41, 2.7.48]{CG97}}
The group $\HBMt(Z;K)$ has a structure of an associative unital $K$-algebra.

 \item {\cite[Lemma 2.7.49]{CG97}}
 Assume that  the set of irreducible components $\{L_w\}_{w\in W}$ of $Z$ is indexed by a finite set $W$.
 If all the components have the same dimension, then the fundamental classes $\{[L_w] \}_{w\in W}$ form a basis of the algebra $\HBMt(Z;K)$.
 \end{enua}
 \end{prop}

Following \cite[\S3]{CG97}, the above-mentioned results apply to the the Steinberg variety $Z_m$ for $\GL_{m}$. 
Within this section, we abbreviate $\HBMt(-;\QQ)$ by $H(-)$ when it is convenient.
It is further proved therein that the top Borel-Moore homology with rational coefficients gives a geometric realization of the group algebra of $\Sigma_m$ as below:

 \begin{prop}\label{prop:CGiso}
 \begin{enua}
\item {\cite[Theorem 3.4.1, Claim 3.4.13]{CG97}}
There is an isomorphism $\QQ[\Sigma_m] \to \HBMt(Z_m;\QQ)$ of algebras, where each $w\in \Sigma_m$ is sent to a certain fundamental class $[\Lambda^0_w]$ obtained by taking specialization. 
 \item {\cite[Lemma 3.4.14]{CG97}}
For $w \in \Sigma_m$, let $Y'_w$ be the conormal bundle to the orbit $\cO(w) \in \GL_m \backslash (\cF l_m \times \cF l_m)$.
Then, with respect to the Bruhat order on $\Sigma_m$, there are $n_{x,w} \in \QQ$ such that 
\begin{equation} \label{classconormal}
	[\Lambda^0_w] = { {\textstyle \sum_{y \leq w}} \red{n_{y,w}}{[\overline{Y'_{y}}]},
	\quad n_{w,w}=1}.
	%= [Y_w] + \sum_{y < w} n_{w,y}[Y_w].
\end{equation} 
%{\cite[Theorem 3.4.1, Claim 3.4.13]{CG97}}
%There is an algebra isomorphism $\QQ[\Sigma_m] \to \HBMt(Z_m), w \mapsto [\Lambda^0_{w}]$, where  $[\Lambda^0_w]$ is a certain fundamental class obtained by taking specialization. 
\end{enua}
\end{prop}

\rmk
Following \cite[\S3.4]{CG97}, for each $w \in \Sigma_m$ and a fixed regular vector $h$ in the abstract Cartan subalgebra $\mathfrak{H}$ of $\fgl_m$, there is a fiber  
$\Lambda_w^{h}:= (\nu \times \nu)\inv(w.h, h) $, where  $\nu$ is the projection $\widetilde{\fgl}_m \to \mathfrak{H}$. That is,
$\Lambda_w^{h} = \{(x, \fb, \fb') \in \fgl_m \times \fB_m \times \fB_m ~|~ x \in \fb \cap \fb',\ \nu(x,\fb) = w.h, \ \nu(x, \fb') = h\}$.
By taking the specialization map (see \cite[(2.6.30)]{CG97}), one defines a class $[\Lambda_w^0]$ that is not the class of some subvariety, and does not depend on the choice of $h$.

It is tempting to mimic this approach to construct the desirable classes for the generators $t_k$'s in $\Sigma_m \wr \Sigma_d$ (see Example \ref{ex:Smd}). 
However, such a specialization map produces classes $[\Lambda_{t_k}^0]$ that do not sit inside $\HBMt(Z_{m \wr d})$, in general.
We have to construct explicitly classes in $\HBMt(Z_{m \wr d})$ that play the same role as the elements $t_k$'s in the group algebra.
\endrmk

\red{Finally, we mention the notion of set-theoretic composition which will appear during the computation of convolutions in Borel--Moore homology. Let $M_1$, $M_2$, $M_3$ be oriented $C^{\infty}$-manifolds, and let $Z_{12} \subset M_1 \times M_2$, $Z_{23} \subset M_2 \times M_3$ be closed subsets. Then we define their set-theoretic composition to be 
\eq \label{set-theoreticcomposition}
Z_{12} \circ Z_{23} \coloneqq \{(m_1,m_3) \in M_1 \times M_3 \ | \ \exists \ m_2 \in M_2 \ \text{such that} \ (m_1,m_2) \in Z_{12}, \ (m_2,m_3) \in Z_{23} \}.
\endeq}

%===========================
\subsection{The Algebra $\HBMt(Z_{m \wr d};\QQ)$}
%==========================
Combining Propositions \ref{basis} and \ref{prop:HBM}(b), we obtain a basis
$\{[\overline{Y_{w,\tau}}]\}$ of the  algebra $H(Z_{m \wr d})$ indexed by
$(w, \tau) \in (\Sigma_m \wr \Sigma_d)\times\Sigma_d$. 
One may think that this Steinberg variety  $Z_{m \wr d}$ may not be the right choice since the algebra is too large.
However, we are  able to construct a certain subalgebra $A_{m \wr d} \subseteq H(Z_{m \wr d})$ that realizes the group algebra $\QQ[\Sigma_m \wr \Sigma_d]$. 
Moreover, it is essential to work with such a algebra $H(Z_{m \wr d})$ with a larger dimension, so that the corresponding isotypic components of top Borel-Moore homology of Springer fibers indeed realizes simple $\QQ[\Sigma_m \wr \Sigma_d]$-modules arising from the Clifford theory. See Section \ref{sec:irrep}.

Next, we are to describe the multiplication rules with respect to this basis.
Since $\pi$ is proper and finite, $\overline{Y_{w,\tau}} :=\overline{\pi^{-1}(\cO(w)_{\tau})} =\pi^{-1}(\overline{\cO(w)_{\tau}})$. 
By the description in \eqref{componentoforbits}, 
$\overline{\cO(w)_{\tau}}=\bigcup_{w' \leq_{m \wr d} w } \cO(w')_{\tau}$, and hence
%\end{equation*} Thus
%\eq
%\overline{Y_{w,\tau}} =\overline{\pi^{-1}(\cO(w)_{\tau})}=\pi^{-1}(\overline{\cO(w)_{\tau}}) 
%=\pi^{-1}(\bigcup_{w' \leq_{m \wr d} w } \cO(w')_{\tau}) 
%= \bigcup_{w' \leq_{m \wr d} w} Y_{w',\tau},
%\endeq 
\eq\label{eq:overY}
[\overline{Y_{w,\tau}}]={\textstyle \sum_{w' \leq_{m \wr d} w} [Y_{w',\tau}]}.
\endeq
Note that since $Y_{w,\tau}$ is generally non-compact, it does not have a fundamental class in the sense of singular homology.
That being said, there is still a fundamental class $[Y_{w,\tau}] \in H(Y_{w,\tau})$ in the Borel-Moore homology (see \cite[Section 2.6.12]{CG97}).  
Here, an inclusion $H(Y_{w,\tau}) \rightarrow H(Z_{m \wr d})$ exists  if $Y_{w,\tau} \hookrightarrow Z_{m \wr d}$ is proper,
which means that we can only regard $[Y_{w,\tau}]$ as an element in $H(Z_{m \wr d})$ when $Y_{w, \tau}$ is closed.

In other words, the algebra structure on $H(Z_{m \wr d})$ can be determined from the following multiplication lemma between $[Y_{w,\tau}]$'s when at least one of the $w$'s is purely in $1\times \Sigma_d \subseteq \Sigma_m \wr \Sigma_d$. Such a lemma does not apply to general elements in  $H(Z_{m \wr d})$.
%----------------------------------
\begin{lemma} \label{convolution}
Suppose that $\tau, \tau' \in \Sigma_d$ and 
$w := (w_i)_i \sigma, w':= (w'_i)_i \sigma' \in \Sigma_m \wr\Sigma_d$.  Set
\[
Z_{12} :=Y_{w,\tau} =T^*_{\cO(w)_{\tau}}(\cF l_{m \wr d} \times \cF l_{m \wr d}), %\subset \widetilde{\cN}_{m \wr d} \times \widetilde{\cN}_{m \wr d},  \\
\quad
Z_{23}:=Y_{w',\tau'} =T^*_{\cO(w')_{\tau'}}(\cF l_{m \wr d} \times \cF l_{m \wr d}).% \subset \widetilde{\cN}_{m \wr d} \times \widetilde{\cN}_{m \wr d}.
\]
Then, 
the following equality holds in $H(Z_{13})$, where \red{$Z_{13}=Z_{12} \circ Z_{23}$ is the set-theoretic composition defined in \eqref{set-theoreticcomposition}}
\begin{equation}\label{eq:convYY}
[Y_{w,\tau}] \ast [Y_{w',\tau'}] =[Z_{12}] \ast [Z_{23}] =\begin{cases}
0 & \text{if} \ \tau\sigma \neq \tau' \\
[Y_{ww',\tau}] & \text{if} \ \tau\sigma=\tau', \ \text{and either} \  (w_i)_i=(e)_i \ \text{or} \ (w'_i)_i=(e)_i.
\end{cases}
\end{equation}
Moreover, if either the inclusions $Z_{i,i+1} \hookrightarrow Z_{m \wr d}$ for $i\in\{1,2\}$ is proper, then \eqref{eq:convYY} holds in $H(Z_{m \wr d})$.
%		\end{enumerate}
	\end{lemma}
%------------------------------------
\begin{proof}
	Recall that
	\eq
\begin{split}
			Y_{w,\tau}
			&=\pi^{-1}(\cO(w)_{\tau})
			\\
			&=\{ (\redtwo{x},\cF_{\bullet},\cF'_{\bullet}) \in Z_{m \wr d} \ | \ (\cF_{\bullet},\cF'_{\bullet}) \in ( (g_i\Fstd_\bullet)_i\tau, (g_iw_{\tau^{-1}(i)}\Fstd_\bullet)_i\tau\sigma ), \
			(g_i)_i \in \GL^d_m  \}. \label{wsig} 
			%\\
			%Y_{w',\tau'}&=\{ (\nu,\cF_{\bullet},\cF'_{\bullet}) \in Z_{m \wr d} \ | \ (\cF_{\bullet},\cF'_{\bullet})=( (g'_i\Fstd_\bullet)_i\tau', (g'_iw'_{\tau'^{-1}(i)}\Fstd_\bullet)_i\tau'\sigma' ), \ (g'_i)_i \in \GL^d_m \}. \label{w'sig'}
		\end{split}
		\endeq
	When $\tau\sigma \neq \tau'$, the multiplication is zero since the composition $Y_{w,\tau} \circ Y_{w',\tau'}$ is empty.
	When $\tau\sigma = \tau'$, our strategy is to use \cite[Theorem 2.7.26]{CG97}, which is only applicable if the intersection {$\pi_{12}^{-1}(\cO(w)_{\tau}) \cap \pi_{23}^{-1}(\cO(w')_{\tau'})$ is transverse in the ambient space ${\cF l}_{m \wr d} \times {\cF l}_{m \wr d} \times  {\cF l}_{m \wr d}$}. 
	
	We begin with the case when $w_i = e = w'_i$ for all $i$.
	{A direct calculation shows that
		\eq
		\dim \pi_{12}^{-1}(\cO(w)_{\tau})=2d\dim \cF l_m=dm(m-1)=\dim  \pi_{23}^{-1}(\cO(w')_{\tau'}).
		\endeq Thus, $\codim\pi_{12}^{-1}(\cO(w)_{\tau})=\codim \pi_{23}^{-1}(\cO(w')_{\tau'})=d\dim \cF l_m$.
	}
	%\eq
	%\dim \pi_{12}^{-1}(Y_{(e)_i\sigma,\tau}) =  \dim Y_{(e)_i\sigma,\tau} = d \dim p^{-1}(\cO(e)) =dm(m-1) =
	%\dim \pi_{23}^{-1}(Y_{(e)_i\sigma',\tau'}),
	%\endeq 
	%where $p:Z_{m}=\widetilde{\cN}_m \times_{\cN_m} \widetilde{\cN}_m \rightarrow \cF l_{m} \times \cF l_{m}$ is the natural projection.
	%Thus, $\codim\pi_{12}^{-1}(Y_{(e)_i\sigma,\tau})=\codim \pi_{23}^{-1}(Y_{(e)_i\sigma',\tau'})=0$ since
	%$\dim \widetilde{\cN}_{m \wr d} \times_{\cN^d_m} \widetilde{\cN}_{m \wr d} \times_{\cN^d_m} \widetilde{\cN}_{m \wr d} = \dim \widetilde{\cN}_{m \wr d} = d \dim \widetilde{\cN}_{m} = dm(m-1)$.
	
	On the other hand, we also have 
	{
		\eq \label{eq:p12p13YY}
		\pi_{12}^{-1}(\cO(w)_{\tau}) \cap  \pi_{23}^{-1}(\cO(w')_{\tau'}) = \{ (\cF_{\bullet},\cF'_{\bullet},\cF''_{\bullet}) \ | \ (\cF_{\bullet},\cF'_{\bullet}) \in \cO((e)_i\sigma)_{\tau}, \ (\cF'_{\bullet},\cF''_{\bullet}) \in \cO((e)_i\sigma')_{\tau'}  \}
		\endeq 
	}
	%\eq
	%\pi_{12}^{-1}(Y_{(e)_i\sigma,\tau}) \cap  \pi_{23}^{-1}(Y_{(e)_i\sigma',\tau'}) =\{ (\nu,\cF_{\bullet},\cF'_{\bullet},\cF''_{\bullet}) \ | \ (\cF_{\bullet},\cF'_{\bullet}) \in \cO((e)_i\sigma)_{\tau}, \ (\cF'_{\bullet},\cF''_{\bullet}) \in \cO((e)_i\sigma')_{\tau'}  \}.    
	%\endeq
	Since $\tau\sigma=\tau'$, the condition on the wreath flags $(\cF_{\bullet},\cF'_{\bullet},\cF''_{\bullet})$ appearing in \eqref{eq:p12p13YY} is equivalent to 
	\eq
	(\cF_{\bullet},\cF'_{\bullet},\cF''_{\bullet}) \in \GL^d_m \cdot  (\Fstd_{\bullet}\tau, \Fstd_{\bullet}\tau\sigma,\Fstd_{\bullet}\tau\sigma\sigma').
	\endeq 
	The space of such triple flags can thus be identified with  the flag variety $\cF l_m^d$, and hence 
	{\eq
		\dim \pi_{12}^{-1}(\cO(w)_{\tau}) \cap  \pi_{23}^{-1}(\cO(w')_{\tau'})= d \dim \cF l_m.
		\endeq}
	We can then conclude that 
	{\eq
		\codim\pi_{12}^{-1}(\cO(w)_{\tau}) \cap  \pi_{23}^{-1}(\cO(w')_{\tau'}) =\codim\pi_{12}^{-1}(\cO(w)_{\tau}) + \codim \pi_{12}^{-1}(\cO(w')_{\tau'}).
		\endeq}
	That is, the intersection {$\pi_{12}^{-1}(\cO(w)_{\tau}) \cap  \pi_{23}^{-1}(\cO(w')_{\tau'})$} is transverse, and their composition is given by 
	{\begin{align*}
			\cO(w)_{\tau} \circ \cO(w')_{\tau'}&=\{ (\cF_{\bullet},\cF'_{\bullet}) \ | \ (\cF_{\bullet},\cF'_{\bullet}) \in \cO((e)_i\sigma\sigma')_{\tau} \}=\cO(ww')_{\tau}. 
	\end{align*}} Since the fiber of {$\pi_{13}:\pi_{12}^{-1}(\cO(w)_{\tau}) \cap  \pi_{23}^{-1}(\cO(w')_{\tau'}) \rightarrow \cO(w)_{\tau} \circ \cO(w')_{\tau'}$} is a point, the multiplication formula follows due to \cite[Theorem 2.7.26]{CG97}.
	
	Next, we consider the case when $(w_i)_i=(e)_i$ and $(w'_i)_i \neq (e)_i$. {Note that $\cO((w'_i)_i\sigma')_{\tau'} \cong \prod_{i=1}^d \cO(w'_i)$, so
		\eq 
		\dim \pi_{23}^{-1}(\cO(w')_{\tau'})=d\dim \cF l_m  + \sum\nolimits_{i=1}^d \dim \cO(w'_i)
		\endeq  It follows that $\codim \pi_{23}^{-1}(\cO(w')_{\tau'}) = 2d\dim \cF l_m  - \sum_{i=1}^d \dim \cO(w'_i)$.
	}
	%\eq
	%\pi_{23}^{-1}(Y_{(w'_i)_i\sigma',\tau'})=\{ (\nu,\cF_{\bullet},\cF'_{\bullet},\cF''_{\bullet}) \ | \ (\cF'_{\bullet},\cF''_{\bullet}) \in \cO((w'_i)_i\sigma')_{\tau'} \} \subseteq \widetilde{\cN}_{m \wr d} \times_{\cN^d_m} \widetilde{\cN}_{m \wr d} \times_{\cN^d_m} \widetilde{\cN}_{m \wr d}.   
	%\endeq
	%Note that $\cO((w'_i)_i\sigma')_{\tau'} \cong \prod_{i=1}^d \cO(w'_i)$, and
	%\eq
	%$\dim \pi_{23}^{-1}(Y_{(w'_i)_i\sigma',\tau'})=\dim Y_{(w'_i)_i\sigma',\tau'}=\sum_{i=1}^d \dim p^{-1}(\cO(w'_i))$.
	%\endeq 
	%It follows that $\codim \pi_{23}^{-1}(Y_{(w'_i)_i\sigma',\tau'}) = \sum_{i=1}^d ( \dim\widetilde{\cN}_m-\dim p^{-1}(\cO(w'_i)) )$.
	
{Since $(w_i)_i=(e)_i$, from the first case, we know  that $\codim \pi_{12}^{-1}(\cO(w)_{\tau})=d\dim\cF l_m$, and hence 
		\eq
		\codim\pi_{12}^{-1}(\cO(w)_{\tau}) +  \codim \pi_{23}^{-1}(\cO(w')_{\tau'}) = 3d\dim \cF l_m  - \sum\nolimits_{i=1}^d \dim \cO(w'_i).
		\endeq On the other hand, 
		\eq\label{eq:p12p13YY2}
		\pi_{12}^{-1}(\cO((e)_i\sigma)_{\tau}) \cap  \pi_{23}^{-1}(\cO((w'_i)_i\sigma')_{\tau'}) =\{ (\cF_{\bullet},\cF'_{\bullet},\cF''_{\bullet}) \ | \ (\cF_{\bullet},\cF'_{\bullet}) \in \cO((e)_i\sigma)_{\tau}, \ (\cF'_{\bullet},\cF''_{\bullet}) \in \cO((w'_i)_i\sigma')_{\tau'}  \}. 
		\endeq 
	} Since $\tau\sigma=\tau'$,  the condition on the wreath flags appearing in \eqref{eq:p12p13YY2} is equivalent to 
	%\eq
	$(\cF_{\bullet},\cF'_{\bullet},\cF''_{\bullet}) \in \GL^d_m \cdot  (\Fstd_{\bullet}\tau, \Fstd_{\bullet}\tau\sigma,(w'_{\tau'^{-1}(i)})_i\Fstd_{\bullet}\tau\sigma\sigma')$.
	%\endeq 
	The space of such triple flags can then be identified with $\prod_{i=1}^d \cO(w'_i)$, and thus, %$\pi_{12}^{-1}(Y_{(e)_i\sigma,\tau}) \cap  \pi_{23}^{-1}(Y_{(w'_i)_i\sigma',\tau'})  \cong \pi_{23}^{-1}(Y_{(w'_i)_i\sigma',\tau'})$. Then,
	{\eq
		\begin{split}
			&\codim \pi_{12}^{-1}(\cO((e)_i\sigma)_{\tau}) \cap  \pi_{23}^{-1}(\cO((w'_i)_i\sigma')_{\tau'}) 
			=3d\dim \cF l_m  - \sum_{i=1}^d \dim \cO(w'_i).
			\\
			&\qquad= \codim \pi_{12}^{-1}(\cO((e)_i\sigma)_{\tau})  +  \codim \pi_{23}^{-1}(\cO((w'_i)_i\sigma')_{\tau'}).
		\end{split}
		\endeq}
	That is, the intersection {$\pi_{12}^{-1}(\cO((e)_i\sigma)_{\tau}) \cap  \pi_{23}^{-1}(\cO((w'_i)_i\sigma')_{\tau'})$} is transverse, and their composition is given by 
	{\begin{equation*}
			\cO((e)_i\sigma)_{\tau} \circ \cO((w'_i)_i\sigma')_{\tau'}=\{ (\cF_{\bullet},\cF'_{\bullet}) \ | \ (\cF_{\bullet},\cF'_{\bullet}) \in \cO((w'_{\sigma^{-1}(i)})_i\sigma\sigma')_{\tau} \}=\cO(ww')_{\tau}. 
	\end{equation*}} Since the fiber of {$\pi_{13}:\pi_{12}^{-1}(\cO((e)_i\sigma)_{\tau}) \cap  \pi_{23}^{-1}(\cO((w'_i)_i\sigma')_{\tau'})  \rightarrow \cO((e)_i\sigma)_{\tau} \circ \cO((w'_i)_i\sigma')_{\tau'}$} is a point, we obtain the desire result.
	The case $(w_i)_i \neq (e)_i$ and $(w'_i)_i=(e)_i$ is omitted since the argument  is similar.
	%the condition $(g_iw_{\tau^{-1}(i)}\Fstd_\bullet)_i\tau\sigma=(g'_i\Fstd_\bullet)_i\tau'$ implies that $g'_i=g_iw_{\tau^{-1}(i)}$ for all $i$. Thus, 
	%\begin{equation*}
	%(g'_iw'_{\tau'^{-1}(i)}\Fstd_\bullet)_i\tau'\sigma'=(g_iw_{\tau^{-1}(i)}w'_{\sigma^{-1}\tau^{-1}(i)}\Fstd_\bullet)_i\tau\sigma\sigma'=(g_i)_i\tau (w_iw'_{\sigma^{-1}(i)}\Fstd_\bullet)_i\sigma\sigma', 
	%\end{equation*} and the composition is given by 
	%\begin{align*}
	%Y_{w,\tau} \circ Y_{w',\tau'}&=\{ (\nu,\cF_{\bullet},\cF'_{\bullet}) \ | \ (\cF_{\bullet},\cF'_{\bullet})= ((g_i\Fstd_\bullet)_i\tau, (g_iw_{\tau^{-1}(i)}w'_{\sigma^{-1}\tau^{-1}(i)}\Fstd_\bullet)_i\tau\sigma\sigma') \}=\pi^{-1}(\cO(ww')_{\tau})=Y_{ww',\tau}. 
	%\end{align*} Since the fiber of $\pi_{13}:\pi_{12}^{-1}(Y_{w,\tau}) \cap \pi_{23}^{-1}(Y_{w',\tau'}) \rightarrow Y_{w,\tau} \circ Y_{w',\tau'}$ is a point, we obtain the second case.
\end{proof}

In the following example, we can see that when both $(w_i)_i $ and $(w'_i)_i$ are not the identity, the intersection  {$\pi_{12}^{-1}(\cO(w)_{\tau}) \cap \pi_{23}^{-1}(\cO(w')_{\tau'})$} is not guaranteed to be transverse, and hence \cite[Theorem 2.7.26]{CG97} does not apply.
In other words, it is not obvious how can one verify whether the ring $H(Z_{m \wr d})$ is semisimple.
\exa
Let $m=2$, $d=1$. Then, {$\dim\pi_{12}^{-1}(\cO(s_1))=3$. Similarly, $\dim\pi_{23}^{-1}(\cO(s_1))=3$.
	%$Y_{s_1}=\{(\nu,\cF_{\bullet},\cF'_{\bullet}) \ | \ (\cF_{\bullet}, \cF'_{\bullet}) \in \cO(s_1) \}$, 
	Since $\dim \cF l_2 \times \cF l_2 \times \cF l_2=3$, 
	\eq    
	\codim\pi_{12}^{-1}(\cO(s_1))+\codim\pi_{23}^{-1}(\cO(s_1))=0.
	\endeq} However, 
{$\pi_{12}^{-1}(\cO(s_1)) \cap \pi_{23}^{-1}(\cO(s_1))=\{(\cF_{\bullet},\cF'_{\bullet},\cF''_{\bullet}) \ | \ (\cF_{\bullet},\cF'_{\bullet},\cF''_{\bullet}) \in \GL_2 \cdot (\Fstd_{\bullet},s_1\Fstd_{\bullet},\Fstd_{\bullet})\}$, and $\dim \pi_{12}^{-1}(\cO(s_1)) \cap \pi_{23}^{-1}(\cO(s_1))=2$}. 
Thus, we conclude that 
{\eq
	\codim\pi_{12}^{-1}(Y_{s_1})+\codim\pi_{23}^{-1}(Y_{s_1})=0 \neq 1=\codim \pi_{12}^{-1}(Y_{s_1}) \cap \pi_{23}^{-1}(Y_{s_1}).    
	\endeq}
\red{That is, the intersection is not transverse. 
As a consequence, the multiplication $[Y_{s_1}] \ast [Y_{s_1}]$ cannot be computed using Lemma~\ref{convolution}.}
\endexa 

%===========================
\subsection{A Geometric Realization for $\QQ[\Sigma_m \wr \Sigma_d]$}
%===========================
Define the following sums of classes in $H(Z_{m \wr d})$:
\eq
\red{[\overline{Y_w} ]}\coloneqq \sum_{\tau \in \Sigma_d} [\overline{Y_{w,\tau}}],
\quad
[Y_w] \coloneqq \sum_{\tau \in \Sigma_d}[{Y_{w,\tau}}], \ \text{where} \ w \in \Sigma_m \wr \Sigma_d.
\endeq
%Define a subspace $A_{m \wr d} := \textup{Span}_\QQ\{[\overline{Y}_w^0 ]~|~w\in \Sigma_m \wr \Sigma_d\}\subseteq \HBMt(Z_{m \wr d})$.
Recall \red{$\zeta^{\tau}_i$} from \eqref{def:zetai}. It induces a map $\red{\zeta^{\tau}_{i,*}}: H(Z_m) \to H(Z_{m \wr d})$.
Then, for $g\in \Sigma_m, \ w = (w_i)_i e \in \Sigma_m \wr \Sigma_d$, we set
\eq
{[\Lambda^0_{g^{(i)}}] := \sum_{\tau \in \Sigma_d} \zeta^{\tau}_{i,*}([\Lambda^0_g])},
\quad
[\Lambda^0_{w}] := [\Lambda^0_{w_1^{(1)}}] * \dots * [\Lambda^0_{w_d^{(d)}}] \in H(Z_{m \wr d}).
\endeq
%By Proposition~\ref{prop:CGiso} 
%
%Then we define the embedding
%\begin{align}
%\begin{split} \label{embedding}
%\HBMt(Z^d_m) &\hookrightarrow \HBMt(Z_{m \wr d}), \\
% [\Lambda^0_{s_i^{(j)}}]=[\overline{Y_{s_i^{(j)}}}]+q[\overline{Y_{(e)_i}}] &\mapsto [\Omega^0_{s_i^{(j)}e}] \coloneqq \sum_{\tau \in \Sigma_d}    [\overline{Y_{s_i^{(j)}e,\tau}}]+q[\overline{Y_{(e)_ie,\tau}}].
%\end{split}
%\end{align}
Let $A_{m \wr d} \subseteq \HBMt(Z_{m \wr d}; \QQ)$ be the $\QQ$-subalgebra generated by $[\Lambda_{s_i^{(j)}}^0]$ and $\red{[\overline{Y_{t_k}}]}$ for $1\leq i \leq m-1, 1\leq j \leq d, 1\leq k \leq d-1$. 

{\thm  \label{mainthm}
There is an algebra isomorphism
\[
\QQ[\Sigma_{m} \wr \Sigma_{d}] \cong A_{m \wr d},
\quad
s_i^{(j)} \mapsto [\Lambda_{s_i^{(j)}}^0], \ t_k \mapsto \red{[\overline{Y_{t_k}}]}.
\]
Moreover, $A_{m \wr d}$ is spanned by $\red{[\overline{Y_{w}}]}$ for $w \in \red{\Sigma_{m} \wr \Sigma_{d}}$.
\endthm}

\begin{proof}
	By construction, $\{[\Lambda^0_{s_i^{(j)}}]\}$ generate a subalgebra of $A_{m\wr d}$ that is  isomorphic to $\QQ[\Sigma^d_m]$. 
	%By Proposition \ref{prop:HBM}, the algebras $\QQ[\Sigma_{m} \wr \Sigma_{d}]$ and $\HBMt(Z_{m \wr d})$ have the same dimension.
	Denote by $(w_\red{l})_\red{l}.s_j$ the place permutation of the $j$th and $(j+1)$th entries of $(w_\red{l})_\red{l} \in \Sigma_m^d.$ 
	We will check that the following relations hold, for $1\leq j \leq d-1, |j-i|\geq 2, 1\leq k \leq d-2$,  $(w_\red{l})_\red{l} \in \Sigma_ m^d$:
\begin{align}
&\textup{(quadratic relations) }\label{def:QR}
&\red{[\overline{Y_{t_j}}]^2} = [\Lambda^0_{e}],
\\
&\textup{(wreath relations) }\label{def:WR}
& \red{[\overline{Y_{t_j}}]}*[\Lambda^0_{(w_\red{l})_\red{l}}] = [\Lambda^0_{(w_\red{l})_\red{l}.s_j}]*\red{[\overline{Y_{t_j}}]},
\\
&\textup{(braid relations) } \label{def:BR3}
& \red{[\overline{Y_{t_k}}]}*\red{[\overline{Y_{t_{k+1}}}]}*\red{[\overline{Y_{t_k}}]} = \red{[\overline{Y_{t_{k+1}}}]}*\red{[\overline{Y_{t_k}}]}*\red{[\overline{Y_{t_{k+1}}}]},  
\\ 
\label{def:BR2}&&\red{[\overline{Y_{t_i}}]}*\red{[\overline{Y_{t_j}}]}=\red{[\overline{Y_{t_j}}]}*\red{[\overline{Y_{t_i}}]}.
\end{align} {Moreover, since $(e)_\red{l}t_k$ are minimal with respect to the Bruhat order in Definition \ref{def:newBruhat}, we have  $\red{[\overline{Y_{t_k}}]}=\sum_{\tau \in \Sigma_d} [\overline{Y_{(e)_\red{l}t_k,\tau}}]=\sum_{\tau \in \Sigma_d}[{Y_{(e)_\red{l}t_k,\tau}}]$. Thus, according to Lemma \ref{convolution}, the above relations are all equations in $H(Z_{m \wr d})$.
}

To verify \eqref{def:QR} and \eqref{def:WR}, it suffices to consider the ``rank one'' case. That is, consider $t = t_1 \in \Sigma_m \wr \Sigma_2$, we need to show that
\eq\label{tt}
 \red{[\overline{Y_{t}}]} \ast \red{[\overline{Y_{t}}]} = [\Lambda^0_{e}],
\endeq
which follows from combining Lemma~\ref{convolution} and Proposition \ref{prop:CGiso}(b), since
\eq
 \red{[\overline{Y_{t}}]} \ast \red{[\overline{Y_{t}}]} = \sum_{\tau, \tau' \in \Sigma_d} [Y_{t, \tau}] * [Y_{t, \tau'}] =\sum_{\tau \in \Sigma_d} [Y_{t^2, \tau}]  = [Y_e] = [\Lambda^0_e].
\endeq
For \eqref{def:WR}, it suffices prove the following, for any $s = s_i \in \Sigma_m$:
\eq\label{ts} 
\red{[\overline{Y_{t}}]} \ast 	[\Lambda^0_{s^{(1)}}] = [\Lambda^0_{s^{(2)}}] \ast \red{[\overline{Y_{t}}]}.
\endeq
By Proposition \ref{prop:CGiso}(b), $[\Lambda^0_{s}] =\red{[\overline{Y'_{s}}]}+q$ for some $q = n_{e,s}\in \QQ$, and hence for $j = 1$ or 2, $[\Lambda^0_{s^{(j)}}] = \red{[\overline{Y_{s^{(j)}}}]}+q$ for the same $q$.
Therefore, \eqref{ts} holds as long as $\red{[\overline{Y_{t}}]} \ast \red{[\overline{Y_{s^{(1)}}}]} = \red{[\overline{Y_{s^{(2)}}}]} \ast \red{[\overline{Y_{t}}]}$, or equivalently, 
\eq\label{eq:wrtocheck}
[{Y}_{t}] \ast [{Y}_{s^{(1)}}] + [{Y}_{t}]*[Y_e] = [{Y}_{s^{(2)}}]*[{Y}_{t}] + [Y_e] \ast [{Y}_{t}].
\endeq
It turns out that Lemma~\ref{convolution} applies to all four multiplications appearing in \eqref{eq:wrtocheck}, 
and hence we prove \eqref{ts}. 

%From the last subsection, we will assign the generators $t_k \in \Sigma_d$ with the fundamental class $[\Lambda^0_{t_k}]=[Y_{t_k}]=[Y_{(e,...,e,s_k)}]$ for all $1 \leq k \leq d-1$. Then we prove that these fundamental classes $[\Lambda^0_{s^{(j)}_i}]$, $[\Lambda^0_{t_k}]$ satisfy the derired relations.
%Note that the verification of relations (\ref{eq:stts}) is similar to the calculations for (\ref{ts}) in the above subsection, and calculation for $[\Lambda^0_{t_k}] \ast [\Lambda^0_{t_k}] =1$ is the same as (\ref{tt}). Thus, it
It  remains to prove the braid relations \eqref{def:BR3} -- \eqref{def:BR2} for the ``rank three case'' in $\Sigma_m \wr \Sigma_4$, i.e., 
\eq \label{braid}
	\red{[\overline{Y_{t_1}}]} \ast \red{[\overline{Y_{t_2}}]} \ast \red{[\overline{Y_{t_1}}]}=\red{[\overline{Y_{t_2}}]} \ast \red{[\overline{Y_{t_1}}]} \ast \red{[\overline{Y_{t_2}}]},
	\quad 
	\red{[\overline{Y_{t_1}}]} \ast \red{[\overline{Y_{t_3}}]} =\red{[\overline{Y_{t_3}}]} \ast \red{[\overline{Y_{t_1}}]}.
\endeq
Since $\red{[\overline{Y_{t_j}}]} = [Y_{t_j}]$ for all $j$, we can once again apply Lemma~\ref{convolution} to verify \eqref{braid}.
The proof of the isomorphism is complete.

Next, note that the set $\{[\overline{Y_w}] \}_{ w \in \Sigma_m \wr \Sigma_d }$ is linear independent. 
Denote by $A'$ the subspace of $H(Z_{m \wr d})$ spanned by $\{[\overline{Y_w}] \}_{ w \in \Sigma_m \wr \Sigma_d }$. Thanks to the first part of the theorem, the dimensions of $A_{m \wr d}$ and $A'$ coincide.
Thus, it suffices to show that $A_{m \wr d} \subseteq A'$. 

Thanks to the wreath relation \eqref{def:WR} and the fact that $\red{[\overline{Y_{t_j}}]} * \red{[\overline{Y_{t_k}}]} = \red{[\overline{Y_{t_jt_k}}]}$, any typical element of $A_{m\wr d}$ must be of the form $[\Lambda^0_{(w_i)_i}] * \red{[\overline{Y_{\sigma}}]}$ for some $(w_i)_i\sigma \in \Sigma_m \wr \Sigma_d$ and $\sigma \in 1\times \Sigma_d \subseteq \Sigma_m \wr \Sigma_d$. 
By Proposition ~\ref{prop:CGiso}, for each $1 \leq i \leq d$ we have $[\Lambda^0_{w_i}]=\sum_{y_i \leq_m w_i} n_{y_i,w_i}\red{[\overline{Y_{y_i}}]}$ for some $n_{y_i,w_i} \in \QQ$, and hence, by Lemma \ref{convolution},
\eq
\begin{split}
[\Lambda^0_{(w_i)_i}] * \red{[\overline{Y_{\sigma}}]} 
&=
\sum_{y_i \leq_m w_i} 
({\textstyle\prod_{i=1}^d n_{y_i,w_i}}) \biggl ( 
\sum_{\tau \in \Sigma_d}
\red{[\overline{Y_{(y_i)_i,\tau}}]}   \ast 
\sum_{\tau' \in \Sigma_d} \red{[\overline{Y_{\sigma,\tau'}}]} \biggr)
\\
&=\sum_{y'_i \leq_m y_i \leq_m w_i}
({\textstyle\prod_{i=1}^d n_{y_i,w_i}}) 
\sum_{\tau \in \Sigma_d} 
[{Y_{(y'_i)_i\sigma, \tau}}] 
= \sum_{y_i \leq_m w_i} ({\textstyle\prod_{i=1}^d n_{y_i,w_i}})\sum_{\tau \in \Sigma_d} \red{[\overline{Y_{(y_i)_i\sigma,\tau}}]},
\end{split}
\endeq 
which lies in $A'$. 
\end{proof}

%===========================
\section{Springer Correspondence}\label{sec:irrep}
%===========================
In this section, we obtain a geometric classification of the simples of  $\Sigma_{m} \wr \Sigma_d$ by establishing a new geometric classification theorem of simple modules over the subalgebra $A_{m \wr d}$ produced in Theorem \ref{mainthm}.
We remark that the counterpart in \cite{CG97} requires semisimplicity of  $H(Z)$, and hence does not apply to our case.
As a result, we establish the geometric interpretation of the Clifford theory for wreath products, for the first time.

%===========================
\subsection{A Lagrangian Construction}
%===========================
We first recall a useful result in geometric representation theory -- the classification theorem for complex irreducible representations over $\HBMt(Z)$ (see \cite[Theorem 3.5.7]{CG97}).  
Let us list  the required assumptions in \cite[\S 3.5]{CG97}:

\Def\label{Def:Lag}
Let $G$ be an algebraic group with a Borel subgroup $B$.
We call a morphism $\mu:\widetilde{\cN} \rightarrow \cN$ of $G$-variety a {\em Springer resolution} if the following conditions hold: 
\begin{enumerate} 
	\item $\widetilde{\cN}$ is smooth.
	\item $\cN$ has finitely many $G$-orbits.
	\item $\mu:\widetilde{\cN} \rightarrow \cN$ is $G$-equivariant and proper. 
	\item (dimension property) For each $x\in \cN$, all irreducible components of $\fB_x := \mu^{-1}(x)$ have the same dimension given by
	$\dim\fB_x =  \dim (G/B) - \frac{1}{2}\dim (G \cdot x)$.
\end{enumerate}
\endDef
Within this section, assume that $\mu$ is a Springer resolution. 
Let $Z := \tcN {\times}_\cN \tcN$ be the Steinberg variety.
It follows from  \cite[2.7.40]{CG97} that each $H(\fB_x)$ has a left  and a right $H(Z)$-module structure, denoted by $H(\fB_x)_L$ and $H(\fB_x)_R$, respectively.

For any finite-dimensional left module $V$ of $H(Z)$, denote by $V^\vee$ the right $H(Z)$-module with  underlying space $V^*:= \Hom_\QQ(V, \QQ)$ on which the action is given by
\eq
(f\cdot a) : v\mapsto f(a \cdot v), \quad a\in H(Z), \ f \in V^\vee, \ v\in V.
\endeq
For  $x \in \cN$, 
let $G(x)$ be the centralizer of $x$ in $G$,  
$G^0(x)$ be the identity component, and $C(x):=G(x)/G^0(x)$ be the component group. 
Note that ${H}(\fB_x)_R$ admits a left action over $G$ and over $C(x)$ that are compatible with the right $H(Z)$-action. %\cite[Lemma 3.5.3]{CG97}
By  \cite[Claim 3.5.5]{CG97},  ${H}(\fB_x)_L^\vee$ is also a left $C(x)$-module via
\eq
{(g\cdot {\check{v}})(v)} = {\check{v}} (g\inv \cdot v), \quad \textup{where}\  g\in C(x), \ {\check{v}}  \in H(\fB_x)_L^\vee, \ v\in H(\fB_x)_L.
\endeq 

For a group $\Gamma$, denote by Irr-$\CC[\Gamma]$ the set of all its irreducible complex representations, and by $\widehat{\Gamma} = \textup{Irr-}\CC[\Gamma]/\sim$ the set of iso classes of Irr-$\CC[\Gamma]$.
We identify $\widehat{\Gamma}$ with a fixed set of representatives of the \red{isomorphism classes}.
For $x \in \cN$, since $\cH(\fB_x)_L$ is a $(H(Z),C(x))$-bimodule (see \cite[Lemma 3.5.3]{CG97}), there is a bimodule decomposition
$\CC\otimes_\QQ H(\fB_x)_L = \bigoplus_{\psi \in C(x)^\wedge} \psi \otimes H(\fB_x)_\psi$,
where $H(\fB_x)_\psi$ is called the isotypic component, 
and
$\psi$ runs over all (\red{isomorphism classes} of) irreducibles which occur in $\CC\otimes H(\fB_x)$, i.e.,
\eq\label{def:Cxwedge}
C(x)^{\wedge} := \{\psi \in \widehat{C(x)} ~|~  \ [\CC \otimes_\QQ \HBMt(\fB_x; \QQ)_L: \psi]\neq 0\}.
\endeq
As already mentioned in the comments following \cite[Lemma 3.5.3]{CG97}, it is necessary to work with complex coefficients.

\begin{prop}[{\cite[Theorem 3.5.7]{CG97}}]\label{prop:GinzLagr}
Let $\mu$ be a Springer resolution in the sense of Definition~\ref{Def:Lag}. Suppose that
\enu[\textup{(C}1\textup{)}] 
\item  $H(Z)$ is semisimple, and
\item 
For any $x\in \cN$, %the isomorphism $\HBM(\fB_x)_R \cong \HBM(\fB_x)_L^\vee$ is compatible with the respective $C(x)$-actions
the isomorphism ${H}(\fB_x)_R \cong {H}(\fB_x)_L^\vee$ of right $H(Z)$-modules is compatible with their respective $C(x)$-actions.
\endenu
%For each $\psi \in C(x)^{\wedge}$, let $H(\fB_x)_{\psi}$ be the $\psi$-isotypic component of the $(C(x),H(Z))$-bimodule $\CC\otimes_\QQ H(\fB_x)_L$. 
Then, the complete set of irreducible $\HBMt(Z;\QQ)$-modules over $\CC$, up to isomorphism, is given by the set $\{{H}(\fB_x)_\psi ~|~ [x, \psi] \in \II\}$,
where
\eq\label{def:II}
\II := \{ G\textup{-conjugacy class }[x,\psi] \textup{ of }(x,\psi) ~|~ x\in \cN, \psi \in C(x)^\wedge\}.
\endeq
%Moreover, $\HBM(\fB_x)_\psi$ and $\HBM(\fB_{x'})_{\psi'}$ are isomorphic if and only if $(x, \psi)$ is $G$-conjugate to $(x', \psi')$.
\end{prop}
The proof of  \cite[Theorem 3.5.7]{CG97} relies on a closer understanding of 
the preimage $Z_\cO$ of the $G$-orbit $\cO \subseteq \cN $ under the projection $Z \to \cN$.
For one, there is a corresponding filtration $(Z_{\leq \cO}:= \bigsqcup_{\cO'}Z_{\cO'} )_{\cO \subseteq \cN}$ of $Z$ induced from the Bruhat order given by
$\cO \leq \cO' \Leftrightarrow \cO \subseteq \overline{\cO'}$.
It follows from \cite[Corollary 3.5.13]{CG97} that there is  a $H(Z)$-bimodule 
\eq
H_\cO := H(Z_{\leq \cO})/H(Z_{< \cO}),
\endeq
with basis formed by the fundamental classes of irreducible components of $Z_\cO$.
Secondly, it is crucial that the $C(x)$-orbits of the set of irreducible components of $\fB_x \times \fB_x$ are in bijection with the irreducible components of $Z_{\cO}$. 

We will see that this correspondence is no longer a bijection in our setup.
% becomes a $d!$-to-one correspondence in our setup. 
Nevertheless, we can still establish the following  variant of \cite[Theorem 3.5.7]{CG97} which leads to the Springer correspondence for wreath products.

\begin{thm}\label{thm:irreps}

Let $\mu$ be a Springer resolution in the sense of Definition~\ref{Def:Lag}, 
%and let  $\{X_i\}$ be the set of irreducible components of the corresponding Steinberg variety $Z$. 
and let $A \subseteq H(Z)$ be a subalgebra. 
Suppose that
\begin{enumerate}[\textup{(A}1\textup{)}]
    \item $A$ is semisimple, 
    \item For all $x\in \cN$, the isomorphism ${H}(\fB_x)_R \cong {H}(\fB_x)_L^\vee$ of right $A$-modules is compatible with their respective $C(x)$-actions, 
    \item For each $G$-orbit $\cO \subseteq \cN$, there is an isomorphism $A_\cO \cong H(\fB_x \times \fB_x)^{C(x)}$, where $A_\cO := (A \cap H_{\leq \cO})/(A \cap H_{<\cO})$,
    \item For all $x\in \cN$, $H(\fB_x)_L$ is an $(A, C(x))$-bimodule. 
Thus, the $\psi$-isotypic component $H(\fB_x)_{\psi}$ in the $(A,C(x))$-bimodule decomposition of $\CC\otimes_\QQ H(\fB_x)_L$ is well-defined.
\end{enumerate} 
Then, the complete set of irreducible $A$-modules over $\CC$,  up to isomorphism, is given by the set $\{{H}(\fB_x)_\psi ~|~ [x, \psi] \in \II\}$.
\end{thm}
\begin{proof}
Denote by $\{L_\lambda \}_{\lambda\in I}$  the complete set of irreducible $A$-modules over $\CC$,  up to isomorphism.
Fix $[x, \psi] \in \II$.
Thanks to (A4), the complex $A$-module $H(\fB_x)_\psi$ is well-defined,
and we may write
$
H(B_x)_\psi = \bigoplus_{\lambda \in I} L_\lambda^{\oplus m_\lambda}
$
\textup{for some}
$m_\lambda = m_\lambda(x,\psi) \in \ZZ_{\geq 0}$.
Hence,
\eq\label{eq:HLM}
\End_{\CC}( H(\fB_x)_\psi) 
= \bigoplus_{\lambda, \mu  \in I}
\Hom_{\CC}(L_\lambda, L_\mu)^{\oplus m_\lambda m_\mu}.
\endeq
Let $\cO$ be the $G$-orbit  containing $x$. Then,
\eq
\begin{split}
\End_{C(x)}( H(\fB_x)_L) &=  (H(\fB_x)_L \otimes H(\fB_x)_L^\vee)^{C(x)} \quad\textup{(by definition)}
\\
&\cong (H(\fB_x)_L \otimes H(\fB_x)_R)^{C(x)}\quad \textup{(thanks to (A2))}
\\
&= H(\fB_x\times \fB_x)^{C(x)}\quad \textup{(by definition)}
\\
&\cong A_\cO \quad \textup{(thanks to (A3))}.
\end{split}
\endeq
Then, we take the associated graded space
$\textup{gr}A = \bigoplus_{\cO \subseteq \cN} A_\cO$.
By (A1), on one hand the semisimplicity implies that $A \cong \textup{gr}A$, and hence
\eq\label{eq:AHO}
A \cong  \bigoplus_{\cO\subseteq \cN} A_\cO \cong \bigoplus_{\cO\subseteq \cN} \End_{C(x)}( H(\fB_x)_L).
\endeq
Therefore, by combining \eqref{eq:HLM} and \eqref{eq:AHO},
\eq
\begin{split}
\CC\otimes_\QQ A 
&\cong \CC\otimes_\QQ \bigoplus_{\cO\subseteq \cN} \End_{C(x)}( H(\fB_x)_L) 
= \bigoplus_{[x,\psi]\in \II} \End_{\CC}( H(\fB_x)_\psi) 
\\
&= \bigoplus_{\lambda, \mu  \in I}
\Hom_{\CC}(L_\lambda, L_\mu)^{\oplus (\sum_{[x,\psi] \in \II} m_\lambda(x,\psi) m_\mu(x,\psi))}.
\end{split}
\endeq
On the other hand, (A1) implies that $\CC\otimes_\QQ A $ is also semisimple, and hence $\CC\otimes_\QQ A = \bigoplus_{\lambda \in I}\End_{\CC}(L_\lambda)$ (by \cite[(3.5.22)]{CG97}),
or, $\delta_{\lambda,\mu} = \sum_{[x,\psi] \in \II} m_\lambda(x,\psi) m_\mu(x,\psi)$ for all $\lambda, \mu \in I$.
In words, each $[x, \psi] \in \II$ is associated with a unique $\lambda = \lambda(x,\psi)$ such that $m_{\lambda}(x,\psi)=1$ and $m_\mu({x,\psi}) = 0$ for all $\mu \neq \lambda$. 
That is, the complete set of non-isomorphic irreducibles are given by $\{{H}(\fB_x)_\psi ~|~ [x, \psi] \in \II\}$.
\end{proof}

For the rest of this section, we will verify that Theorem~\ref{thm:irreps} does apply, in our setup.
%===========================
\subsection{Springer Resolution}
%===========================
Let $G=G_{m \wr d}$, $B=B^d_m$, and recall the
wreath Springer resolution $\mu = \mu_{m \wr d}:\tcN_{m \wr d} \rightarrow \cN_m^d$  from \eqref{wreath-springer-resolution}.
%Now we can give set-theoretic descriptions for the objects appearing in Ginzburg's Lagrangian construction.
For  $x=(\redtwo{x}_i)_i \in \cN^d_m$, the set theoretic description for Springer fiber is given by
\eq \label{wreathSpringerfiber}
\fB_x =\{ (F^i_{\bullet})_iw \in \cF l_{m \wr d} \ | \  \redtwo{x}_iF^i_j \subseteq F^i_{j-1} \ \text{for all} \ i,j \}.
\endeq

{\prop \label{setup}
The wreath Springer resolution $\mu_{m \wr d}: \red{\tcN_{m \wr d}} \rightarrow \cN_m^d$ is a Springer resolution in the sense of Definition \ref{Def:Lag}.
\endprop}

\begin{proof}
By  Propositions \ref{prop:GLmwrd} and  \ref{compatabilitywrfl} (a), there are identifications $G_{m \wr d}\equiv \GL_m \wr  \Sigma_d$ and $\cF l_{m \wr d} = \cF l_{m} \wr \Sigma_d$, respectively. Thus, $G_{m \wr d}$ and $T^*\cF l_{m \wr d}$ are disjoint unions of $d!$-copies of $\GL^d_m$ and $T^*\cF l_m^d$, respectively, and hence (1) holds. 
Next, (2) follows from the orbit decomposition \eqref{orbitdecomp}. 

For (3), the $G_{m \wr d}$-equivariance follows from the fact that
\eq
\mu_{m \wr d} (\gamma \cdot( x, \cF )) =\mu_{m \wr d}( \gamma \cdot  x, \gamma \cdot \cF )=\gamma \cdot  x= \gamma \cdot \mu_{m \wr d} ( x,\cF ) ,
\endeq 
for all  $\gamma  \in G_{m \wr d}$, $(x, \cF) \in \tcN_{m \wr d}$.
For properness, it follows from the fact that $\mu_{m \wr d}$ is the restriction of the projection $\gl_{md} \times \cF l_{md} \rightarrow \gl_{md}$.
%Note that the morphism $\mu_{m\wr d}$ can be viewed as the following restriction 
%\eq
%\xymatrix{
%\gl_{md} \times \cF l_{md} \ar[r]^>>>>>{\pi_1} &\gl_{md} \\
%T^*\cF l_{md} \ar[r]^{\mu_{md}} \ar@{^{(}->}[u] &\cN_{md} \ar@{^{(}->}[u]  \\
%T^*\cF l_{m \wr d}  \ar[r]^{\mu_{m \wr d}} \ar@{^{(}->}[u]  &\cN_m^d \ar@{^{(}->}[u] 
%}
%\endeq and the properness of $\mu_{m \wr d}$ follows from the fact that the projection $\pi_1: \gl_{md} \times \cF l_{md} \rightarrow \gl_{md}$ is proper.

For the dimension property (4),  let $x = (\redtwo{x}_i)_i \in \cN_m^d$. 
Note that the fiber $\fB_x$ is a disjoint union of $d!$ copies of $\prod_{i=1}\fB_{\redtwo{x}_i}$. 
Next, the orbit $\cO := G_{m \wr d} \cdot x$ has dimension equals to the sum of dimensions of all $\GL_m \cdot \redtwo{x}_i$.
Finally, $\dim (G_{m \wr d}/B_m^d) = d\dim (\GL_m/B_m)$, and hence
\eq
\dim \fB_x = {\textstyle\sum_{i=1}^d} \dim \fB_{\redtwo{x}_i} = d\dim (\GL_m/B_m) - {\textstyle\frac{1}{2} \sum_{i=1}^d} (\GL_m \cdot \redtwo{x}_i) = \dim (G_{m \wr d}/B_m^d) - {\textstyle\frac{1}{2}} \dim \cO.
\endeq
\end{proof}

%===========================
\subsection{Verification of (A2)}
%===========================
Note that switching factors in $Z_{m \wr d}$ defines an involution $T'$, which induces an algebra anti-involution $T$ on $H(Z_{m \wr d})$ via $c \mapsto c^T$.
Consequently,  as left $H(Z_{m \wr d})$-modules, 
\eq\label{eq:LRT}
H(\fB_x)_L \cong H(\fB_x)_R^T, 
\endeq
where the latter 
is the module with underlying space $H(\fB_x)_R$ on which  $H(Z_{m \wr d})$ acts by 
$c \cdot v \coloneq v \cdot c^T$
for all $v \in H(\fB_x)$,  $c \in H(Z_{m \wr d})$.
\begin{lem}\label{lem:t}
	Under the isomorphism $A_{m \wr d} \cong \QQ[\Sigma_m \wr \Sigma_d]$, the anti-involution $T$ restricts to the anti-involution on $ \QQ[\Sigma_m \wr \Sigma_d]$ via $w \mapsto w^{-1}$ for all $w \in \Sigma_m \wr \Sigma_d$.
\end{lem}

\begin{proof}
Thanks to Theorem \ref{mainthm}, it suffices to consider the generators $[\Lambda^0_{(w_i)_i}]$ and $\red{[\overline{Y_{\sigma}}]}$, for some $w_i \in \Sigma_m$ and some  $\sigma = t_k \in \Sigma_m \wr \Sigma_d$.
It follows from {\cite[Lemma 3.6.11]{CG97}} that $[\Lambda^0_{(w_i)_i}]^{T}= [\Lambda^0_{(w_i\inv)_i}]$. 

%Since $[\overline{Y_{\sigma}}]=\sum_{\tau \in \Sigma_d} [Y_{\sigma,\tau}]$, i
It suffices to compute $[Y_{\sigma,\tau}]^{T}$ explicitly. 
By definition, 
$Y_{\sigma,\tau} = \{ (x,\cF_{\bullet},\cF'_{\bullet}) ~ | ~ (\cF_{\bullet},\cF'_{\bullet}) = (g_i)_i \cdot (\Fstd_{\bullet}\tau, \Fstd_{\bullet}\tau\sigma) \}$, and thus 
\eq
	Y_{\sigma,\tau}^{T'}=\{(x,\cF'_{\bullet},\cF_{\bullet}) \ | \ (\cF'_{\bullet},\cF_{\bullet})=(g_i)_i \cdot ( \Fstd_{\bullet}\tau\sigma, \Fstd_{\bullet}\tau)\}=Y_{\sigma^{-1},\tau\sigma}.
\endeq 
%Since $\sigma, \tau \in \Sigma_d$, we conclude that 
Therefore,
\eq
[\overline{Y_{\sigma}}]^T
=\sum_{\tau \in \Sigma_d} [Y_{\sigma,\tau}]^T
=\sum_{\tau \in \Sigma_d} [Y_{\sigma,\tau}^{T'}] 
=\sum_{\tau \in \Sigma_d} [Y_{\sigma^{-1},\tau\sigma}]
=[\overline{Y_{\sigma^{-1}}}].
\endeq
In other words, $T$ sends an arbitrary element $[\Lambda^0_{(w_i)_i}] \ast [\overline{Y_{\sigma}}]$ to $[\overline{Y_{\sigma^{-1}}}] \ast [\Lambda^0_{(w^{-1}_i)_i}]$, 
which corresponds to map  $w \mapsto \sigma^{-1}(w^{-1}_i)_i = w^{-1}$ under the desired isomorphism.
\end{proof}

\lem\label{lem:A2}
For all $x\in \cN_m^d$, the isomorphism ${H}(\fB_x)_R \cong {H}(\fB_x)_L^\vee$ of right $A_{m \wr d}$-modules is compatible with their respective $C(x)$-actions, 
\endlem
\proof
It suffices to show that there is an isomorphism $H(\fB_x)_L \cong (H(\fB_x)_L^\vee)^T$ that is compatible with the respective $C(x)$-actions, for all $x \in \cN_m^d$.
Since $H(\fB_x)$ is finite dimensional, 
by Lemma~\ref{lem:t},
there is a isomorphism $H(\fB_x)_L \to {(H(\fB_x)_L^\vee)^{T}}, v\mapsto v^*$ of {left} $\QQ[\Sigma_m \wr \Sigma_d]$-modules such that 
\eq 
(w \cdot v)^*({u}) = v^*(w\inv\cdot {u})
\endeq for all $v, {u}\in H(\fB_x)_L$ and  $w\in \Sigma_m \wr \Sigma_d$.
Therefore,
for each $g \in C(x), v \in H(\fB_x)_L$, $(g \cdot v)^*(u) = v^*(g\inv\cdot u)=(g\cdot v^*)(u)$ for all  $u \in H(\fB_x)$. 
\endproof
%===========================
\subsection{Verification of (A3)}
%===========================
Suppose that $Z$ is the Steinberg variety $Z$ corresponding to a Weyl group $W$ as in \cite[Theorem 3.5.7]{CG97}.
There is an algebra homorphism $\overline{f}: H_\cO \cong H(\fB_x \times \fB_x)^{C(x)}$ obtained using homomorphisms ${f}_1: H(Z_{\leq\cO}) \to H(Z_{\cO \cap U})$ and $f_2: H(Z_{\cO \cap U}) \to H(\fB_x \times \fB_x)$, where $U$ is a certain neighborhood of $x$. 
In fact, $\overline{f}$ is an isomorphism since the irreducible components of $Z_{\cO}$ are in bijection with the set of irreducible components of $\fB_x \times \fB_x$.

However, if $Z= Z_{m \wr d}$, such a bijection becomes a $d!$-to-one correspondence.
The idea of the lemma below is that, by taking an intersection with $A_{m \wr d}$, this multi-to-one correspondence still provides the desired isomorphism.

We begin with recalling the description of $f_1$ and $f_2$.
Let  $U \subseteq \cN^d_m$ be a neighborhood of $x$ such that $U \cap \cO = U \cap \overline{\cO}$.
%Given an orbit $\cO \subset \cN^d_m$ and $x \in \cO$. Clearly, we have $\cO=G_{m \wr d} \cdot x$. 
By \cite[Lemma 3.2.20]{CG97}, there is a transversal slice $S \subseteq \cN^d_m$ to $\cN^d_m$ at $x$ through the orbit $\cO$. 
Let $\widetilde{U}=\mu_{m \wr d}^{-1}(U)$, and let $\widetilde{S}=\mu_{m \wr d}^{-1}(S)$. 
Then, according to \cite[Definition 3.2.19]{CG97} and \cite[Corollary 3.2.21]{CG97},  there are isomorphisms 
\begin{equation} \label{transiso}
	(\cO \cap U) \times S \simeq U, \quad  (\cO \cap U) \times \widetilde{S} \simeq \widetilde{U}.
\end{equation}
Note that  the restriction from $\cN^d_m$ to $U$ induces an algebra homomorphism $H(Z) \rightarrow H(Z_{U})$, where $Z_{U}=\widetilde{U} \times_{U} \widetilde{U}$. %=Z \cap (\widetilde{U} \times_{U} \widetilde{U})$. 
%By replacing $Z$ with $Z_{\leq \cO}$, 
Thus, we obtain the algebra homomorphism $H(Z_{\leq \cO}) \rightarrow H(Z_{\leq \cO, U})$ where 
\begin{equation} \label{ZOU}
Z_{\leq \cO, U} =Z_{\leq \cO} \cap (\widetilde{U} \times_{U} \widetilde{U}) = \bigsqcup_{\cO' \leq \cO} Z_{\cO'} \cap  (\widetilde{U} \times_{U} \widetilde{U}).
\end{equation} 
It follows from $U \cap \cO = U \cap \overline{\cO}$ that  %$U \cap (\overline{\cO} \backslash \cO) = \emptyset$. 
%Thus, 
$U \cap \cO' = \emptyset$, and  hence $Z_{\cO'} \cap  (\widetilde{U} \times_{U} \widetilde{U})=\emptyset$ for all $\cO' < \cO$. Therefore, $Z_{\leq \cO, U} =Z_{\cO} \cap (\widetilde{U} \times_{U} \widetilde{U}) =Z_{\cO \cap U}$.
%\eq
%	Z_{\leq \cO, U} =Z_{\cO} \cap (\widetilde{U} \times_{U} \widetilde{U}) = \widetilde{\cO \cap U} \times_{\cO \cap U} \widetilde{\cO \cap U}=Z_{\cO \cap U},
%\endeq 
That is, we obtain a homomorphism $f_1:H(Z_{\leq \cO}) \rightarrow H(Z_{\cO\cap U})$.
(Note that it induces the map $\overline{f}_1 : H_{\cO} \rightarrow H(Z_{\cO\cap U})$ since it takes  $H(Z_{<\cO})$ to zero.)

Next, by \eqref{transiso}, we obtain the following isomorphism:
\eq
	Z_{U}=\widetilde{U} \times_{U} \widetilde{U} \cong ((\cO \cap U) \times \widetilde{S}) \times_{(\cO \cap U) \times S} ((\cO \cap U) \times \widetilde{S}) \cong \Delta_{\cO \cap U} \times (\widetilde{S} \times_{S} \widetilde{S}) = \Delta_{\cO \cap U} \times Z_{S},
\endeq
where $\Delta_{\cO \cap U} = (\cO \cap U) \times_{\cO \cap U} (\cO \cap U)$.
Thus, there is 
%an algebra isomorphism $H(Z_{U}) \cong H(\Delta_{\cO \cap U}) \otimes H(Z_{S})$ together with 
a homomorphsim $H(Z_{U}) \rightarrow H(Z_{S})$. 
%\begin{equation} \label{1sthomo}
%	f_1: H(Z_{\leq \cO}) \rightarrow H(Z_{\cO\cap U}). 
%\end{equation}
Let $l:=|\Sigma_d/C(x)|$.
Since there are $l$ components in the orbit $\cO$, we may write $S \cap \cO=\bigsqcup_{i=1}^{l} \{x^{(i)}\}$ with $x^{(1)}=x$. 
By restriction to $Z_{\cO \cap U}$, 
we obtain %the isomorphism $H(Z_{\cO \cap U}) \cong H(\Delta_{\cO \cap U}) \otimes H(Z_{S \cap \cO})$, and thus 
another homomorphism $f_2:H(Z_{\cO \cap U}) \rightarrow \bigoplus_{i=1}^l H(Z_{x^{(i)}})$.
%\begin{equation} \label{2ndhomo}
%	f_2:H(Z_{\cO \cap U}) \rightarrow H(Z_{x^{(1)}}) \bigoplus ... \bigoplus H(Z_{x^{(l)}}).
%\end{equation}
In other words, we obtain a homomorphism $f := f_2 \circ f_1 : H(Z_{\leq\cO}) \to \bigoplus_{i=1}^l H(Z_{x^{(i)}})$.
%Combining (\ref{1sthomo}) and (\ref{2ndhomo}), we obtain the following homomorphism
%\begin{equation*} 
%	f: H(Z_{\leq \cO})  \xrightarrow{f_1} H(Z_{U \cap \cO}) \xrightarrow{f_2} H(Z_{x^1}) \bigoplus ... \bigoplus H(Z_{x^l}).
%\end{equation*} We observe that $f_1$ maps $H(Z_{< \cO})$ to zero. So, it induces the following map
%\eq
%	\Bar{f}:H_{\cO}  \xrightarrow{\Bar{f_1}} H(Z_{U \cap \cO}) \xrightarrow{f_2} \bigoplus_{i=1}^l H(Z_{x^i}).
%\endeq

\begin{lemma} \label{lem:A3}
Suppose that $A = A_{m \wr d}$. Then,
for each $G$-orbit $\cO \subseteq \cN_m^d$, there is an isomorphism $A_\cO \cong H(\fB_x \times \fB_x)^{C(x)}$.
\end{lemma}
\begin{proof}
The lemma will follow from a detailed description of the restriction of $f$ onto $A \cap \cH(Z_{\leq \cO})$.
Recall that $H_{\cO}$ has a basis indexed by the classes of the irreducible components in $Z_{\cO}$, i.e.,
\eq
\{ [\overline{Y_{w,\tau}}] + H(Z_{< \cO}) ~|~ Y_{w,\tau} \cap Z_{\cO} \neq \emptyset\},
\endeq
thanks to Proposition \ref{basis}. 
We claim that if $Y_{w,\tau} \cap Z_{\cO} \neq \emptyset$ for some $\tau \in \Sigma_d$, then $Y_{w,\tau'} \cap Z_{\cO} \neq \emptyset$ for all $\tau' \in \Sigma_d$.
Note that
\begin{equation}
Z_{\cO} \cap Y_{w,\tau} = \{(x,F_{\bullet},F'_{\bullet}) \in Z_{m \wr d} \ | \ x \in \cO, \ (F_{\bullet},F'_{\bullet}) \in \cO(w)_{\tau} \}.    
\end{equation} 
Let $w=(w_i)_i\sigma$. Given $(F_{\bullet},F'_{\bullet}) \in \cO(w)_{\tau}$, then
$(F_{\bullet},F'_{\bullet})=(g_i)_i \cdot (\Fstd_{\bullet}\tau,(w_{\tau^{-1}(i)})_{\red{i}}\Fstd_{\bullet}\tau\sigma)$
for some $g_i \in \GL_m$. 
Since $Y_{w,\tau} \cap Z_{\cO} \neq \emptyset$, we pick $x=(x_i)_i \in \cO$ such that $(x,F_{\bullet},F'_{\bullet}) \in Y_{w,\tau} \cap Z_{\cO}$.
Suppose that $\tau' \in \Sigma_d\setminus \{\tau\}$. Consider the following pairs of flags:
\begin{equation}
(F''_{\bullet},F'''_{\bullet})=(g_{\tau\tau'^{-1}(i)})_i \cdot (\Fstd_{\bullet}\tau',(w_{\tau'^{-1}(i)})_{\red{i}}\Fstd_{\bullet}\tau'\sigma) {\in} \cO(w)_{\tau'}.
\end{equation} 
Then, $y:=(y_i)_i \coloneqq (x_{\tau\tau'^{-1}(i)})_i$ lies in $\cO$, and thus $(y,F''_{\bullet},F'''_{\bullet}) \in Z_{\cO} \cap Y_{w,\tau'}$. The claim is proved.

As a result, $H(Z_{\cO \cap U})$ has a basis $[\overline{Y_{w,\tau}|_{U}}]$  for some $w \in \Sigma_m \wr \Sigma_d$ that depends on $\cO$ and for all $\tau \in \Sigma_d$. 
Next, we choose $e=\xi_1,...,\xi_l$ to be the coset representatives of $\Sigma_d/C(x)$ such that $x^{(i)}=\xi_i \cdot x$. 
Then, $\cO$ decomposes into connected components 
$\cO=G_{m \wr d} \cdot x = \bigsqcup_{i=1}^l \GL^d_m \cdot x^{(i)}= \bigsqcup_{i=1}^l \cO_i$. 
Fix $\eta \in \Sigma_d$ such that $Y_{w,\eta} \cap Z_{\cO_1} \neq \emptyset$. 
Then, $Y_{w,\xi_i\eta} \cap Z_{\cO_i} \neq \emptyset$ for all $i$.
Thus, $f_2$ sends each $\overline{[Y_{w,\xi_i\eta}|_{U}]}$ to $(0,...,[\overline{\cO(w)_{\xi_i\eta}}],...,0)$, which is supported at the $i$th component.

Write $Z_x:= \fB_x \times \fB_x$ for short.
For each $1 \leq i \leq l$, the natural isomorphism $H(Z_x) \xrightarrow{\cong} H(Z_{x^{(i)}})$, $[X] \mapsto [\xi_i(X)]$ tells us that  $\cO(w)_{\xi_i\eta}=\xi_i(\cO(w)_{\eta})$. 
In other words, there is a copy of $H(Z_x)$ (denoted by $\Delta_{H(Z_x)}$) inside $\bigoplus_{i=1}^l H(Z_{x^{(i)}})$ via 
$[X] \mapsto ([X],[\xi_2(X)],...,[\xi_l(X)])$.

Finally, we consider the $C(x^{(i)})$-orbit on the components of $Z_{x^{(i)}}$ for each $i$. 
Then, there is an invariant subspace $H(Z_{x^{(i)}})^{C(x^{(i)})}$ for each $i$. 
Furthermore, since all the component groups are isomorphic $C(x) \cong C(x^{(2)}) \cong ... \cong C(x^{(l)})$, 
we also obtain an invariant subspace of the diagonal. 
In other words, the following diagram commutes:
\begin{equation}
	\xymatrix{
		\Delta_{H(Z_x)^{C(x)}} \ar@{^{(}->}[r] \ar@{^{(}->}[d] &  \Delta_{H(Z_x)} \ar@{^{(}->}[d] \\
		\bigoplus_{i=1}^l H(Z_{x^{(i)}})^{C(x^{(i)})} \ar@{^{(}->}[r]  & \bigoplus_{i=1}^l H(Z_{x^{(i)}}).
	}
\end{equation}
By Theorem \ref{mainthm}, %$A$ is spanned by the elements $\sum_{\tau \in \Sigma_d} [\overline{Y_{w,\tau}}]$. Thus, by the above claim, when we intersect with $A$, 
$A_{\cO} \coloneqq A \cap H(Z_{\leq \cO})/A \cap H(Z_{<\cO})$ is spanned by elements of the form $\sum_{\tau \in \Sigma_d} [\overline{Y_{w,\tau}}]+A \cap H(Z_{<\cO})$. 
Restricting ${f}$ to $A_{\cO}$, 
%and decomposing $\Sigma_d$ into cosets $\bigcup_{i=1}^l \xi_iC(x)$, 
we find its image lies in $\Delta_{H(Z_x)^{C(x)}}$.
Precisely, the following describes the desired isomorphism:
\eq
\begin{array}{crcl}
	\Bar{f}:
	& A \cap H(Z_{\leq \cO})/A \cap H(Z_{<\cO}) %\rightarrow A \cap H(Z_{\cO \cap U}) 
	&\to
	&\Delta_{H(Z_x)^{C(x)}} \cong H(Z_x)^{C(x)} = H(\fB_x \times \fB_x)^{C(x)} 
	\\
	&\sum_{\tau \in \Sigma_d} [\overline{Y_{w,\tau}}]+A \cap H(Z_{<\cO}) %\mapsto \sum_{\tau \in \Sigma_d} [\overline{Y_{w,\tau}|_U}] 
	&\mapsto 
	&\sum_{\phi \in C(x)} ([\overline{\cO(w)_{\phi\eta}}], \ldots, [\overline{\cO(w)_{\phi\xi_l\eta}}]).
	\end{array}
\endeq

%Thus, the irreducible components of  $Z_{\cO}$ decompose into the irreducible components of $Z_{\cO} \cap Y^{\tau}_w$ where $w \in \Sigma_m \wr \Sigma_d$ and $\tau \in \Sigma_d$. 

%Since $B$ is spanned by the fundamental classes $[Y_w]$, by intersecting with $B$, $B \cap H(Z_{\cO})$ is spanned by 
%\begin{equation*}
%\sum_{\tau \in \Sigma_d} [\overline{\text{an irreducible component of} \ Z_{\cO} \cap Y^{\tau}_{\sigma}}]?
%\end{equation*} Thus, we obtain an algebra isomorphism $ B \cap H(Z_{\cO}) \cong H(\fB_x \times \fB_x)^{C(x)}$?

\end{proof}

\begin{exa}
	Let $m=d=2$. 
	There are three $G_{2\wr 2}$-orbits:
	\begin{equation*}
		\cN^2_2=\cO[(1,1),(1,1)] \bigsqcup \cO[(2),(1,1)] \bigsqcup \cO[(2),(2)],
	\end{equation*} 
	We will present details for the above Lemma for the orbits of the following elements:
	\begin{equation*}
		x=(\begin{psmallmatrix}
			0 & 1 \\ 0 & 0
		\end{psmallmatrix}, \begin{psmallmatrix}
			0 & 1 \\ 0 & 0
		\end{psmallmatrix}) \in \cO[(2),(2)], 
		\quad 
		y=(\begin{psmallmatrix}
			0 & 1 \\ 0 & 0
		\end{psmallmatrix}, \begin{psmallmatrix}
			0 & 0 \\ 0 & 0
		\end{psmallmatrix}) \in \cO[(2),(1,1)]. 
	\end{equation*}
	For the orbit $\cO[(2),(2)]$,  the component group is $C(x) \cong \Sigma_2$.
	The Springer fiber is $\fB_x=\{\Fstd_{\bullet} \}^2 \times \Sigma_2$. 
	Since $Y_{(e,e)\sigma,\tau} \cap Z_{\cO[(2),(2)]} \neq \emptyset$ for all $\sigma, \tau \in \Sigma_2$, 
	the restriction of ${f}$ in Lemma \ref{lem:A3} is
\eq
		H(Z_{U \cap \cO[(2),(2)]}) \rightarrow H(\fB_x \times \fB_x) \quad
		[\overline{Y_{(e,e)\sigma,\tau}}]=[Y_{(e,e)\sigma,\tau}] \mapsto [\{\Fstd_{\bullet} \}^2\tau \times \{\Fstd_{\bullet} \}^2\tau\sigma]. 
\endeq
	As a result, $\Bar{f}$ sends the sum $\sum_{\tau \in \Sigma_2}[\overline{Y_{(e,e)\sigma,\tau}}]$ to the $C(x)$-orbit of $[\{\Fstd_{\bullet} \}^2 \times \{\Fstd_{\bullet} \}^2\sigma]$.
	
	Next, for the orbit $\cO[(2),(1,1)]$, we have $C(y) \cong \Sigma_1 \times \Sigma_1$ and $\fB_{y}=\{\Fstd_{\bullet} \} \times \cF l_2 \times \Sigma_2$. 
	In this case, the orbit has two components $\cO[(2),(1,1)]=\cO((2),(1,1)) \bigsqcup \cO((1,1),(2))$. Let  $y'=(\begin{psmallmatrix}
		0 & 0 \\ 0 & 0
	\end{psmallmatrix}, \begin{psmallmatrix}
		0 & 1 \\ 0 & 0
	\end{psmallmatrix}) \in \cO((1,1),(2))$.
	
	It can be checked that both $Y_{(s_1,e) \sigma,\tau}$ and $Y_{(e,s_1) \sigma,\tau}$ interact with $Z_{\cO[(2),(2)]}$ non-trivially for all $\sigma, \tau \in \Sigma_2$. Thus, the restriction of ${f}$ is given by 
	\begin{align*}
		H(Z_{U \cap \cO[(2),(1,1)]}) &\rightarrow H(\fB_y \times \fB_y)
		\oplus  H(\fB_{y'} \times \fB_{y'}) 
		\\
		[\overline{Y_{(e,s_1)\sigma,e}}] &\mapsto ([\{\Fstd_{\bullet}\} \times \cF l_2 \times \{\Fstd_{\bullet} \} \times \cF l_2 \sigma],0), 
		\\
		[\overline{Y_{(e,s_1)\sigma,t}}] &\mapsto (0,
		[ \cF l_2  \times \{\Fstd_{\bullet}\} t \times \cF l_2  \times \{\Fstd_{\bullet}\} t\sigma]),
		\\
		[\overline{Y_{(s_1,e)\sigma,e}}] &\mapsto (0,[ \cF l_2  \times \{\Fstd_{\bullet}\}  \times \cF l_2  \times \{\Fstd_{\bullet}\} \sigma]) ,
		\\
		[\overline{Y_{(s_1,e)\sigma,t}}] &\mapsto ([\{\Fstd_{\bullet}\} \times \cF l_2 t \times \{\Fstd_{\bullet} \} \times \cF l_2 t\sigma],0).
	\end{align*} 
	As a result, $\Bar{f}$ sends the sum $\sum_{\tau \in \Sigma_2}[\overline{Y_{(e,s_1)\sigma,\tau}}]$ to the $C(x)$-orbit of the diagonal $([\{\Fstd_{\bullet}\} \times \cF l_2 \times \{\Fstd_{\bullet} \} \times \cF l_2 \sigma],[ \cF l_2  \times \{\Fstd_{\bullet}\} t \times \cF l_2  \times \{\Fstd_{\bullet}\} t\sigma])$.
 Similarly, the other sum $\sum_{\tau \in \Sigma_2}[\overline{Y_{(s_1,e)\sigma,\tau}}]$ is mapped to the $C(x)$-orbit of the diagonal $([\{\Fstd_{\bullet}\} \times \cF l_2 t \times \{\Fstd_{\bullet} \} \times \cF l_2 t\sigma],[ \cF l_2  \times \{\Fstd_{\bullet}\}  \times \cF l_2  \times \{\Fstd_{\bullet}\} \sigma])$.
\end{exa}

%===========================
\subsection{Verification of (A4)}
%===========================
In \red{Chriss--Ginzburg \cite{CG97}}, the proof of that $H(\fB_x)$ is a $(H(Z), C(x))$-bimodule is based on a fact that $H(Z)$ is invariant under the $G$-action $z \mapsto g(z)$ that arises from the automorphism on $Z$ (see \cite[Lemma 3.5.2]{CG97}).
However, it can be seen in the example below that the invariance fails in our setup:
\exa \label{ex:A4}
Let $m=d=2$,  $w =(s_1,e) \in \Sigma_2 \wr \Sigma_2$, $\tau = t \in \Sigma_2$, and  $g = (\id,\id,t) \in G_{2 \wr 2}$.
Then, $[\red{\overline{Y_{w,t}}}]=[Y_{w,t}]+[Y_{e,t}]$, and hence
\eq
g([\red{\overline{Y_{w,t}}}])
=g([Y_{w,t}]) + g([Y_{e,t}])
=[Y_{w,e}]
+[Y_{e,e}]
=[\red{\overline{Y_{w,e}}}] \neq [\red{\overline{Y_{w,t}}}].
\endeq
\endexa 
This is another evidence that one should be considering summing over the classes $[Y_{w,\tau}]$ for $\tau \in \Sigma_d$ in order to obtain a $G$-invariant subalgebra of $H(Z)$.

Now, fix $x\in \cN_m^d$. We begin with detailed descriptions of actions
\eq
H(Z_{m \wr d})\times H(\fB_x) \to H(\fB_x), \ (z, c) \mapsto z * c,
\quad
H(\fB_x)\times G_{m \wr d} \to \bigoplus_{g \in G_{m \wr d}} H(\fB_{g \cdot x}), \ (c, g) \mapsto c \cdot g.
\endeq
Firstly, the left $H(Z_{m \wr d})$-action on $H(\fB_x)_L$ arises from the composition $Z \circ \fB_x = \fB_x$.
Secondly, note that $g=(g_i)_iw \in G_{m \wr d}$ gives a map
\eq
g:\fB_x \rightarrow \fB_{g \cdot x}, 
\quad 
(x,(F^i_{\bullet})\sigma) \mapsto (g \cdot x,(g_i)_iw(F^i_{\bullet})\sigma),
\endeq 
via the conjugation $g \cdot x$ (see \eqref{nilact}) on $\cN_m^d$,
and the natural action $(g_i)_iw(F^i_{\bullet})\sigma$ on $\cF l_{m \wr d}$ (see Proposition~\ref{compatabilitywrfl}(c)). 
Thus, $g$ induces a morphism $H(\fB_x)  \to H(\fB_{g \cdot x})$ and thus a right $C(x)$-action on $H(\fB_x)$.  

\begin{lemma} \label{lem:A4}
For all  $x\in \cN_m^d$, $H(\fB_x)_L$ is an $(A_{m \wr d}, C(x))$-bimodule.
\end{lemma}

\begin{proof}
It follows from construction that $(z * c) \cdot g = g(z) * (c \cdot g)$ for all $z \in H(Z_{m \wr d})$, $c\in H(\fB_x)$, and $g\in G_{m \wr d}$.
The lemma follows as long as $g(a)=a$ for all $g \in C(x)$ and $a \in A_{m \wr d}$. 
Thanks to Theorem \ref{mainthm},  $A_{m \wr d}$ is spanned by $[\overline{Y}_w] =\sum_{\tau \in \Sigma_d} [\overline{Y}_{w,\tau}]$.
Hence, each $g=(g_i)_i\eta \in G_{m \wr d}$ acts by
\eq
\begin{split}
g([\red{\overline{Y_{w}}}])
&=\sum_{\tau \in \Sigma_d} g([\red{\overline{Y_{w,\tau}}}]) 
=\sum_{\tau \in \Sigma_d, w' \leq_{m \wr d} w} g([Y_{w',\tau}]) 
\\
&=\sum_{\tau \in \Sigma_d, w' \leq_{m \wr d} w} [Y_{w',\eta\tau}] 
= \sum_{\tau' \in \Sigma_d, w' \leq_{m \wr d} w} [Y_{w',\tau'}]=[\red{\overline{Y_{w}}}]. 
\end{split}
\endeq
%When $g \in G_{m \wr d}(x)$, it maps $H(\fB_x)$ to itself which gives an action of $G_{m \wr d}(x)$ on $H(\fB_x)$. Moreover, the identity component $G^0_{m \wr d}(x)$ acts trivially on homology so that it factors through the action of $C(x)=G_{m \wr d}(x)/G^0_{m \wr d}(x)$. 
The proof is complete.
\end{proof}

%===========================
\subsection{Springer Correspondence for Wreath Products}
%===========================
We are now in a position to prove the classification theorem for simple modules over $\CC[\Sigma_m \wr \Sigma_d]$.
\begin{theorem}[Springer correspondence for $\Sigma_m \wr \Sigma_d$] \label{irredrep}
Let $H(\fB_x)_{\psi}$ and $C(x)^\wedge$ be the $\psi$-isotypic component and the subset of $\widehat{C(x)}$ consisting of simple modules which occur in $\HBMt(\fB_x;\CC)$, respectively. 
Then,
$\widehat{\Sigma_m \wr \Sigma_d} = \{ H(\fB_x)_{\psi} ~|~ [x,\psi] \in I^S_{m \wr d}\} $, where
\eq\label{def:IS}
I^S_{m \wr d}:= \{ G_{m \wr d}\textup{-conjugacy class } [x,\psi] ~|~ x \in \cN^d_m,\ \psi \in C(x)^{\wedge}  \}.
\endeq
\end{theorem}
\begin{proof}
%We are to apply Theorem~\ref{thm:irreps}.
We have checked in Proposition~\ref{setup} that $\mu_{m \wr d}$ is a Springer resolution.
It suffices to check that the subalgebra $A_{m \wr d}$ satisfy (A1) -- (A4). 
For (A1), it follows from Theorem \ref{mainthm} since group algebras of finite groups are semisimple. 
For (A2), it follows from Lemma~\ref{lem:A2}.
For (A3), it follows from Lemma~\ref{lem:A3}.
For (A4), it follows from Lemma~\ref{lem:A4}.
Thus, Theorem~\ref{thm:irreps} applies, and we are done.
\end{proof}

\red{In the next section, we will provide some examples for this result, including Springer correspondence for type $B/C/D$ Weyl groups.}

%===========================
\section{Clifford Theory}\label{sec:CT}
%===========================
%===========================
\subsection{Clifford Theory for Finite Groups}
%===========================
We take a detour to review the well-known classification of irreducible modules over the group algebra $\CC [\Sigma_m \wr \Sigma_d]$ using Clifford theory,
following \cite[\S2.6.1]{CST14} (see also \cite{CR81}). 

Recall that  $\Pi_m$ is the set of partitions $\nu$ of $m\in \ZZ_{\geq 0}$. Let $\Pi = \bigcup_{m \geq 0} \Pi_m$.
Let $|\nu| \in \ZZ_{\geq 0}$ be the number partitioned by $\nu \in \Pi$.
Then, $\redtwo{\widehat{\Sigma_{m}}}  = \{S_m^\nu ~|~ \nu \in \Pi_m\}$, where  $S_m^\nu$ is the corresponding (irreducible) Specht module over $\Sigma_m$.

For any map $\bblambda : \Pi_m \to \Pi$ (equivalently, a multi-partition $\bblambda = (\bblambda(1), \dots, \bblambda(\red{\#}\Pi_m))$ of $\red{\#}\Pi_m$ components), define a map $|\bblambda|: \Pi_m \to \ZZ_{\geq 0}$ by setting $|\bblambda|(\nu) := |\bblambda(\nu)|$ for all $\nu \vdash m$.
\red{Each map $\gamma:\Pi_m \to \ZZ_{\geq 0}$ with $\sum_{\nu \vdash m} \gamma(\nu)= d$} is associated with a Young subgroup $\Sigma_{\gamma} := \prod_{\nu \vdash m}\Sigma_{\gamma(\nu)} \subseteq \Sigma_d$. Set
\eq\label{def:IC}
I^C_{m\wr d} := \{ \bblambda : \Pi_m \to \Pi ~|~ \textstyle{\sum_{\nu \vdash m}}|\bblambda (\nu)| = d\}.
%\quad
%J_{m\wr d} := \{\gamma: \Pi_m \to \ZZ_{\geq 0} ~|~ \textstyle{\sum_{\nu \vdash m}}\gamma (\nu) = d \}.
\endeq
Let $\bblambda \in I^C_{m\wr d}$, and let $\gamma = |\bblambda|$.
Define the following irreducible modules:
\eq \label{irreducibleSpecht}
S^\bblambda := {\textstyle\bigotimes_{\nu \vdash m}} S_{|\bblambda(\nu)|}^{\bblambda(\nu)} \in 
\textup{Irr-}\CC[\Sigma_\gamma],
%\redtwo{\widehat{\Sigma_{\gamma}}} ,
\quad
S^{\gamma} := {\textstyle\bigotimes_{\nu \vdash m}} (S_m^\nu)^{\otimes \gamma(\nu)} \in 
\textup{Irr-}\CC[\Sigma_m^d].
%\redtwo{\widehat{\Sigma^d_{m}}} .
\endeq
Denote the extension module of $S^\gamma$ over $\CC[\Sigma_m \wr \Sigma_\gamma]$ by $\widetilde{S^\gamma}$, and denote the inflation module over $\CC[\Sigma_m \wr \Sigma_\gamma]$ of $S^\bblambda$ to inertia group $\Sigma_m \wr \Sigma_\gamma$  by $\Infl S^\bblambda$.
That is, $\widetilde{S^\gamma} = S^\gamma$ and $\Infl S^\bblambda = S^\bblambda$ as vector spaces, while the actions of additional group elements are given by, respectively, place permutations on tensor factors, and the trivial action.

\begin{prop}[Clifford theory for $\Sigma_m \wr \Sigma_d$]
The complete list of \red{isomorphism classes} of irreducible representations of $\CC[\Sigma_m \wr \Sigma_d]$ is given by
\[
\widehat{\Sigma_m \wr \Sigma_d}
=
\left\{  L^\bblambda ~\middle|~ \bblambda \in I^C_{m \wr d} \right\},
\quad
\textup{where }\gamma = |\bblambda|, \ 
L^\bblambda := \Ind_{\Sigma_m \wr \Sigma_{\gamma}}^{\Sigma_m \wr \Sigma_d}(\widetilde{S^{\gamma}} \otimes \Infl S^\bblambda).
\]
\end{prop}

\subsection{Characterization of the Index Set}

For completeness, in this subsection, we give an explicit characterization for the index set $I^S_{m \wr d}$ from the Springer correspondence (Theorem \ref{irredrep}), and then provide an identification with the index set $I^C_{m \wr d}$ in the algebraic approach from Clifford theory.

From \eqref{def:IS}, each irreducible is indexed by a $G_{m \wr d}$-conjugacy  class  $[x,\psi]$ for some $x =(x_i)_i \in \cN_m^d$ and $\psi \in C(x)^{\wedge}$.
%Note that the action of $G_{m \wr d}$ on $\cN^d_m$ is via conjugation and the orbit decomposition is given by \eqref{orbitdecomp}.
Each $x_i \in \cN_m$ has Jordan type $\lambda_i \vdash m$.
Equivalently, $x$ is associated with a map $\gamma(x)$ given by
\eq\label{def:gammax}
\gamma(x): \Pi_m \to \ZZ_{\geq 0},
\quad
\nu \mapsto  \red{\#}\{ i \in \{1, \dots, d\} ~|~ \lambda_i = \nu\}.
\endeq
%Thus, it remains to know the structure of the (finite) component group $C(x)^{\wedge}$ for each $x \in \cN$ and the action of $G_{m \wr d}$ on each irreducible representation $\psi \in C(x)^{\wedge}$. 

{\prop \label{structureCx}
Let $x \in \cN^d_m$. Recall $\gamma(x)$ from \eqref{def:gammax}.
%For $x \in \cN$ which is in the nilpotent conjugacy class $[(\lambda_1,...,\lambda_d)]$ where $\lambda_i \dashv m$ for all $1 \leq i \leq d$. Let $p(m)$ denote the number of total partitions of $m$, and $a_1,...,a_{p(m)}$ be the list of all partitions of $m$. We define $k$ to be the number of distinct partition in $(\lambda_1,...,\lambda_d)$ such that $\{\lambda_1,...,\lambda_d\}=\{a_{i_1},...,a_{i_k}\}$. Finally, we define $r_{i_j}$ to be number of times of $a_{i_j}$ shows up in $(\lambda_1,...,\lambda_d)$ such that $d=r_{i_1}+...+r_{i_k}$. Then we have 
Then, the component group $C(x)$ is isomorphic to the Young subgroup $\Sigma_{\gamma(x)}$.
%\begin{equation}
%	C(x) \cong \Sigma_{r_{i_1}} \times ... \times \Sigma_{r_{i_k}}.
%\end{equation} 
\endprop}

\begin{proof}
%Fix an element $x=(x_i)_i \in \cN^d_m$. There are $\gamma(\nu)$ entries $x_i$'s in $x$ that are of Jordan type $\nu$, for each $\nu$.
%Enumerate the set $\{\nu \vdash m ~|~ \gamma(\nu) \neq 0\}$ by $\{\nu_1,..., \nu_k\}$, and hence
%Since $[(\lambda_1,...,\lambda_d)]$ is an equivalence class under the $\Sigma_d$-action, by the assumption, we can rearrange the partitions such that 
%\begin{equation} \label{rearrange}
%[(\lambda_1,...,\lambda_d)]=[(\underbrace{a_{i_1},...,a_{i_1}}_{r_1 \ \mrm{terms}},\underbrace{a_{i_2},...,a_{i_2}}_{r_2 \ \mrm{terms}},...,\underbrace{a_{i_k},...,a_{i_k}}_{r_k \ \mrm{terms}})].
%\end{equation}Thus, for $x \in [(\lambda_1,...,\lambda_d)]$, we can assume for simplicity that $x=(n_i)_i$ is of nilpotent type $[(a_{i_1},...,a_{i_1},a_{i_2},...,a_{i_2},...,a_{i_k},...,a_{i_k})]$. 
%To calculate $C(x)$, we have to calculate the centralizer $G_{m \wr d}(x)$ first, 
The centralizer of $x$ in $G_{m \wr d}$ is given by
\eq
\begin{split}\label{centralizerx}
C_{G_{m \wr d}}(x)&= \{ (g_i)_iw \in G_{m \wr d} \ | \ (g_ix_{w^{-1}(i)}g^{-1}_i)_i = (x_i)_i,  \ w\in \Sigma_d\}
\\
&=\{ (g_i)_iw \in G_{m \wr d} \ | \  (g_ix_{w^{-1}(i)}g^{-1}_i)_i=(x_i)_i, \ w\in \Sigma_\gamma \},
\end{split}
\endeq
%Note that if $n_i$ and $n_j$ are of different nilpotent types, then there is no $g \in \GL_m$ such that $gn_ig^{-1}=n_j$. 
%On the other hand, 
since \red{$x_i$ is conjugate to $x_j$ if and only if $\lambda_i = \lambda_j$}.
%Thus, those $w \in \Sigma_d$ that appear in $(g_in_{w^{-1}(i)}g^{-1}_i)_i=(n_i)_i$ can only permute those $n_i$ with the same nilpotent type. Since there are $\gamma(\nu_i)$ components of $x$ of the same nilpotent type $\nu_i$ for all $1 \leq i \leq k$,  as a conclusion, we obtain
%\begin{equation} 
%	G_{m \wr d}(x)=
%\end{equation} 
Moreover, the identity of $G_{m \wr d}$  lies in $C_{G_{m \wr d}}(x)$. 
%Let $G_{m \wr d}(x)^0$ be the connected compoent of $G_{m \wr d}(x)$ that contains the identity. 
%We show that $G_{m \wr d}(x)/G_{m \wr d}(x)^0 \cong \Sigma_{\gamma(x)}$.

%Without lose of generality, we assume $x$ is of the following form
%\begin{equation} \label{formofx}
%x = (\underset{\mrm{type} \ \nu_1}{\underbrace{n_1, \cdots, n_{\gamma(\nu_1)}}}, \underset{\mrm{type} \ \nu_2}{\underbrace{n_{\gamma(\nu_1)+1}, \cdots, n_{\gamma(\nu_1)+\gamma(\nu_2)}}},\cdots,\underset{\mrm{type} \ \nu_k}{\underbrace{n_{\gamma(\nu_1)+...+\gamma(\nu_{k-1})+1}, \cdots, n_{\gamma(\nu_1)+...+\gamma(\nu_k)}}}).
%\end{equation}
For now, we abbreviate the identity component $C_{G_{m \wr d}}(x)^0$ by $C^0$.
We are going to show that the component group $C(x) := C_{G_{m \wr d}}(x)/C^0$ is isomorphic to $\Sigma_{\gamma(x)} := \prod_{\nu\vdash m} \Sigma_{\gamma(x)(\nu)}$. 
\red{We know that $x=(x_1,...,x_d)$ has Jordan type $(\lambda_1,...,\lambda_d)$. Under the $G_{m \wr d}$-conjugation action, we obtain the equivalence class   $[(\lambda_1,...,\lambda_d)]$ under the $\Sigma_d$-action. Thus, without loss of generality, we can rearrange the partitions such that that
$x_{1}$, $\dots$, $x_{\gamma(x)(\lambda_1)}$ are of the same type $\lambda_1 \vdash m$, and $x_{\gamma(x)(\lambda_1)+1}$, $\dots$, $x_{\gamma(x)(\lambda_1)+\gamma(x)(\lambda_2)}$ are of the same type $\lambda_2 \vdash m$, and so on. We write $a:= \gamma(x)(\lambda_1)$ for short.}

Next, we want to construct corresponding elements in $C(x)$ that generate a copy of $\Sigma_a$.
Consider elements of the form $(g_1, \cdots, g_{a}, 1, \cdots, 1) t_k \in C_{G_{m \wr d}}(x)$  for some $g_i \in \GL_m$ and  $t_k \in \Sigma_{a}$. 
For all $1 \leq j \leq a-1$, let $r_j \in \GL_m$ be such that $r_j x_{j+1}r_j\inv = x_{j}$.
%We denote $s_i$ to be the simple reflections in $\Sigma_d$. 
Thus, each equality $g_i x_{t_{k}^{-1}(i)}g^{-1}_i=x_i$ appearing in \eqref{centralizerx} 
becomes one of the following:
$g_{k} x_{k+1}g^{-1}_{k} = r_k x_{k+1}r_k^{-1}$,
$g_{k+1}x_{k}g^{-1}_{k+1} = r_k^{-1}x_{k}r_k$, or
$g_i x_i g_i\inv = x_i$ if $i \neq k,k+1$. 
Equivalently, %$r_i^{-1}g_{i}n_{i+1}g^{-1}_{i}r_i=n_{i+1}$, and $r_i g_{i+1}n_{i}g^{-1}_{i+1}r_i^{-1}=n_{i}$. %The above equations implies
there exists $g'_i \in C_{\GL_m}(x_i)$ for all $1\leq i \leq a $ such that 
\eq
\red{g'_{k} = r_kg_{k+1}}, \quad
\red{g'_{k+1} = r^{-1}_k g_k}, \quad
\red{g'_i = g_i} \quad \tif i\neq k,k+1.
\endeq 
Let $\sigma_k := (g_1, \cdots, g_{a}, 1^{d-a}) t_k C^0  \in C(x)$ for $1\leq k \leq \gamma(x)\red{(\lambda_1)}-1$.
Then, 
\eq\label{generators}
\begin{split}
	\sigma_k
%	&=(g_1,...,g_{i},g_{\gamma(\nu_1)+...+\gamma(\nu_{l-1})+i+1},...,g_d)s_{\gamma(\nu_1)+...+\gamma(\nu_{l-1})+i}G_{m \wr d}(x)^0 \\
	&=(g'_1,\dots, g'_{k-1}, r_k g'_{k+1}, r_k^{-1}g'_{k}, g'_{k+2},\dots,g'_{a}, 1^{d-a}) t_{k}C^0 \\
	&=(1^{k-1}, r_k, r_k^{-1},1^{d-k-1})  t_{k} (g'_1, \dots,  g'_{a}, 1^{d-a})C^0 
	=(1^{k-1}, r_k, r_k^{-1}, 1^{d-k-1})  t_{k} C^0,
\end{split}
\endeq 
%since $g'_{i} \in \GL_{m} n_{i}$, $g'_{i+1} \in \GL_{m} n_{i+1}$, and $g_j \in \GL_m n_j$ for all $j \neq i, i+1$. 
%Therefore, the coset representatives of $G_{m \wr d}(x)/G^0_{m \wr d}(x)$ are given by
%\eq \label{generators}
%(r_i, r_i^{-1}, 1, \dots, 1) t_{i}G_{m \wr d}(x)^0.
%\endeq 
and thus $\sigma_k$ is of order 2. Next, we verify the braid relations for $\sigma_1, \dots, \sigma_{a}$.
We may assume that $d=4= \gamma(x)(\nu_1)$. 
For adjacent braids, consider $\sigma_1 := (r_1, r^{-1}_1,1,1) t_1 C^0$, $\sigma_2 := (1, r_2, r^{-1}_2,1) t_2 C^0$. 
The braid relation follows from the fact that 
\eq
\begin{split}
\sigma_1 \sigma_2 \sigma_1&=(r_1r_2,r^{-1}_1,r^{-1}_2,1)t_1t_2(r_1,r^{-1}_1,1,1)t_1 C^0=(r_1r_2,1,r_2^{-1}r^{-1}_1,1)t_1t_2t_1C^0, 
\\
\sigma_2 \sigma_1 \sigma_2&=(r_1,r_2,r_2^{-1}r_1^{-1},1)t_2t_1(1,r_2,r^{-1}_2,1)t_2C^0=(r_1r_2,1,r_2^{-1}r_1^{-1},1) t_2t_1t_2C^0.
\end{split}
\endeq
For non-adjacent braids (if any), consider $\sigma_3 := (1,1,r_3,r_3\inv) t_3 C^0$, and then
\eq
\sigma_1 \sigma_3=(r_1,r^{-1}_1,r_3, r_3\inv) t_1t_3C^0=
\sigma_3 \sigma_1.
\endeq
\red{
We have verified that $\sigma_1,...,\sigma_{a-1}$ satisfy the relations in the symmetric group $\Sigma_a$. Moreover, there are no other relations between them. Otherwise there will be relations between $t_1,...,t_{a-1}$, which gives additional relations in the group $\Sigma_a$. Therefore, $C(x) \cong  \Sigma_{a} \times C_{G_{m \wr (d-a)}}(x_{a}+1, \dots,  x_d)$.}
It then follows from an induction on $d$ that $C(x) \cong \Sigma_{\gamma(x)}$, and we are done.
\end{proof}

Recall $C(x)^\wedge$ from \eqref{def:Cxwedge}.
Next, we show that $C(x)^\wedge$ coincides with all irreducibles over $\CC[\Sigma_{\gamma(x)}]$.

%--------------------------------------------------
\begin{prop} \label{prop:C(x)id}
Let $x \in \cN_m^d$ and let $\gamma = \gamma(x)$. Then, 
\[
C(x)^\wedge = \redtwo{\widehat{\Sigma_{\gamma}}} = \{S^\bblambda ~|~ \bblambda \in I^C_{m \wr d},\ |\bblambda| = \gamma \}.
\]
Moreover, if $\psi = S^\bblambda \in C(x)^\wedge$, then $H(\fB_x)_\psi = L^\bblambda$
\end{prop}
%--------------------------------------------------
\begin{proof}
It suffices to show that $C(x)^\wedge \supseteq \redtwo{\widehat{\Sigma_{\gamma}}} $. That is, given $x \in \cN_m^d$ and $\bblambda \in I^C_{m \wr d}$ such that $|\bblambda|=\gamma$, the Specht module $S^\bblambda$ occurs in 
$\CC \otimes_\QQ H(\fB_x)$.
Since $C(x)$ is a product of Weyl groups of type A, it suffices to show that $S^\bblambda$ occurs in $\HBMt(\fB_x; \CC)$.

It follows from \eqref{wreathSpringerfiber} and the identification $\cF l_{m \wr d} \equiv \cF l_{m} \wr \Sigma_d$ in Proposition \ref{compatabilitywrfl} (a) that 
\begin{equation} \label{BMwreathSpringer}
\begin{split}
	\HBMt(\fB_x; \CC) &\cong \CC[\Sigma_d] \otimes \HBMt(\fB_{x_1}; \CC) \otimes \dots \otimes \HBMt(\fB_{x_d}; \CC)
\\
&\cong \bigoplus_{w \in \Sigma_d} w ( \bigotimes_{\nu \vdash m} (S_m^\nu)^{\otimes \gamma(\nu)} ) 
= \bigoplus_{w \in \Sigma_d}  w S^{\gamma},
\end{split}
\end{equation} 
on which $\Sigma_d$ acts  by tensor factor permutations (since $\Sigma_{\nu \vdash m} \gamma(\nu)=d$).
It in turn gives the explicit $\CC[\Sigma_m \wr \Sigma_d]$-module structure on $\HBMt(\fB_x;\CC)$. 

%Note that if $w \in \Sigma_d$ is in the Young subgroup $\Sigma_{\gamma}$, then the action of w is trivial on $S^{\gamma}$. So 
%\begin{equation*}
%H^{BM}_{top}(\fB_x,\CC) \cong   (S^{\gamma})^{\oplus |\Sigma_{\gamma}|} \oplus \bigoplus_{w \in \Sigma_d/\Sigma_{\gamma}} w S^{\gamma}
%\end{equation*}	

%Let $\bblambda \in I^C_{m \wr d}$ be  a multi-partition such that $|\lambda|=\gamma$.
Recall from \eqref{irreducibleSpecht} the Specht module $S^{\bblambda}$ over $\CC[\Sigma_{\gamma}]$. 
Via the inclusion $\Sigma_{\gamma} \subseteq \Sigma_d \subseteq  \Sigma_m \wr \Sigma_d$, $\HBMt(\fB_x ;\CC)$ affords a $\CC[\Sigma_{\gamma}]$-module structure. 
Moreover, from \eqref{BMwreathSpringer}, $\HBMt(\fB_x ;\CC)$ contains the subspace $\bigoplus_{w \in \Sigma_{\gamma}} w S^{\gamma}$, which is a regular representation of $\Sigma_{\gamma}$.
Since the regular representation contains all irreducible representations, we are done.

%Since $|\lambda|=\gamma$, for each $\nu \vdash m$ we have $\lambda(\nu) \vdash |\lambda(\nu)|=|\lambda|(\nu)=\gamma(\nu)$.
\end{proof}
%--------------------------------------------------
As a consequence, we can now identify the irreducibles arising from Clifford theory with those arising from Springer theory.
Recall the sets $I^S_{m\wr d}$, $I^C_{m\wr d}$, and the module $S^\bblambda$ from \eqref{def:IS}, \eqref{def:IC}, and \eqref{irreducibleSpecht}, respectively. 

\cor  \label{identindexset}
There is a natural bijection $\Psi:I^C_{m\wr d} \to I^S_{m\wr d}$ between the index sets such that $H(\fB_{x})_\psi \cong L^{\Psi\inv([x,\psi])}$.
The maps $\Psi$ and $\Psi\inv$ are given, respectively, by
\eq
\Psi(\bblambda) = [x_\bblambda, S^\bblambda],
\quad
\Psi\inv([x,\psi]) = \bblambda_\psi,
\endeq
where $x_\bblambda \in \cN_m^d$ is an element of Jordan type $(\lambda_i)_i \in \Pi_m^d$ such that $|\bblambda(\nu)| = \gamma(x_\bblambda) := \#\{1\leq i \leq d ~|~ \lambda_i = \nu\}$ for  $\nu \vdash m$;
while $\bblambda_\psi$ is the element in $I^C_{m \wr d}$ such that $\psi \cong S^{\bblambda_\psi} := \bigotimes_{\nu \vdash m} S^{\bblambda_\psi(\nu)} \in \Irr\CC[\Sigma_{\gamma(x)}]$.
%The map $\Psi:I^C_{m\wr d} \to I^S_{m\wr d}$ is a bijection, where $x_\lambda \in \cN_m^d$ is an element of Jordan type $(x_i)_i \in \Pi_m^d$ such that $|\lambda(\nu)| = \gamma(x_\lambda) := \#\{1\leq i \leq d ~|~ x_i = \nu\}$ for all $\nu \vdash m$.
%
%Conversely, the inverse map $\Psi\inv$ is given by $[x,\psi] \mapsto \lambda_\psi$, 
%where $\lambda_\psi$ is the element in $I^C_{m \wr d}$ such that $\psi \cong S^{\lambda_\psi} := \bigotimes_{\nu \vdash m} S^{\lambda_\psi(\nu)} \in \Irr\CC[\Sigma_{\gamma(x)}]$.
\endcor 
\proof
It follows from the construction that $\Psi$ is well-defined. On the other hand, any $\psi \in C(x)^\wedge$ is isomorphic to $S^{\bblambda_\psi}$ thanks to Proposition~\ref{prop:C(x)id}.
It is routine to check that $\Psi$ and $\Psi\inv$ are inverse to each other.
\endproof

\subsection{Extreme Cases}\label{sec:EC}
\subsubsection{Type B and Type C}
First, we consider an extreme case that $m=2$ so that the wreath product $\Sigma_2 \wr \Sigma_d$ is a Weyl group of type $B_d = C_d$.
Our theorem (that uses a essentially type A geometry) gives a new Springer correspondence for Weyl groups of type B/C.

\exa[new Springer correspondence for type B/C]
Let $\Pi_d^{(2)} := \{ (\nu', \nu'') ~|~ \nu' \vdash a, \nu'' \vdash d-a\}$ be the set of bipartitions of $d$.
The index set $I^C_{2 \wr d} = \{ \bblambda : \{\iyng{2}, \iyng{1,1}\}\to \Pi ~|~ |\bblambda(\iyng{2})|+ |\bblambda(\iyng{1,1})| = d\}$ is in bijection with $\Pi_d^{(2)}$ via $\bblambda \mapsto (\lambda^{(1)}, \lambda^{(2)})$, where $\lambda^{(1)} = \bblambda(\iyng{2})$, $\lambda^{(2)} = \bblambda(\iyng{1,1})$.

For the geometric index set $I^S_{2 \wr d}$, note that any $x \in \cN_2^d$ has Jordan type $(2)^a (1,1)^{d-a}$ for some $a$.
The component group $C_{G_{2 \wr d}}(x)$ is then isomorphic to the Young subgroup  $\Sigma_a \times \Sigma_{d-a}$, and hence any $\psi \in C(x)^\wedge$ is of the form $S^{\lambda'}\otimes S^{\lambda''}$ for some $\lambda' \vdash a$, $\lambda'' \vdash d-a$.
Therefore, we obtain the following bijection:
\eq
\Pi_d^{(2)} \equiv I^C_{2 \wr d} \to I^S_{2 \wr d},
\quad
\bblambda \mapsto [x, S^{\bblambda(\iyng{2})}\otimes S^{\bblambda(\iyng{1,1})}],
\endeq
as well as a new Springer correspondence in terms of type A geometry:
\eq
\widehat{W_{B_d}} \equiv \widehat{W_{C_d}} 
= \{ L^\bblambda ~|~ \bblambda \in I^C_{2 \wr d}\}
= \{ H(\fB_{x})_{\psi} ~|~ [x,\psi] \in I^S_{2\wr d}\}.
\endeq
\endexa

\subsubsection{Type Hu and Type D}
Next, we consider the other extreme case when $d=2$. 
The wreath product $\Sigma_m \wr \Sigma_2$ is generally not a Coxeter group.
We remark that irreducibles for the Hu algebras share the same index set 
\eq\label{def:IH}
I^H_{m \wr 2}:= \{[\nu', \nu''] ~|~ \nu' \neq \nu'' \in \Pi_m\} \sqcup \{[\nu,\nu]_+, [\nu,\nu]_- ~|~ \nu \in \Pi_m\}
\endeq 
with the irreducibles for $\CC[\Sigma_m \wr \Sigma_2]$
(see \cite{Hu02, LNX24}), where $[\nu', \nu'']$ denotes the equivalence class in $\Pi_m^2$ under the obvious $\Sigma_2$-action.
Note that each equivalence class of the form $[\nu, \nu]$ corresponds to two simple modules, and hence one can index these modules (called $D(\nu)_+$ and $D(\nu)_-$ in \cite[\S4]{Hu02}) by $[\nu, \nu]_+$ and $[\nu,\nu]_-$, respectively.

\exa[The wreath product $\Sigma_m \wr \Sigma_2$] \label{ex:Hu}
Recall that $I^C_{m \wr 2} = \{ \bblambda : \Pi_m \to \Pi ~|~ \sum_{\nu\vdash m} |\bblambda(\nu)| = 2\}$.
We have a bijection $I^C_{m \wr 2} \to I^H_{m \wr 2}$ given by
\eq
\bblambda \mapsto \begin{cases}
[\nu, \nu]_+ &\tif \bblambda(\nu) = \iyng{2};
\\
[\nu, \nu]_- &\tif \bblambda(\nu) = \iyng{1,1};
\\
[\nu', \nu''] &\tif \bblambda(\nu') = \bblambda(\nu'') = \iyng{1}.
\end{cases}
\endeq
On the other hand,  any $x \in \cN_m^2$ either has Jordan type  $\nu^2$ or $\nu'\nu''$ for some $\nu \in \Pi_m$ or some distinct partitions $\nu' \neq \nu'' \in \Pi_m$.
For the former case, $C(x) \cong \Sigma_2$ and hence $x$ corresponds to two distinct elements $[x, S^{\iyng{2}}]$ and $[x, S^{\iyng{1,1}}]$ in $I^S_{m \wr 2}$.
Note that the plus and minus signs in \eqref{def:IH} now correspond to a choice of the trivial module $S^{\iyng{2}}$ or the sign module $S^{\iyng{1,1}}$ over $\CC[\Sigma_2]$.
For the latter case, $C(x) \cong \Sigma_1\times \Sigma_1$ and hence $x$ corresponds to exactly one element $[x, S^{\iyng{1}}]$ in $I^S_{m \wr 2}$.

Finally, $\Psi$ sends $\bblambda$ to $[x, S^{\bblambda(\nu)}]$ if $|\bblambda(\nu)| =2$ for some $\nu\vdash m$ and some $x\in \cN_m^2$ with Jordan type $\nu^2$;
 and it sends $\bblambda$ to $[x, S^{\iyng{1}}]$ if $\bblambda(\nu') = \bblambda(\nu'') = (1)$ for some  $\nu' \neq \nu'' \vdash m$ and some $x\in \cN_m^2$ with Jordan type $\nu' \nu''$.
\endexa
\rmk[new Springer correspondence for type D]
Combining Example~\ref{ex:Hu} with the Morita equivalence result, the index set of the irreducibles for the Weyl group (and hence the Hecke algebra) of type $D$ is in bijection with the following set:
\eq
I^H_{D_{d}}:=
\begin{cases}
 \{[\nu', \nu''] ~|~ (\nu', \nu'') \in \Pi_{d}^{(2)}, \ \nu' \neq \nu''\} 
&\tif d=2m+1;
\\
 \{[\nu', \nu''] ~|~ (\nu', \nu'') \in \Pi_{d}^{(2)}, \ \nu' \neq \nu''\} 
\sqcup
\{[\nu, \nu]_+, [\nu, \nu]_- ~|~ (\nu, \nu) \in \Pi_{d}^{(2)}\} &\tif d=2m.
\end{cases}
\endeq
Let $\cN_{d}^{(2)} := \bigsqcup_{1\leq a \leq d} \cN_a \times \cN_{d-a}$.
For $x \in \cN_{d}^{(2)}$, let $(\nu'_x, \nu''_x) \in \Pi_{d}^{(2)}$ be the corresponding Jordan type. 
Define
\eq
I^S_{D_d}:= 
\begin{cases}
\{[x, S^{\iyng{1}}] ~|~ x\in  \cN_{d}^{(2)}, \nu'_x \neq \nu''_x \} 
&\tif d=2m+1;
\\
\{[x, S^{\iyng{1}}] ~|~ x\in  \cN_{d}^{(2)}, \nu'_x \neq \nu''_x \} 
\sqcup 
\{[x, S^{\iyng{2}}], [x, S^{\iyng{1,1}}] ~|~ x\in  \cN_{d}^{(2)}, \nu'_x = \nu''_x \} &\tif d=2m.
\end{cases}
\endeq
Therefore, we obtain a bijection $\widehat{W_{D_d}} \equiv I^H_{D_d} \to I^S_{D_d}$ given by
\eq
[\nu', \nu''] \mapsto [x, S^{\iyng{1}}],
\quad
[\nu', \nu'']_{\epsilon} \mapsto 
\begin{cases}
{}[x, S^{\iyng{2}}] &\tif \nu' = \nu'',\ \epsilon = +;
\\
{}[x, S^{\iyng{1,1}}] &\tif  \nu' = \nu'',\ \epsilon = -,
\end{cases}
\endeq
where $x \in \cN^{(2)}_d$ is some element of Jordan type $\nu'\nu''$ with $(\nu', \nu'') \in \Pi^{(2)}_d$.
In other words, we obtain a new Springer correspondence for type D using geometry of type A:
\eq\label{eq:SpringerD}
\widehat{W_{D_d}} 
= \{ H(\fB_{x})_{\psi} ~|~ [x,\psi] \in I^S_{D_d}\},
\endeq
where the isotypic components are given by the wreath Springer fiber construction if $\psi \not\eq S^{\iyng{1}}$. 
Otherwise, they are tensor products $H(\fB_{x})_{\psi} \cong H(\fB_{x_1}) \otimes H(\fB_{x_2})$ of the type A isotypic components.

We remark that the correspondence \eqref{eq:SpringerD} differs from the original Springer correspondence, which can be described via the Lusztig-Shoji algorithm \cite{Lu79, Sh79, Sh83} in terms of type D geometry.
See also \cite{SW19} for a diagrammatic approach to Springer correspondence for type D, for two-row partitions.
\endrmk

\end{document}